\documentclass{amsart}

\newtheorem{theorem}{Theorem}[section]
\newtheorem{lemma}[theorem]{Lemma}
\newtheorem{proposition}[theorem]{Proposition}
\newtheorem{corollary}[theorem]{Corollary}

\theoremstyle{definition}

\newtheorem{remark}[theorem]{Remark}

\newtheorem*{problem}{Problem}

\numberwithin{equation}{section}

\newcommand{\R}{\mathbf{R}}

\newcommand{\cB}{\mathcal{B}}
\newcommand{\cD}{\mathcal{D}}
\newcommand{\cI}{\mathcal{I}}
\newcommand{\cJ}{\mathcal{J}}
\newcommand{\cK}{\mathcal{K}}
\newcommand{\cE}{\mathcal{E}}
\newcommand{\cR}{\mathcal{R}}

\newcommand{\eps}{\varepsilon}
\newcommand{\loc}{\textrm{loc}}
\newcommand{\weak}{\textrm{weak}}
\DeclareMathOperator{\supp}{\operatorname{supp}}
\DeclareMathOperator{\card}{\operatorname{card}}

\begin{document}

\title[Blow-up of the critical norm]{Blow-up of the 
critical norm for a supercritical semilinear heat equation}

\author[H. Miura]{Hideyuki Miura}
\address{Department of Mathematics, 
Tokyo Institute of Technology, Tokyo 152-8551, Japan}
\email{hideyuki@math.titech.ac.jp}

\author[J. Takahashi]{Jin Takahashi}
\address{Department of Mathematical and Computing Science, 
Tokyo Institute of Technology, Tokyo 152-8552, Japan}
\email[Corresponding author]{takahashi@c.titech.ac.jp}



\begin{abstract}
We consider the scaling critical Lebesgue norm of blow-up solutions 
to the semilinear heat equation  
$u_t=\Delta u+|u|^{p-1}u$ in an arbitrary
$C^{2+\alpha}$ domain of $\R^n$. 
In the range $p>p_S:=(n+2)/(n-2)$, 
we show that the critical norm must be unbounded near the blow-up time, 
where the type I blow-up condition is not imposed. 
The range $p>p_S$ is optimal in view of 
the existence of type II blow-up solutions 
with bounded critical norm for $p=p_S$. 
\end{abstract}

\maketitle

\tableofcontents

%%%%%%%%%%%%%%%%%%%%%%%%%%%%%%%%%%%%%%%%%%%%%%%%
%%%%%%%%%%%%%%%%%%%%%%%%%%%%%%%%%%%%%%%%%%%%%%%%
%%%%%%%%%%%%%%%%%%%%%%%%%%%%%%%%%%%%%%%%%%%%%%%%
\section{Introduction}
%%%%%%%%%%%%%%%%%%%%%%%%%%%%%%%%%%%%%%%%%%%%%%%%
%%%%%%%%%%%%%%%%%%%%%%%%%%%%%%%%%%%%%%%%%%%%%%%%
%%%%%%%%%%%%%%%%%%%%%%%%%%%%%%%%%%%%%%%%%%%%%%%%
\subsection{Background}
%%%%%%%%%%%%%%%%%%%%%%%%%%%%%%%%%%%%%%%%%%%%%%%%
%%%%%%%%%%%%%%%%%%%%%%%%%%%%%%%%%%%%%%%%%%%%%%%%
%%%%%%%%%%%%%%%%%%%%%%%%%%%%%%%%%%%%%%%%%%%%%%%%
We study blow-up solutions of the following semilinear heat equation: 
\begin{equation}\label{eq:main}
\left\{ 
\begin{aligned}
	&u_t=\Delta u+|u|^{p-1}u, &&x\in \Omega, \  t>0, \\
	&u(x,t)=0, &&x\in \partial \Omega, \  t>0, \\
	&u(x,0)=u_0(x), &&x\in \Omega. 
\end{aligned}
\right.
\end{equation}
Here $p>1$, $\Omega$ is a domain in $\R^n$ with $n\geq1$
and $u_0\in L^q(\Omega)$ with $q\geq1$.  
The boundary condition is not present if $\Omega=\R^n$.

The equation in \eqref{eq:main} has attracted much attention 
as one of the simplest model 
for scaling invariant nonlinear parabolic equations. 
For each solution $u$, the rescaled function 
$u_\lambda(x,t):=\lambda^{2/(p-1)}u(\lambda x,\lambda^2 t)$ 
($\lambda>0$) 
also satisfies the equation. 
This implies that $L^{q_c}(\Omega)$ with 
\[
	q_c:=\frac{n(p-1)}{2} 
\]
is the scaling critical Lebesgue space for \eqref{eq:main}. 
The critical space plays a crucial role in well-posedness. 
If $u_0\in L^{q_c}(\Omega)$ and $q_c>1$, 
it is well-known \cite{BC96,We79,We80} that 
there exists a unique classical $L^{q_c}$-solution $u$ 
of \eqref{eq:main} with the maximal existence time $T\in(0,\infty]$. 
For the definition of the classical $L^q$-solution, 
see Remark \ref{rem:defsol}. 
The solution $u$ is smooth for $t\in(0,T)$, belongs to 
$C([0,T); L^{q_c}(\Omega))\cap C( (0,T); L^\infty(\Omega))$ 
and admits the following blow-up criterion: 
If $T<\infty$, then 
$\lim_{t\to T}\|u(\cdot,t)\|_{L^\infty(\Omega)}=\infty$. 
This leads to the critical norm blow-up problem. 

\begin{problem}
If $T<\infty$, does the following property hold?
\begin{equation}\label{eq:CNB}\tag{CNB}
	\lim_{t\to T}\|u(\cdot,t)\|_{L^{q_c}(\Omega)}=\infty. 
\end{equation}
\end{problem}

The problem is stated in Brezis and Cazenave \cite[Open problem 7]{BC96} 
and also a variant can be found in
Quittner and Souplet \cite[OP 2.1, Section 55]{QSbook2} 
which asks the existence of blow-up solutions with bounded $L^{q_c}$ norm. 
Many sufficient conditions for \eqref{eq:CNB} are known, see 
\cite{Brun,FHV00,GK87,MM04,MM09,Ma98,MS19,We86} and 
\cite[Section 16]{QSbook2}. 
Recently, a significant progress was made by 
Mizoguchi and Souplet \cite{MS19}, where they proved 
that \eqref{eq:CNB} holds whenever the blow-up is type I. 
Here the blow-up of $u$ is called type I 
if the blow-up rate is bounded by 
a spatially homogeneous solution up to the coefficient, that is, 
$\limsup_{t\to T} (T-t)^{1/(p-1)}\|u(\cdot,t)\|_{L^\infty(\Omega)}<\infty$. 
This rate is natural in view of the scaling. 
Note that the blow-up is called type II if it is not of type I. 
In the celebrated work of Giga and Kohn \cite{GK87},
they showed that 
the blow-up is always type I 
if $\Omega$ is convex and 
either $p<p_S$ and $u$ is nonnegative, 
or $p<(3n+8)/(3n-8)$.  
Here $p_S$ is the Sobolev critical exponent given by 
\[
	p_S:=
	\left\{ 
	\begin{aligned}
	&\frac{n+2}{n-2} &&\mbox{ for }n\geq3, \\
	&\infty &&\mbox{ for }n=1,2.    \\
	\end{aligned}
	\right.
\]
Up to now, it has been proved that the blow-up is type I 
under the conditions that either $p<p_S$ and $\Omega$ is convex \cite{GMS04,GMS04s}, 
or $p<p_S$ and $u$ is nonnegative \cite{Qu21}. 
Therefore, under such conditions, (CNB)  holds for any blow-up solutions
in the subcritical range $p<p_S$.
For related results concerning sufficient conditions 
for type I blow-up, we refer 
\cite{CMR17,CMR17c,CRS19,FM85,GMS04,GMS04s,IM03,MM04,MM11,MRS20,Mi11,
PQS07,Qu16,Qu21,We85} and \cite[Section 23]{QSbook2}.

The situation is different in the case $p \geq  p_S$.
Type II blow-up solutions were constructed 
in a number of papers, 
see Remark \ref{rem:type2} for a brief review.  
In particular, the recent development 
\cite{dPMW19,dPMWZZpre,dPMWZ20,Ha20,LSWpre,Sc12} 
provides type II blow-up solutions satisfying 
\[
	\sup_{0<t<T}\|u(\cdot,t)\|_{L^{q_c}(\Omega)}<\infty
\]
for $p=p_S$ and $3 \le n \le 5$, 
see \cite[Section 4]{MS19} 
for computations of the $L^{q_c}$ norm 
of the solutions constructed in \cite{dPMW19,Sc12}. 
This demonstrates that 
type I assumption in \cite{MS19} is indeed necessary 
if $p=p_S$ and $3\le n\le 5$. 
For $p=p_S$ and $n\geq 6$, no counter-examples are known.
Moreover, for $n\geq7$, it was shown \cite{WW21pre} that
\eqref{eq:CNB} holds for interior blow-up solutions 
with $u \ge 0$ and either $\Omega=\R^n$ or $\Omega$ bounded.
In contrast with the case $p=p_S$, 
all the known type II blow-up solutions  
do satisfy \eqref{eq:CNB} in the range $p>p_S$. 
Taking the above results into account, 
we expect that 
\eqref{eq:CNB} holds whenever $p>p_S$. 

%%%%%%%%%%%%%%%%%%%%%%%%%%%%%%%%%%%%%%%%%%%%%%%%
%%%%%%%%%%%%%%%%%%%%%%%%%%%%%%%%%%%%%%%%%%%%%%%%
%%%%%%%%%%%%%%%%%%%%%%%%%%%%%%%%%%%%%%%%%%%%%%%%
\subsection{Main theorem}
%%%%%%%%%%%%%%%%%%%%%%%%%%%%%%%%%%%%%%%%%%%%%%%%
%%%%%%%%%%%%%%%%%%%%%%%%%%%%%%%%%%%%%%%%%%%%%%%%
%%%%%%%%%%%%%%%%%%%%%%%%%%%%%%%%%%%%%%%%%%%%%%%%
Our main result shows that, 
in the optimal range $p>p_S$, 
the critical norm of all finite time blow-up solutions  
must be unbounded without assuming 
nonnegativity, monotonicity, symmetry, convexity or the type of blow-up.

\begin{theorem}\label{th:main}
Let $n\geq3$, $p>p_S$, $\Omega$ be any 
$C^{2+\alpha}$ domain in $\R^n$ with $0<\alpha<1$
and $u$ be a classical $L^{q_c}$-solution of \eqref{eq:main} 
with $u_0\in L^{q_c}(\Omega)$. 
If the maximal existence time $T>0$ is finite, then 
\[
	\limsup_{t\to T}\|u(\cdot,t)\|_{L^{q_c}(\Omega)}=\infty. 
\]
\end{theorem}

The theorem immediately shows the nonexistence 
of blow-up solutions with bounded $L^{q_c}$ norm, and so 
this resolves the open problem \cite[OP 2.1, Section 55]{QSbook2} for the supercritical case. 
We now give comments on the proof, and then 
we list remarks concerning the statement of this theorem 
and related results including other scaling invariant nonlinear evolution equations. 

As in \cite{ESS03,  Wa08} for related equations,
our proof of Theorem \ref{th:main} consists of two parts: 
(i) the blow-up (rescaling and compactness) procedure and 
(ii) the analysis of the blow-up limit.
However, several additional difficulties appear 
from the differences of the nonlinear structure,
in  particular, the lack of the coercivity of the energy
and the absence of the derivative in the nonlinear term.
Compared with the earlier work \cite{MS19},
there are also some novelties in the proof.
In (i), a concentration theorem of the $L^{q_c}$ norm
near a blow-up point plays a crucial role 
for the nondegeneracy of the blow-up limit in \cite{MS19}.
Unfortunately, we could not rely on their concentration theorem due to the absence of the type I assumption. 
In order to circumvent this difficulty, we will prove 
a new $\eps$-regularity theorem (Theorem \ref{th:epsreg}), 
which guarantees the energy concentration near a 
blow-up point.
Theorem \ref{th:epsreg} is motivated by similar 
$\eps$-regularity theorems in \cite{CDZ07}.
In comparison, we do not need to 
assume that the solution is globally defined in time 
or has a certain lower bound of the energy. 
In (ii), unlike the case of \cite{MS19}, 
the smoothness of our blow-up limit 
is no longer clear even before the final time. 
To overcome this issue, we invoke 
a monotonicity estimate similar to \cite{GK87}, 
which plays a key role to identify the blow-up limit. 
We note that the uses of the $\eps$-regularity theorem
 and the monotonicity estimate are partially 
 inspired by the related works \cite{St88, Wa08} 
for the harmonic map heat flow, but 
several modifications are needed for adapting
to our problem as explained above.

\begin{remark}[Classical $L^q$-solution]\label{rem:defsol}
For $1\leq q\leq \infty$, 
the definition of the classical $L^q$-solution is as follows 
(see \cite[Definition 15.1]{QSbook2}). 
Let $u_0\in L^q(\Omega)$ and $T\in(0,\infty]$. 
We say that $u\in C([0,T); L^q(\Omega))$ is 
a classical $L^q$-solution of \eqref{eq:main} 
if $u\in C^{2,1}(\Omega\times(0,T))\cap C(\overline{\Omega}\times (0,T))$, 
$u(\cdot,0)=u_0$ and 
$u$ is a classical solution of \eqref{eq:main} for $t\in (0,T)$. 
If $\Omega$ is unbounded, we also impose 
$u\in L^\infty_\loc((0,T);L^\infty(\Omega))$. 
If $q=\infty$, then $u\in C([0,T); L^q(\Omega))$ is replaced with 
$u\in C((0,T); L^\infty(\Omega))$ and 
$\lim_{t\to0}\|u(\cdot,t)-e^{t\Delta} u_0\|_{L^\infty(\Omega)}=0$, 
where $e^{t\Delta}$ is the Dirichlet heat semigroup in $\Omega$. 
We note that the maximal existence time 
can be defined for classical $L^q$-solutions, 
see \cite[Proposition 16.1 (i), (ii)]{QSbook2} 
for the definition. 
\end{remark}

\begin{remark}[Uniqueness]
By \cite{We79,We80} and \cite[Theorem 1]{BC96}, 
a unique classical $L^{q_c}(\Omega)$-solution of \eqref{eq:main} exists 
for each $u_0\in L^{q_c}(\Omega)$ with $q_c>1$. 
Moreover, if we further assume $q_c>p$, 
the unconditional uniqueness \cite[Theorem 4]{BC96} 
holds for mild solutions, that is, 
solutions of the corresponding integral equation 
\[
	u(\cdot,t)=e^{t\Delta}u_0
	+\int_0^t e^{(t-s)\Delta}|u(\cdot,s)|^{p-1}u(\cdot,s) ds
	\quad \mbox{ for }t\in (0,T)
\]
in $C([0,T);L^{q_c}(\Omega))$. 
Note that $q_c>p$ is equivalent to $p>n/(n-2)$ 
and is satisfied for $p>p_S$. 
Hence, under the assumption of Theorem \ref{th:main}, 
the condition $u_0\in L^{q_c}(\Omega)$ 
implies the existence of a unique mild solution 
in $C([0,T);L^{q_c}(\Omega))$. 
This solution is also a classical $L^{q_c}$-solution, since $q_c>1$. 
\end{remark}

\begin{remark}[Other Lebesgue spaces]
We recall the case $u_0\in L^q(\Omega)$ with $q\neq q_c$. 
For $1\leq q<q_c$, there are results of 
the nonexistence and nonuniqueness of solutions, 
see \cite{Ba83,CDNW20,HW82,We80} and \cite[Section 6]{PQS07}. 
By \cite[Theorem 2.4]{FM85}, there exist solutions such that 
the maximal existence time $T$ is finite and 
\[
	\sup_{0<t<T}\|u(\cdot,t)\|_{L^q(\Omega)}<\infty. 
\] 
For $q_c<q\leq \infty $ with $q\geq1$, it is known 
that \eqref{eq:main} has a unique classical $L^q$-solution $u$ 
(see \cite[Theorem 15.2, Proposition 51.40]{QSbook2} for instance). 
By \cite[Corollary 13]{BC96}, it is also known that 
if $T<\infty$, then 
$\lim_{t\to T}\|u(\cdot,t)\|_{L^q(\Omega)}=\infty$. 
More precisely, the lower estimate of the blow-up rate 
\[
	\|u(\cdot,t)\|_{L^q(\Omega)}\geq C(T-t)^{-\frac{n}{2}(\frac{1}{q_c}-\frac{1}{q})} 
\]
holds for some constant $C>0$, see \cite[Section 6]{We81} and 
\cite[Remark 16.2 (iii), Proposition 23.1]{QSbook2}.
\end{remark}

\begin{remark}[Application to backward self-similar solutions]
Let $\Omega=\R^n$. 
We call the solution $u$ of \eqref{eq:main} backward self-similar 
if it is of the form
$u(x,t)=(T-t)^{-1/(p-1)}U(x/\sqrt{T-t})$ 
for some $T>0$ and some profile function $U \in C^2(\R^n)$. 
Theorem \ref{th:main} immediately 
shows the following corollary on the Liouville type theorem 
for backward self-similar solutions $u$ of \eqref{eq:main} with $p>p_S$.
It recovers \cite[Corollary 2]{MS19} in the case $p>p_S$.
\end{remark}

\begin{corollary}
Let $n\geq3$, $p>p_S$ 
and $u$ be a backward self-similar solution of \eqref{eq:main}. 
If the profile function $U$ of $u$ 
belongs to $L^{q_c}(\R^n)$, then $u\equiv0$. 
\end{corollary}

\begin{remark}[Brief review of type II blow-up]\label{rem:type2}
Type II blow-up solutions were first found by 
Herrero and Vel\'azquez \cite{HVun,HV94}
in the Joseph--Lundgren supercritical range 
$p> p_{JL}:=(n-2\sqrt{n-1})/(n-4-2\sqrt{n-1}) \, (>p_S)$ with $n\geq11$. 
For a refined construction, see \cite{Mi04,MS21}. 
See also \cite{MM11,Mi05,Mi07} 
for the Lepin supercritical range 
$p>p_L:=(n-4)/(n-10) \,(>p_{JL})$. 
The critical cases $p=p_{JL}$ and $p=p_L$ were handled in \cite{Se18,Se20}. 
In \cite{NS07}, the case where $p=p_S$, $n=3$ and 
a suitably shrinking $\Omega=\Omega(t)$ was handled. 
Remark that the above type II blow-up solutions 
are radially symmetric. 
In the range $p_S<p<p_{JL}$, 
it was proved \cite{MM04} 
that all radially symmetric blow-up solutions 
are of type I if either $\Omega$ is a ball, 
or $\Omega=\R^n$ with assumptions on intersection properties, 
see also \cite{MM09,Mi11}. 
For the existence of non-radial type II blow-up solutions, 
see \cite{Co17,CMR20} for some $p>p_{JL}$, 
\cite{dPMW21} for $p=p_2:=(n+1)/(n-3)$ and $n\geq7$, 
and \cite{dPLMWZpre} for $p=p_{n-3}:=3$ and $5\leq n\leq 7$. 
We note that $p_2$ and $p_{n-3}$ are 
the so-called second critical exponent 
(after \cite{dPMP10}) and $(n-3)$-th critical exponent (after \cite{dPLMWZpre}). 
They satisfy $p_S<p_2<p_{JL}$ for $n\geq 4$ 
and $p_2\leq p_{n-3}$ for $n\geq 5$, 
where $p_{JL}:=\infty$ for $n\leq 10$. 
One of the reasons why such exponents appear 
is explained in \cite[Subsection 1.4]{CMR20}.

In the Sobolev critical case $p=p_S$, 
type II blow-up solutions were formally found 
by Filippas, Herrero and Vel\'azquez \cite{FHV00} for $3\leq n\leq 6$
and were rigorously constructed in \cite{Sc12} for $n=4$. 
The recent development of 
the inner-outer gluing method refined the construction, 
see \cite{dPMWZZpre} for $n=3$, 
\cite{dPMWZ20,LSWpre} for $n=4$, 
\cite{dPMW19,Ha20} for $n=5$ 
and \cite{Ha20s} for $n=6$. 
On the other hand, 
it was proved \cite{CMR17} for $n\geq7$ that 
there is no type II blow-up solution 
if $u_0$ is close to the Aubin-Talenti function. 
Recently, it was also shown \cite{WW21pre} that 
all interior blow-up solutions are of type I 
provided that $n\geq7$, $u\geq0$ and either $\Omega=\R^n$ 
or $\Omega$ is bounded. 
\end{remark}

\begin{remark}[Ancient solutions]
We say that $u$ is an ancient solution if it satisfies 
$u_t-\Delta u=|u|^{p-1} u$ for $t\in (-\infty,T)$ with some $T<\infty$. 
Classification results for such solutions were obtained in 
\cite{MZ98} for $p<p_S$ and 
\cite{PQ21} for $p_S<p<p_{JL}$ and $p>p_L$, 
see also the references given there. 
In our context, as far as the authors know, 
the following question is open: 
Does there exist a nontrivial solution satisfying 
$\sup_{-\infty<t<T}\|u(\cdot,t)\|_{L^{q_c}(\Omega)}<\infty$ 
for $p>p_S$? 
\end{remark}

\begin{remark}[Infinite time blow-up]
Infinite time blow-up (or grow-up) solutions, that is, global-in-time 
solutions satisfying $\lim_{t\to \infty}\|u(\cdot,t)\|_{L^\infty(\Omega)}=\infty$, 
were constructed for $p\geq p_S$, 
see \cite{GK03} for $p=p_S$, \cite{PY14} for $p_S<p<p_{JL}$
and \cite{PY03} for $p\geq p_{JL}$. 
For $p=p_S$, possible asymptotic behavior was conjectured 
by Fila and King \cite{FK12}. 
Recently, the conjecture was confirmed 
by \cite{dPMW20} for $n=3$, \cite{WZZ22pre} for $n=4$ 
and \cite{LZZWpre} for $n=5$. 
Although this paper focuses on finite time blow-up solutions, 
it may be interesting to study the behavior of 
the critical norm 
for infinite time blow-up solutions. 
\end{remark}

\begin{remark}[Critical norm blow-up for the Navier-Stokes equations]
Theorem \ref{th:main} corresponds to the pioneering work of 
Escauriaza, Seregin and \v{S}ver\'{a}k \cite{ESS03} 
for the three-dimensional Navier-Stokes equations. 
They showed the blow-up of the critical norm in the sense that 
if the maximal existence time $T$ is finite, then 
\begin{equation}\label{eq:qualibc}
	\limsup_{t\to T}\|u(\cdot,t)\|_{L^3} =\infty. 
\end{equation}
The limit superior condition was later improved 
to the limit condition in \cite{Se12}.
In the case of the domains with boundary, the condition 
\eqref{eq:qualibc} was verified for the flat boundary \cite{Se05}
and for general cases \cite{MS06}.  
These results were also refined to the limit condition 
for the flat case \cite{BS17, MMP19} 
and for general cases \cite{AB20}.  
On the other hand, the $L^3$ norm in \eqref{eq:qualibc} 
was further refined to the Lorentz norm \cite{Ph15} and 
the Besov norm \cite{Al18,GKP16}. 
Actually, our norm in Theorem \ref{th:main} can be replaced by the Lorentz norm
$L^{q_c,r}$ with $r<\infty$, but we do not pursue this issue here. 
We also refer \cite{DD09,DW20} for the critical norm blow-up
in higher dimensions $n \ge 4$. 

Recently, Tao \cite{Ta21} proved that if $T<\infty$, then 
\[
	\limsup_{t\to T}\frac{\|u(\cdot,t)\|_{L^3}}
	{(\log\log\log(1/(T-t)))^c} =\infty 
\]
for some constant $c>0$. See \cite{BP21,Pa21} for further developments
in this subject.  
By analogy, it is expected that there is a 
general quantitative blow-up criterion 
for the semilinear heat equation \eqref{eq:main} with $p>p_S$. 
This direction seems interesting and also challenging.
Moreover, it remains an open problem 
whether the limit superior condition in Theorem \ref{th:main} 
can be replaced with the limit condition. 
We note that the result of \cite{MS19} for $p > p_S$ 
under the type I blow-up assumption 
is not a consequence of Theorem \ref{th:main}.
\end{remark}

\begin{remark}[Critical norm blow-up for other equations]
Wang \cite{Wa08} studied the critical norm of the harmonic map heat flow
between compact Riemannian manifolds without boundaries
in the energy supercritical dimension 
$n \ge 3$. 
It was shown that if the maximal existence time $T$ is finite, then 
\[
	\limsup_{t\rightarrow T} \|\nabla u(\cdot,t)\|_{L^n}=\infty.
\]
One of the key ideas in the proof 
is the monotonicity formula of Struwe \cite{St88}.

For nonlinear dispersive equations with power nonlinearities, 
there are also many works on the blow-up of 
the critical norm of the form
\[
	\limsup_{t\to T} \|u(\cdot, t)\|_{\dot{H}^{s_c}} =\infty. 
\]
Here $s_c$ is the scaling critical exponent for each of the equations.
Kenig and Merle \cite{KM10} 
showed that blow-up solutions of 
the cubic defocusing Schr\"{o}dinger equation in $\R^3$ 
must satisfy the above condition with $s_c=1/2$. 
Their method is based on the 
concentration compactness procedure and the rigidity theorem. 
A similar method is applicable to the defocusing supercritical 
nonlinearity, see \cite{KV10}.
In the radial case, Merle and Rapha\"el \cite{MR08} 
gave an explicit 
lower bound of the critical norm in some energy subcritical range.
See also a recent result \cite{Bu20} for the lower bound of the critical norm 
in the radial supercritical case.
Similar results were also obtained 
for the nonlinear wave equation starting from \cite{KM11}.
For the focusing case, see \cite{DKM14,DY18}.
\end{remark}

\begin{remark}[Related results for supercritical elliptic equations]
In the proof of Theorem \ref{th:main}, 
the obtained blow-up limit $\overline{u}$ 
is a weak solution of the semilinear heat equation and 
satisfies a monotonicity estimate. 
In addition, the singular set of $\overline{u}(\cdot,t)$ 
consists of finitely many points for each $t$. 
A similar situation can be found for 
the so-called stationary solutions of 
the semilinear elliptic equation $-\Delta u=|u|^{p-1}u$ 
for $p>p_S$, see \cite{Du14,Pa93,Pa94,WW15}. 
In this context, 
a weak solution $u$ of the elliptic equation 
is called a stationary solution 
if $u$ is a critical point of the corresponding energy functional 
with respect to domain variations. 
\end{remark}

%%%%%%%%%%%%%%%%%%%%%%%%%%%%%%%%
%%%%%%%%%%%%%%%%%%%%%%%%%%%%%%%%
%%%%%%%%%%%%%%%%%%%%%%%%%%%%%%%%
\subsection{Organization of this paper}
%%%%%%%%%%%%%%%%%%%%%%%%%%%%%%%%
%%%%%%%%%%%%%%%%%%%%%%%%%%%%%%%%
%%%%%%%%%%%%%%%%%%%%%%%%%%%%%%%%
This paper is organized as follows. 
In Section \ref{sec:grad}, we derive a key gradient estimate. 
In Section \ref{sec:pre}, we define 
a localized weighted energy and prove its quasi-monotonicity. 
In Section \ref{sec:epreg}, we prove an $\eps$-regularity theorem 
by analyzing the energy. 
In Section \ref{sec:Blim},  
we construct and examine a blow-up limit with the aid of the $\eps$-regularity, 
and then we prove Theorem \ref{th:main}. 
In Appendix \ref{sec:reges}, we give regularity estimates 
used in Section \ref{sec:Blim}. 
In Appendix \ref{sec:cpt}, 
we recall an Aubin-Lions type compactness result 
also used in Section \ref{sec:Blim}. 

%%%%%%%%%%%%%%%%%%%%%%%%%%%%%%%%
%%%%%%%%%%%%%%%%%%%%%%%%%%%%%%%%
%%%%%%%%%%%%%%%%%%%%%%%%%%%%%%%%
\subsection{Notation}
%%%%%%%%%%%%%%%%%%%%%%%%%%%%%%%%
%%%%%%%%%%%%%%%%%%%%%%%%%%%%%%%%
%%%%%%%%%%%%%%%%%%%%%%%%%%%%%%%%
For $x\in\R^n$, we often write $x=(x',x_n)$ 
with $x'\in\R^{n-1}$ and $x_n\in\R$. 
Set $\R^n_+:=\{x\in\R^n; x'\in\R^{n-1}, x_n>0\}$. 
We denote by $\chi_A$ and $|A|$ the characteristic function 
and the Lebesgue measure of a measurable set $A$, respectively. 
For $r>0$ and $(x,t)\in\R^{n+1}$, we write 
\[
\begin{aligned}
	&B_r(x):= \{y\in \R^n; |x-y|<r\}, &&B_r:=B_r(0), \\
	&\Omega_r(x):= \Omega \cap B_r(x), &&\Omega_r:=\Omega_r(0), \\
	&P_r(x,t):= B_r(x)\times (t-r^2,t), &&P_r=P_r(0,0), \\
	&Q_r(x,t):= \Omega_r(x) \times (t-r^2,t), &&Q_r=Q_r(0,0). 
\end{aligned}
\]
For $\rho>0$ and $(x,t)\in\R^{n+1}$, we write 
\[
\begin{aligned}
	&B_\rho^+(x):=\R^n_+ \cap B_\rho(x), &&B_\rho^+:=\R^n_+ \cap B_\rho(0), \\
	&Q_\rho^+(x,t):= B_\rho^+(x)\times (t-\rho^2,t), 
	&&Q_\rho^+:= Q_\rho^+(0,0). 
\end{aligned}
\]
We denote by $G_\Omega=G_\Omega(x,y,t)$ 
the Dirichlet heat kernel in $\Omega$. 
Set 
\[
	K_{(\tilde x,\tilde {t})}(x,t):= 
	(\tilde t-t)^{-\frac{n}{2}} 
	e^{-\frac{|x-\tilde x|^2}{4(\tilde t-t)}}
\]
for $x,\tilde x\in\R^n$ and $t<\tilde t$. 
The critical exponents are defined by
\[
	p_S:=\frac{n+2}{n-2}, 
	\quad 
	q_c:= \frac{n(p-1)}{2}, 
	\quad 
	q_*:= \frac{n(p-1)}{p+1}. 
\]
Note that each of the conditions $q_c>p+1$ and $q_*>2$ is equivalent to $p>p_S$. 
In what follows, we always assume $n\geq3$ and $p>p_S$.

%%%%%%%%%%%%%%%%%%%%%%%%%%%%%%%%%%%%%%%%%%%%%%%%
%%%%%%%%%%%%%%%%%%%%%%%%%%%%%%%%%%%%%%%%%%%%%%%%
%%%%%%%%%%%%%%%%%%%%%%%%%%%%%%%%%%%%%%%%%%%%%%%%
\section{Gradient estimate}\label{sec:grad}
%%%%%%%%%%%%%%%%%%%%%%%%%%%%%%%%%%%%%%%%%%%%%%%%
%%%%%%%%%%%%%%%%%%%%%%%%%%%%%%%%%%%%%%%%%%%%%%%%
%%%%%%%%%%%%%%%%%%%%%%%%%%%%%%%%%%%%%%%%%%%%%%%%
Let $R>0$ and $\Omega$ be any $C^{2+\alpha}$ 
domain in $\R^n$ 
with $0\in \overline{\Omega}$. 
As will be seen in Section \ref{sec:Blim}, 
the proof of Theorem \ref{th:main} 
is based on the study of the localized problem
\begin{equation}\label{eq:fujitaeq}
\left\{ 
\begin{aligned}
	&u_t=\Delta u+|u|^{p-1}u
	&&\mbox{ in }\Omega_R\times(-1,0), \\
	&u=0
	&&\mbox{ on }(\partial \Omega \cap B_R)\times(-1,0), \\
	&u\mbox{ is }C^{2,1}
	&&\mbox{ on }\overline{\Omega_R}\times(-1,0), 
\end{aligned}
\right.
\end{equation}
under the assumption that there exists $M>0$ satisfying 
\begin{equation}\label{eq:Mdef2}
	\sup_{-1<t<0 } \| u(\cdot,t)\|_{L^{q_c}(\Omega_R)} \leq M. 
\end{equation}
Here the boundary condition in \eqref{eq:fujitaeq} 
is ignored if $\partial \Omega \cap B_R=\emptyset$.

In this section, we show a gradient estimate in the 
Lorentz space $L^{q_*,\infty}$ with $q_*:= n(p-1)/(p+1)$, 
which is our key tool to bound a weighted energy 
defined in Section \ref{sec:pre}. 
The method to estimate a term from $|u|^{p-1}u$ 
is based on the idea due to Meyer \cite[Theorem 18.1]{Me97} 
(see also \cite[Proposition 1.5]{Te02}).

\begin{proposition}\label{pro:gradoricri}
If $u$ satisfies \eqref{eq:fujitaeq} and \eqref{eq:Mdef2}, 
then there exists a constant $C>0$ depending on $R$ such that 
\[
	\sup_{-3/4<t<0}\|\nabla u(\cdot,t)\|_{L^{q_*,\infty}(\Omega_{3R/4})}
	\leq  C (M+M^p). 
\]
\end{proposition}

\begin{proof}
In the same spirit of \cite[Proposition A.1]{Fu66}, 
we derive a localized integral equation, 
and then we estimate each of the terms. 
Let $\phi\in C^\infty_0(\R^n)$ satisfy 
$0\leq \phi\leq 1$ in $\R^n$, 
$\phi=0$ in $\R^n\setminus B_{15R/16}$ and  
$\phi=1$ in $B_{7R/8}$. 
Set $v(x,t):=u(x,t)\phi(x)$. 
Then $v$ belongs to $C^{2,1}(\overline{\Omega}\times(-1,0))$ and satisfies 
\[
\left\{ 
\begin{aligned}
	&v_t - \Delta v
	= \phi |u|^{p-1} u -2\nabla \phi \cdot \nabla u
	-u\Delta \phi 
	&&\mbox{ in }\Omega\times (-1,0), \\
	&v=0
	&&\mbox{ on }\partial \Omega\times (-1,0).
\end{aligned}
\right. 
\]
Thus, 
\begin{equation}\label{eq:uGompu}
\begin{aligned}
	u(x,t)  
	&=\int_{\Omega} G_\Omega(x,y,t+7/8) \phi(y) u(y,-7/8) dy  \\
	&\quad 
	+ \int_{-7/8}^t \int_{\Omega} G_\Omega(x,y,t-s) 
	\phi(y) |u(y,s)|^{p-1} u(y,s)  dyds \\
	&\quad 
	- \int_{-7/8}^t \int_{\Omega} G_\Omega(x,y,t-s) 
	(2\nabla \phi \cdot \nabla u+u\Delta \phi) dyds
\end{aligned}
\end{equation}
for $x\in \Omega_{3R/4}$ and $-7/8<t<0$, 
where  $G_\Omega=G_\Omega(x,y,t)$ 
is the Dirichlet heat kernel in $\Omega$. 
Since $G_\Omega(x,y,t)=0$ for $y\in \partial \Omega$ and 
$u\nabla \phi=0$ on $\partial \Omega$, 
integrating by parts in the third term 
in the right-hand side yields 
\begin{equation}\label{eq:uint78}
\begin{aligned}
	u(x,t)  
	&=\int_{\Omega} G_\Omega(x,y,t+7/8) \phi(y) u(y,-7/8) dy  \\
	&\quad 
	+ \int_{-7/8}^t \int_{\Omega} G_\Omega(x,y,t-s) 
	( \phi |u|^{p-1} u + u\Delta \phi) dyds \\
	&\quad 
	+2 \int_{-7/8}^t \int_{\Omega} \nabla_y G_\Omega(x,y,t-s) \cdot 
	\nabla \phi(y) u(y,s) dyds
\end{aligned}
\end{equation}
for $x\in \Omega_{3R/4}$ and $-7/8<t<0$.

Since $\Omega$ is $C^{2+\alpha}$, 
the following estimate holds 
(see \cite[Theorem IV.16.3]{LSUbook} for instance): 
There exists a constant $C>0$ such that 
\begin{equation}\label{eq:derivKjes}
	|\nabla_x^j G_\Omega(x,y,t)|\leq C K_j(x-y,t)\quad  (j=0,1,2) 
\end{equation}
for $x,y\in\Omega$ and $0<t<1$, where 
\begin{equation}\label{eq:Kjdef}
	K_j(x,t) :=t^{-\frac{n}{2}-\frac{j}{2}} e^{-\frac{|x|^2}{C t}} 
	\quad (j=0,1,2). 
\end{equation}
Remark that the constant in \eqref{eq:derivKjes} 
and \eqref{eq:Kjdef} depends only on $n$, $\Omega$ and 
the length of the time interval $(0,1)$. 
We prepare an estimate of 
$\partial_{x_i}\partial_{y_j} G_\Omega(x,y,t)$. 
By the semigroup property and \eqref{eq:derivKjes}, we have 
\[
\begin{aligned}
	|\partial_{x_i}\partial_{y_j} G_\Omega(x,y,t)|
	&= 
	\left| \int_\Omega \partial_{x_i} G_\Omega(x,z,t/2) 
	\partial_{y_j} G_\Omega(z,y,t/2) dz \right| \\
	&\leq 
	C \int_\Omega 
	K_1(x-z,t/2)
	K_1(z-y,t/2) dz \\
	&\leq 
	C t^{-1} \int_{\R^n} 
	G(x-z,Ct/8) G(z-y,Ct/8) dz \\
	&= 
	C t^{-1} G(x-y,Ct/4)
	\leq CK_2(x-y,t), 
\end{aligned}
\]
where $G(x,t):=(4\pi t)^{-n/2}e^{-|x|^2/(4t)}$ and 
we changed the constant $C$ in \eqref{eq:Kjdef}. 
Then by differentiating the integral equation \eqref{eq:uint78} and 
using $K_1(x,t)\leq K_2(x,t)$ for $x\in \Omega$ and $0<t<1$, 
we see that 
\begin{equation}\label{eq:U123def}
\begin{aligned}
	|\nabla u(x,t)|
	&\leq C \int_{\R^n} K_1(x-y,t+7/8) |u(y,-7/8)| 
	\chi_{\Omega_R}(y) dy \\
	&\quad 
	+ C \int_{-7/8}^t \int_{\Omega_R} K_1(x-y,t-s) |u(y,s)|^p dy ds \\
	&\quad 
	+ C \int_{-7/8}^t \int_{\R^n} K_2(x-y,t-s) 
	|u| \chi_{\overline{\Omega_{15R/16}} \setminus \Omega_{7R/8} }  dyds \\
	&=:C U_1(x,t) + C U_2(x,t)+ C U_3(x,t)
\end{aligned}
\end{equation}
for $x\in \Omega_{3R/4}$ and $-7/8<t<0$. 
Remark that each $U_i$ is defined for all $x\in\R^n$ and $-7/8<t<0$.

For $U_1$, from $q_*<q_c$ and the same argument to prove 
the $L^{q_c}$-$L^{q_c}$ estimate for the heat semigroup 
(see \cite[Section 1.1.3]{GGSbook} for instance), it follows that
\[
\begin{aligned}
	\|U_1(\cdot,t)\|_{L^{q_*,\infty}(\Omega_{3R/4})}
	&\leq C\|U_1(\cdot,t)\|_{L^{q_c}(\R^n)} \\
	&\leq C(t+7/8)^{-1/2} 
	\|u(\cdot,-7/8)\chi_{\Omega_R}\|_{L^{q_c}(\R^n)} \\
	&\leq C\|u(\cdot,-7/8)\|_{L^{q_c}(\Omega_R)}
	\leq CM 
\end{aligned}
\]
for $-3/4<t<0$. 
Then, 
\begin{equation}\label{eq:U1concl}
	\sup_{-3/4<t<0} 
	\|U_1(\cdot,t)\|_{L^{q_*,\infty}(\Omega_{3R/4})}
	\leq CM. 
\end{equation}
We consider $U_3$. 
Since $|x-y|\geq R/8$ for $x\in \Omega_{3R/4}$ 
and $y\in \overline{\Omega_{15R/16}} \setminus \Omega_{7R/8}$, 
we have 
\[
	\begin{aligned}
	K_2(x-y,t-s) 
	\chi_{\overline{\Omega_{15R/16}} \setminus \Omega_{7R/8} }(y)
	&\leq 
	C \chi_{\Omega_R}(y) \sup_{s<t} (t-s)^{-\frac{n}{2}-1} e^{-\frac{R^2}{C(t-s)}}  \\
	&\leq C\chi_{\Omega_R}(y)
	\end{aligned}
\]
for $x\in \Omega_{3R/4}$, $y\in \R^n$ and $-7/8<s<t<0$. Therefore, 
\[
	U_3(x,t)
	\leq 
	C \int_{-7/8}^t \int_{\Omega_R} |u| dyds
	\leq 
	CM 
\]
for $x\in \Omega_{3R/4}$ and $-7/8<t<0$. Thus, 
\begin{equation}\label{eq:U3concl}
	\sup_{-3/4<t<0} 
	\|U_3(\cdot,t)\|_{L^{q_*,\infty}(\Omega_{3R/4})}
	\leq CM. 
\end{equation}

We estimate $U_2$. This part is a modification 
of \cite[Theorem 18.1]{Me97}. 
Let $-7/8<t<0$. 
By the change of variables, we have 
\[
\begin{aligned}
	&U_2(x,t)\leq \tilde U(x;t), \quad 
	\tilde U=\tilde U(x;t):=\int_0^\infty S(x,s;t) ds, \\
	&S=S (x,s;t):= 
	\chi_{(0,t+7/8)}(s) \int_{\Omega_R} K_1(x-y,s) |u(y,t-s)|^p dy. 
\end{aligned}
\]
For $\lambda>0$ and $\tau>0$, 
define $E:=\{x\in\Omega_R; \tilde U(x)>\lambda\}$ and 
\[
	\tilde U(x)=\left( \int_0^\tau + \int_\tau^\infty\right) S(x,s) ds 
	=: \tilde U_1(x)+\tilde U_2(x). 
\]
We estimate the Lebesgue measure $|E|$ of $E$. 
By the same argument to prove the $L^{q_c/p}$-$L^\infty$ estimate 
for the heat semigroup, we have 
\[
\begin{aligned}
	\|S(\cdot,s)\|_{L^\infty(\Omega_R)} 
	&\leq C s^{-\frac{np}{2q_c}-\frac{1}{2}} \chi_{(0,t+7/8)}(s)
	\||u(\cdot,t-s)|^p\|_{L^{q_c/p}(\Omega_R)} \\
	&\leq C M^p s^{-\frac{np}{2q_c}-\frac{1}{2}}
\end{aligned}
\]
for any $s>0$, where $C>0$ is independent of $t$. Then, 
\[
	\tilde U_2(x) \leq \int_\tau^\infty S(x,s) ds 
	\leq C M^p \int_\tau^\infty s^{-\frac{np}{2q_c}-\frac{1}{2}} ds 
	=C' M^p \tau^{-\frac{p+1}{2(p-1)}}, 
\]
where $C'>0$ is also independent of $t$.  
For $\lambda>0$, we choose $\tau$ such that 
\begin{equation}\label{eq:tausig}
	C' M^p \tau^{-\frac{p+1}{2(p-1)}} = \frac{\lambda}{2}. 
\end{equation}
Then $\tilde U_2\leq \lambda/2$. 
By setting $E_1:=\{x\in\Omega_R; \tilde U_1(x) >\lambda/2\}$ 
and using $\tilde U_2\leq \lambda/2$ and $\tilde U=\tilde U_1+\tilde U_2$, 
we see that $E\subset E_1$.

From the same argument to prove the $L^{q_c/p}$-$L^{q_c/p}$ estimate, 
it follows that 
\[
\begin{aligned}
	\|S(\cdot,s)\|_{L^{\frac{q_c}{p},\infty}(\Omega_R)}
	&\leq \|S(\cdot,s)\|_{L^\frac{q_c}{p}(\Omega_R)} \\
	&\leq C s^{-\frac{1}{2}} \chi_{(0,t+7/8)}(s) 
	\||u(\cdot,t-s)|^p\|_{L^\frac{q_c}{p}(\Omega_R)} 
	\leq C M^p s^{-\frac{1}{2}}
\end{aligned}
\]
for any $s>0$. Thus, 
\[
	\|\tilde U_1\|_{L^{\frac{q_c}{p},\infty}(\Omega_R)} 
	\leq \int_0^\tau \|S(\cdot,s)\|_{L^{\frac{q_c}{p},\infty}(\Omega_R)} ds 
	\leq C M^p \tau^\frac{1}{2}. 
\]
This together with the H\"older inequality 
for the Lorentz spaces (see \cite[Proposition 2.1]{KY99} for instance) 
shows that
\[
	\int_{E_1} \tilde U_1(x) dx 
	\leq C \| \chi_{E_1} \|_{L^{\frac{q_c}{q_c-p},1}(\Omega_R)} 
	\|\tilde U_1\|_{L^{\frac{q_c}{p},\infty}(\Omega_R)} 
	\leq C M^p |E_1|^{1-\frac{p}{q_c}} \tau^\frac{1}{2}. 
\]
On the other hand, the definition of $E_1$ gives 
$\int_{E_1} \tilde U_1(x) dx\geq (\lambda/2) |E_1|$. 
By $E\subset E_1$ and \eqref{eq:tausig}, we obtain 
\[
	|E|\leq |E_1|\leq C \lambda^{-\frac{q_c}{p}} 
	M^{q_c} \tau^\frac{q_c}{2p} 
	=C M^{p q_*} \lambda^{-q_*}, 
\]
and so $\lambda  |\{ x\in\Omega_R; \tilde U(x)>\lambda\}|^{1/q_*} \leq C M^p$
for any $\lambda>0$. 
This implies that $\|\tilde U(\cdot;t)\|_{L^{q_*,\infty}(\Omega_R)} \leq CM^p$ 
for $-7/8<t<0$, where $C>0$ is independent of $t$. 
Hence by the definition of $\tilde U$ and $M^p$, we see that  
\[
	\sup_{-3/4<t<0}
	\|U_2(\cdot,t)\|_{L^{q_*,\infty}(\Omega_{3R/4})} 
	\leq 
	\sup_{-3/4<t<0} \|\tilde U(\cdot;t)\|_{L^{q_*,\infty}(\Omega_R)}
	\leq CM^p. 
\]
Combining this inequality, \eqref{eq:U123def}, 
\eqref{eq:U1concl} and \eqref{eq:U3concl}, we  obtain 
the desired inequality. 
\end{proof}

%%%%%%%%%%%%%%%%%%%%%%%%%%%%%%%%%%%%%%%%%%%%%%%%
%%%%%%%%%%%%%%%%%%%%%%%%%%%%%%%%%%%%%%%%%%%%%%%%
%%%%%%%%%%%%%%%%%%%%%%%%%%%%%%%%%%%%%%%%%%%%%%%%
\section{Localized weighted energy}\label{sec:pre}
%%%%%%%%%%%%%%%%%%%%%%%%%%%%%%%%%%%%%%%%%%%%%%%%
%%%%%%%%%%%%%%%%%%%%%%%%%%%%%%%%%%%%%%%%%%%%%%%%
%%%%%%%%%%%%%%%%%%%%%%%%%%%%%%%%%%%%%%%%%%%%%%%%
Let $u$ be a solution of \eqref{eq:fujitaeq} satisfying 
the bound \eqref{eq:Mdef2}. 
In this section, 
we define a localized weighted energy of $u$ 
analogous to Giga, Matsui and Sasayama \cite{GMS04,GMS04s}
and prove its quasi-monotonicity 
without assuming the convexity of $\Omega$. 
Our computations to prove quasi-monotonicity are
in the same spirit of Chou and Du \cite{CD10}, 
but the details are different. 
Among the results in this section, 
we will refer only Lemma \ref{lem:Ebdd} 
and Proposition \ref{pro:31} in the subsequent sections.

Let $n\geq3$, $p>p_S$ and 
$\Omega$ be any $C^{2+\alpha}$ domain 
in $\R^n$ with $0\in \overline{\Omega}$. 
We fix $R>0$ as one of the following: 
\begin{align}
	&\label{eq:R1def}
	\mbox{In the case $0\in \Omega$, we fix 
	$0<R<1/2$ such that $\overline{B_R} \subset \Omega$. }\\
	&\label{eq:R2def}
	\left\{ 
	\begin{aligned}
	&\mbox{In the case $\partial\Omega\neq \emptyset$ 
	and $0\in \partial \Omega$, 
	we fix $0<R<1/2$ such that}\\
	&\mbox{there exists 
	$f\in C^{2+\alpha}_0(\R^{n-1})$ satisfying 
	$f(0)=0$, $\nabla'f(0)=0$, }\\
	&\mbox{$\|\nabla' f\|_{L^\infty(\R^{n-1})}\leq 1/2$ and $\Omega_R
	=\{x\in B_R; x_n>f(x')\}$}\\
	&\mbox{by relabeling and reorienting 
	the coordinates axes if necessary. }
	\end{aligned}
	\right.
\end{align}
Here $\nabla'f$ is the gradient on $\R^{n-1}$. 
Remark that the existence of $f$ in \eqref{eq:R2def} 
is guaranteed by the smoothness of $\Omega$.

%%%%%%%%%%%%%%%%%%%%%%%%%%%%%%%%%%%%%%%%%%%%%%%%
%%%%%%%%%%%%%%%%%%%%%%%%%%%%%%%%%%%%%%%%%%%%%%%%
%%%%%%%%%%%%%%%%%%%%%%%%%%%%%%%%%%%%%%%%%%%%%%%%
\subsection{Definition and change of variables}\label{subsec:dcv}
%%%%%%%%%%%%%%%%%%%%%%%%%%%%%%%%%%%%%%%%%%%%%%%%
%%%%%%%%%%%%%%%%%%%%%%%%%%%%%%%%%%%%%%%%%%%%%%%%
%%%%%%%%%%%%%%%%%%%%%%%%%%%%%%%%%%%%%%%%%%%%%%%%
We define a localized weighted energy $E$ 
and show its boundedness by using Proposition \ref{pro:gradoricri}. 
To obtain quasi-monotonicity, 
we locally straighten the boundary. 
After that, we introduce backward similarity variables and 
derive the corresponding representation of $E$.

Let $\varphi\in C^\infty([0,\infty))$ satisfy 
$\varphi(z)=1$ for $0\leq z\leq 1/2$, 
$0< \varphi(z)< 1$ for $1/2< z< 1$, 
$\varphi(z)=0$ for $z\geq 1$ and 
$\varphi'(z)\leq 0$ for $z\geq0$. 
For $x,\tilde x\in \R^n$, $\tilde t>t$ and $r>0$, we set 
$\phi_r=\phi_{\tilde x,r}(x):=\varphi(|x-\tilde x|/r)$ and 
\[
	K=K_{(\tilde x,\tilde {t})}(x,t):= 
	(\tilde t-t)^{-\frac{n}{2}} 
	e^{-\frac{|x-\tilde x|^2}{4(\tilde t-t)}}. 
\]
For $\tilde x\in \overline{\Omega_{R/4}}$
and $-1<t<\tilde t\leq 0$, 
define a localized weighted energy by 
\begin{equation}\label{eq:Edeft}
\begin{aligned}
	&E(t) = E_{(\tilde x,\tilde t)}(t;\phi_{\tilde x,R/4})  \\
	&:= (\tilde t-t)^\frac{p+1}{p-1}
	\int_{\Omega_R} \left( \frac{|\nabla u(x,t)|^2}{2} 
	- \frac{|u(x,t)|^{p+1}}{p+1} 
	+ \frac{|u(x,t)|^2}{2(p-1)(\tilde t-t)}
	\right) \\
	&\qquad \times 
	K_{(\tilde x,\tilde {t})}(x,t) \phi_{\tilde x,R/4}^2(x) dx. 
\end{aligned}
\end{equation}
Remark that $u$ is defined on $(\overline{\Omega_R}\cap B_R)\times (-1,0)$, 
but we mainly consider the time interval $(-1/2,0)$ to 
apply Proposition \ref{pro:gradoricri}. 
Note that 
\begin{equation}\label{eq:tilphisupp}
	\supp \phi_{\tilde x,R/4} 
	\subset B_{R/2}
	\quad\mbox{ for }
	\tilde x\in \overline{\Omega_{R/4}}. 
\end{equation}

The following lemma guarantees the boundedness of $E$. 

\begin{lemma}\label{lem:Ebdd}
There exists $C>0$ such that 
\[
	|E_{(\tilde x,\tilde t)}(t;\phi_{\tilde x,R/4})|
	\leq C(M+M^p)^2 
\]
for any $\tilde x\in \overline{\Omega_{R/4}}$ 
and $-1/2<t<\tilde t\leq 0$. 
\end{lemma}

\begin{proof}
From the H\"older inequality 
for the Lorentz spaces (see \cite[Proposition 2.1]{KY99}), 
\eqref{eq:tilphisupp}, Proposition \ref{pro:gradoricri} 
and a direct computation, 
it follows that  
\begin{equation}\label{eq:nabucalpp}
\begin{aligned}
	&\int_{\Omega_R} 
	|\nabla u|^2 K_{(\tilde x,\tilde t)} \phi_{\tilde x,R/4}^2 dx \\
	&\leq C (\tilde t -t)^{-\frac{n}{2}} 
	\| \nabla u(\cdot,t) \|_{L^{q_*,\infty}(\Omega_{R/2})}^2  
	\left\| e^{-\frac{|\cdot -\tilde x|^2}{8(\tilde t-t)}} 
	\right\|_{L^{\frac{2q_*}{q_*-2},2} (\R^n)}^2 \\
	&\leq 
	C(M+M^p)^2 (\tilde t-t)^{-\frac{p+1}{p-1}}
\end{aligned}
\end{equation}
for $-1/2<t<\tilde t\leq 0$.
The H\"older inequality and \eqref{eq:Mdef2} show that 
\[
\begin{aligned}
	&
	\begin{aligned}
	\int_{\Omega_R} |u|^{p+1} K_{(\tilde x,\tilde t)} 
	\phi_{\tilde x,R/4}^2 dx 
	&\leq 
	(\tilde t-t)^{-\frac{n}{2}} 
	\| u \|_{L^{q_c}(\Omega_{R/2})}^{p+1} 
	\left( \int_{\R^n} e^{-\frac{|x-\tilde x|^2}{C(\tilde t-t)}} dx 
	\right)^{1-\frac{p+1}{q_c}} \\
	&\leq CM^{p+1} (\tilde t-t)^{-\frac{p+1}{p-1}}, 
	\end{aligned} \\
	&\begin{aligned}
	\int_{\Omega_R} |u|^2 
	K_{(\tilde x,\tilde t)} \phi_{\tilde x,R/4}^2 dx 
	\leq CM^2 (\tilde t-t)^{-\frac{2}{p-1}}, 
	\end{aligned}
\end{aligned}
\]
for $-1/2<t<\tilde t\leq0$. 
The lemma follows from the above estimates. 
\end{proof}

We locally straighten the boundary. 
In the case \eqref{eq:R2def}, 
we define $C^{2+\alpha}$ maps 
$\Phi=(\Phi_1,\ldots,\Phi_n)$ and 
$\Psi=(\Psi_1,\ldots,\Psi_n)$ by 
\[
\left\{
\begin{aligned}
	&\xi_i = x_i =: \Phi_i(x), \\
	&\xi_n = x_n - f(x') =: \Phi_n(x), 
\end{aligned}
\right.  
\quad 
\left\{
\begin{aligned}
	&x_i = \xi_i =: \Psi_i(\xi), \\
	&x_n = \xi_n + f(\xi') =: \Psi_n(\xi), 
\end{aligned}
\right. 
\]
for $i=1,\ldots, n-1$. 
To handle the case \eqref{eq:R1def} in a unified way, 
we also set 
\[
\left\{
\begin{aligned}
	&\xi_i = x_i =: \Phi_i(x), \\
	&\xi_n = x_n =: \Phi_n(x), 
\end{aligned}
\right.  
\quad 
\left\{
\begin{aligned}
	&x_i = \xi_i =: \Psi_i(\xi), \\
	&x_n = \xi_n =: \Psi_n(\xi), 
\end{aligned}
\right. 
\]
for \eqref{eq:R1def}. 
We note that the maps $\Phi$ and $\Psi$ 
for \eqref{eq:R1def} are identity maps 
and they are obtained by setting $f\equiv0$ 
in the definitions of $\Phi$ and $\Psi$ for \eqref{eq:R2def}.

We write $\Phi(x)=\xi$ and $\Psi(\xi)=x$. 
Set 
\begin{equation}\label{eq:uhatdef}
	\hat u(\xi,t):= u(\Psi (\xi),t). 
\end{equation}
Then, direct computations show that 
\begin{equation}\label{eq:nabxihatu}
\begin{aligned}
	\nabla_x u(x,t) 
	&= (
	\partial_{\xi_1}\hat u 
	- (\partial_{\xi_n}\hat u) \partial_{\xi_1} f,
	\ldots, 
	\partial_{\xi_{n-1}}\hat u 
	- (\partial_{\xi_n}\hat u) \partial_{\xi_{n-1}} f, 
	\partial_{\xi_n}\hat u ) \\
	&=
	(\nabla'\hat u - (\partial_{\xi_n}\hat u) \nabla' f, \partial_{\xi_n}\hat u)
\end{aligned}
\end{equation}
and that 
\[
\begin{aligned}
	\Delta_x u(x,t)
	&= \Delta_\xi \hat u
	-2\sum_{i=1}^{n-1} (\partial_{\xi_i}\partial_{\xi_n} \hat u) 
	\partial_{\xi_i} f 
	+\partial_{\xi_n}^2 \hat u \sum_{i=1}^{n-1} (\partial_{\xi_i} f)^2  
	-\partial_{\xi_n} \hat u \sum_{i=1}^{n-1} \partial_{\xi_i}^2 f 
	\\
	&= \Delta_{\xi} \hat u 
	- 2\nabla_{\xi}' (\partial_{\xi_n} \hat u )
	\cdot \nabla_{\xi}' f 
	+(\partial_{\xi_n}^2 \hat u) |\nabla_{\xi}' f|^2 
	-(\partial_{\xi_n} \hat u) \Delta_{\xi}' f, 
\end{aligned}
\]
where $\nabla'$ and $\Delta'$ are 
the gradient and Laplacian on $\R^{n-1}$ 
with respect to the first $(n-1)$ components, respectively. 
Remark that 
the terms in the right-hand sides are evaluated at 
$(\xi,t)=(\Phi(x),t)$. 
Since $u$ satisfies \eqref{eq:fujitaeq}, 
we see that $\hat u$ satisfies 
\begin{equation}\label{eq:hatueq}
\left\{ 
\begin{aligned}
	&\hat u_t - \hat A \hat u 
	= |\hat u|^{p-1} \hat u 
	&&\mbox{ in } 
	\Phi(\Omega_R)\times (-1,0),   \\
	&\hat u=0
	&&\mbox{ on }\Phi(\partial\Omega\cap B_R)\times (-1,0). 
\end{aligned}
\right. 
\end{equation}
Here, by abbreviations of the subscripts $\xi$ and $\xi_n$, we set 
\[
	\hat A\hat u:=
	\Delta \hat u 
	- 2\nabla' (\partial_n \hat u )
	\cdot \nabla' f 
	+(\partial_n^2 \hat u) |\nabla' f|^2 
	-(\partial_n \hat u) \Delta' f. 
\]
For $\tilde x\in \overline{\Omega_{R/4}}$, 
we write $\tilde \xi :=\Phi(\tilde x)$. 
Note that $\tilde x=\Psi(\tilde \xi)$. 
We perform the change of variables $x=\Psi(\xi)$ 
in \eqref{eq:Edeft} by using the relation \eqref{eq:nabxihatu}, 
the definition of $K$ and 
the fact that the Jacobian determinant equals $1$ from the definition of $\Psi$, 
that is, $dx=d\xi$. 
Then $E$ satisfies
\begin{equation}\label{eq:Etstr}
\begin{aligned}
	E(t) &=(\tilde t-t)^{\frac{p+1}{p-1} -\frac{n}{2}}
	\int_{\Phi(\Omega_R)} 
	\left( 
	\frac{ |\widehat \nabla \hat u|^2}{2}
	-\frac{|\hat u|^{p+1}}{p+1} +\frac{\hat u^2}{2(p-1)(\tilde t-t)} \right) \\
	&\qquad \times 
	e^{-\frac{|\Psi(\xi)-\Psi(\tilde \xi)|^2}{4(\tilde t-t)} } 
	\phi^2_{\Psi(\tilde \xi),R/4}(\Psi(\xi)) d\xi, 
\end{aligned}
\end{equation} 
where 
$\widehat \nabla \hat u(\xi,t)
:= (\nabla'\hat u - (\partial_n\hat u) \nabla' f, \partial_n\hat u)$.

We introduce the backward similarity variables 
\[
	\eta :=\frac{\xi-\tilde \xi}{(\tilde t-t)^{1/2}}, 
	\quad \tau:=-\log(\tilde t-t). 
\]
Then the rescaled functions are given by 
\begin{align} 
	&\label{eq:backw}
	w(\eta,\tau):= e^{-\frac{1}{p-1}\tau} 
	\hat u( 
	\tilde \xi+e^{-\frac{1}{2}\tau} \eta, \tilde t-e^{-\tau}), \\
	&\label{eq:backg}
	g(\eta',\tau):= 
	e^{\frac{1}{2}\tau} f( \tilde \xi'+e^{-\frac{1}{2}\tau} \eta' ). 
\end{align}
Note that 
\[
\begin{aligned}
	&\xi=\tilde \xi+e^{-\frac{1}{2}\tau} \eta, 
	\quad \tilde t-t=e^{-\tau}, \\
	&\hat u(\xi,t) 
	=(\tilde t-t)^{-\frac{1}{p-1}} 
	w( (\tilde t-t)^{-\frac{1}{2}} (\xi-\tilde \xi), -\log (\tilde t-t) ). 
\end{aligned}
\]

Since $\hat u$ satisfies \eqref{eq:hatueq}, 
we see that $w$ solves \begin{equation}\label{eq:weq}
	\left\{ 
	\begin{aligned}
	&w_\tau + \frac{1}{2} \eta\cdot \nabla w
	+\frac{1}{p-1} w -Aw - |w|^{p-1}w =0, \\
	&\quad \eta\in \Omega(\tau), \  
	\tau\in  (-\log(\tilde t+1/2),\infty), \\
	&w=0, \quad \eta\in e^{\tau/2}(\Phi(\partial\Omega\cap B_R)-\tilde \xi), 
	\end{aligned}
	\right.
\end{equation}
where the time interval $(-\log(\tilde t+1), \infty)$ 
has been shortened to $(-\log(\tilde t+1/2), \infty)$ 
for using Proposition \ref{pro:gradoricri} safely and 
\[
	Aw:=
	\Delta w 
	- 2\nabla' (\partial_n w )
	\cdot \nabla' g 
	+(\partial_n^2 w) |\nabla' g|^2 
	-(\partial_n w) \Delta' g  \\
\]
by abbreviations of the subscripts $\eta$ and $\eta_n$. 
In addition, 
\[
	\Omega(\tau):=\{\eta\in \R^n; \tilde \xi+e^{-\tau/2} \eta
	\in\Phi(\Omega_R)\}
	= e^{\tau/2}(\Phi(\Omega_R)-\tilde \xi). 
\]

By using the backward similarity variables, 
\eqref{eq:Etstr} can be written as 
\[
\begin{aligned}
	E(t) &=
	\int_{\Omega(\tau)} 
	\left( 
	\frac{ |\widehat \nabla w|^2}{2}
	-\frac{|w|^{p+1}}{p+1} 
	+\frac{w^2}{2(p-1)} \right)  \\
	&\qquad 
	\times \exp\left( -\frac{e^\tau}{4} |\Psi(\tilde \xi+e^{-\tau/2}\eta)
	-\Psi(\tilde \xi)|^2 \right) \\
	&\qquad \times 
	\varphi^2\left( \frac{4}{R} 
	|\Psi(\tilde \xi+e^{-\tau/2}\eta)-\Psi(\tilde \xi)|
	\right) d\eta, 
\end{aligned}
\]
where $\widehat \nabla w
:=(\nabla' w-(\partial_n w) \nabla' g, \partial_n w)$. 
We observe that 
\begin{equation}\label{eq:Psieta}
\begin{aligned}
	&|\Psi(\tilde \xi+e^{-\tau/2}\eta)-\Psi(\tilde \xi)|^2 \\
	&= e^{-\tau}
	|(\eta',\eta_n+g(\eta',\tau)-g(0,\tau))|^2 \\
	&= e^{-\tau}
	( |\eta|^2 + 2(g(\eta',\tau)-g(0,\tau))\eta_n + (g(\eta',\tau)-g(0,\tau))^2 ). 
\end{aligned}
\end{equation}
Then by setting 
\begin{align}
	&\label{eq:cEdef}
	\cE(\tau):= 
	\int_{\Omega(\tau)} 
	\left( 
	\frac{ |\widehat \nabla w|^2}{2}
	-\frac{|w|^{p+1}}{p+1} 
	+\frac{w^2}{2(p-1)} \right) 
	\rho \psi^2 d\eta, \\
	&\label{eq:rhodef}
	\begin{aligned}
	\rho=\rho(\eta,\tau) &:=\exp\bigg( 
	-\frac{1}{4} 
	( |\eta|^2 + 2(g(\eta',\tau)-g(0,\tau))\eta_n \\
	&\quad + (g(\eta',\tau)-g(0,\tau))^2 ) \bigg), 
	\end{aligned} \\
	&\label{eq:psietadef}
	\psi=\psi(\eta,\tau):= 
	\varphi\left( 
	\frac{4}{R}
	e^{-\frac{\tau}{2}} |(\eta',\eta_n+g(\eta',\tau)-g(0,\tau))|  
	\right), 
\end{align}
we see that $E(t)=\cE(\tau)$ with $\tau=-\log(\tilde t-t)$. 
By direct computations, we note that 
\begin{equation}\label{eq:rhocalc}
\left\{ 
\begin{aligned}
	&\partial_i \rho 
	= - \frac{1}{2} (\eta_i +(\eta_n +g(\eta',\tau)-g(0,\tau)) \partial_i g) \rho 
	\quad (i=1,\ldots,n-1), \\
	&\begin{aligned}
	\partial_i^2 \rho &= 
	-\frac{1}{2} ( 1+ (\partial_i g)^2 
	+ (\eta_n +g(\eta',\tau)-g(0,\tau)) \partial_i^2 g )\rho \\
	&\quad 
	+\frac{1}{4} ( \eta_i^2 
	+ 2(\eta_n +g(\eta',\tau)-g(0,\tau))\eta_i \partial_i g )\rho \\
	&\quad 
	+\frac{1}{4} (\eta_n +g(\eta',\tau)-g(0,\tau))^2 (\partial_i g)^2 \rho 
	\quad (i=1,\ldots,n-1), 
	\end{aligned} \\
	&\partial_n \rho 
	=-\frac{1}{2} (\eta_n +g(\eta',\tau)-g(0,\tau)) \rho, \\
	&\partial_n^2 \rho 
	= -\frac{1}{2}\rho  +\frac{1}{4} (\eta_n +g(\eta',\tau)-g(0,\tau))^2 \rho
\end{aligned}
\right. 
\end{equation}
and that 
\begin{equation}\label{eq:rhotau}
	\rho_\tau 
	= -\frac{1}{2} \partial_\tau(g(\eta',\tau) - g(0,\tau)) 
	(\eta_n + g(\eta',\tau)-g(0,\tau)) \rho. 
\end{equation}

%%%%%%%%%%%%%%%%%%%%%%%%%%%%%%%%%%
%%%%%%%%%%%%%%%%%%%%%%%%%%%%%%%%%%
%%%%%%%%%%%%%%%%%%%%%%%%%%%%%%%%%%
\subsection{Quasi-monotonicity}
%%%%%%%%%%%%%%%%%%%%%%%%%%%%%%%%%%
%%%%%%%%%%%%%%%%%%%%%%%%%%%%%%%%%%
%%%%%%%%%%%%%%%%%%%%%%%%%%%%%%%%%%
We prove the quasi-monotonicity of $E$. 
This property plays a crucial role 
in the proof of the $\eps$-regularity theorem 
and also in the blow-up analysis.

\begin{proposition}\label{pro:31}
Fix $R>0$ such that either \eqref{eq:R1def} or \eqref{eq:R2def} holds. 
Let $u$ be a solution of \eqref{eq:fujitaeq} satisfying \eqref{eq:Mdef2}. 
Then there exists a constant $C>0$ depending only on $n$, $p$, $\Omega$ 
and $R$ such that 
\begin{equation}\label{eq:419}
\begin{aligned}
	&
	\begin{aligned}
	&E_{(\tilde x,\tilde{t})}(t;\phi_{\tilde x,R/4})
	+
	\frac{1}{2} \int_{t'}^t 
	(\tilde{t}-s)^{\frac{2}{p-1} -\frac{n}{2}-1} \\
	&\qquad \times 
	\int_{\Phi(\Omega_R)}
	\left| \frac{\hat u}{p-1}
	+\frac{(\xi-\Phi(\tilde x))\cdot\nabla \hat u}{2}
	- (\tilde t-s)\hat u_s \right|^2  \\
	&\qquad 
	\times 
	e^{-\frac{|\Psi(\xi)-\tilde x|^2}{4(\tilde t -s)}}
	\phi^2_{\tilde x,R/4}(\Psi(\xi)) d\xi ds 
	\end{aligned} \\
	&\leq E_{(\tilde x,\tilde{t})}(t';\phi_{\tilde x,R/4})
	+ C(M+M^p)^2 (\tilde{t}-t')^\frac{1}{2} 
\end{aligned}
\end{equation}
for any $\tilde x\in \overline{\Omega_{R/4}}$ and 
$-1/2< t'<t<\tilde t\leq 0$. 
\end{proposition}

To prove this, we compute the derivative of $\cE$. 

\begin{lemma}\label{lem:cederivR2}
The derivative of $\cE(\tau)$ satisfies 
\begin{equation}
\label{eq:derivE}
	\frac{d}{d\tau} 
	\cE(\tau) 
	=
	-\int w_\tau^2 \rho \psi^2  
	-\cB(\tau) + \cR(\tau). 
\end{equation}
Here 
\[
\begin{aligned}
	\cB(\tau) &:= \frac{1}{4} \int_{e^{\tau/2} 
	( \Phi(\partial\Omega\cap B_R)-\tilde \xi )} 
	(\partial_\nu w)^2 \rho \psi^2  \\
	&\qquad \times 
	(\nu\cdot \eta) 
	( 1-2\nu_n \partial_{\nu'}g  +|\nabla'g|^2  \nu_n^2 )dS(\eta), 
\end{aligned}
\]
$\nu=(\nu',\nu_n)$ is the outward unit normal, 
$\partial_\nu w:=\nabla w\cdot \nu$, 
$\partial_{\nu'} g:=\nabla' g\cdot \nu'$, 
$dS$ is the surface area element and 
\[
\begin{aligned}
	\cR(\tau)&:= 
	\int  \frac{w^2}{p-1} \rho \psi \psi_\tau  
	- \int \frac{2|w|^{p+1}}{p+1} \rho \psi \psi_\tau 
	-2 \int w_\tau \rho  \psi \nabla w \cdot \nabla \psi \\
	&\quad 
	+ \int |\nabla w|^2 \rho \psi \psi_\tau 
	+ 2 \int w_\tau \rho \psi (\partial_n \psi) \nabla'w\cdot \nabla'g \\
	&\quad 
	- 2\int (\partial_n w) \rho \psi \psi_\tau \nabla' w \cdot \nabla'g 
	+2\int  w_\tau (\partial_n w) \rho \psi \nabla'\psi \cdot \nabla'g \\
	&\quad 
	- 2 \int w_\tau (\partial_n w) \rho \psi (\partial_n \psi) |\nabla' g|^2 
	+ \int  (\partial_n w)^2 \rho \psi \psi_\tau  |\nabla' g|^2 \\
	&\quad 
	-\int (\partial_n w) \rho \psi^2 \nabla' w \cdot \nabla'g_\tau 
	+\int (\partial_n w)^2 \rho \psi^2 \nabla' g\cdot \nabla' g_\tau \\
	&\quad 
	+\frac{1}{2}\int w_\tau (\partial_n w ) \rho \psi^2 
	(  g(\eta',\tau)-g(0,\tau)-\eta'\cdot \nabla' g )  \\ 
	&\quad + 
	\int \left( 
	\frac{ |\widehat \nabla w|^2}{2}
	-\frac{|w|^{p+1}}{p+1} 
	+\frac{w^2}{2(p-1)} \right) 
	\rho_\tau \psi^2 
\end{aligned}
\]
with the abbreviation $\int(\cdots)=\int_{\Omega(\tau)}(\cdots) d\eta$. 
\end{lemma}

\begin{proof}
In what follows, we will perform integration by parts several times, and so we need to observe the boundary value 
of $w\psi$. 
We claim that 
\begin{equation}\label{eq:wpsi0}
	w\psi=0 \quad \mbox{ on }\partial \Omega(\tau). 
\end{equation}
To prove this, we note that 
\[
	\partial \Omega(\tau) = 
	e^{\tau/2} ( \Phi(\partial\Omega\cap B_R)-\tilde \xi )
	\cup e^{\tau/2} ( \Phi(\Omega\cap \partial B_R)-\tilde \xi ). 
\]
For $\eta \in e^{\tau/2} ( \Phi(\partial \Omega\cap B_R)-\tilde \xi )$, 
by the boundary condition in \eqref{eq:weq}, 
we obtain $w(\eta,\tau)=0$.
On the other hand, 
for $\eta\in e^{\tau/2} ( \Phi(\Omega\cap \partial B_R)-\tilde \xi )$, 
we have $\Phi(x) = \xi = e^{-\tau/2}\eta + \tilde \xi 
\in \Phi (\Omega\cap \partial B_R)$.  
Thus $x\in \Omega\cap \partial B_R$ and $|x|=R$. 
This together with $\tilde x\in \overline{\Omega_{R/4}}$ gives 
$|x-\tilde x|\geq 3R/4$. 
Therefore $\psi(\eta,\tau)=\varphi(4|x-\tilde x|/R)=0$ for 
$\eta\in e^{\tau/2} ( \Phi(\Omega\cap \partial B_R)-\tilde \xi )$. 
Hence \eqref{eq:wpsi0} holds.

For simplicity, we write 
$\int(\cdots)=\int_{\Omega(\tau)}(\cdots) d\eta$ 
when no confusion can arise. 
By \eqref{eq:wpsi0}, we see that 
\begin{equation}\label{eq:I0cal}
\begin{aligned}
	&\frac{d}{d\tau} 
	\int \left( 
	- \frac{|w|^{p+1}}{p+1} + \frac{w^2}{2(p-1)}
	\right) \rho \psi^2 \\
	&=
	\int \frac{w w_\tau}{p-1} \rho \psi^2  
	- \int w w_\tau |w|^{p-1} \rho \psi^2 \\
	&\quad + \int \left( 
	- \frac{|w|^{p+1}}{p+1} + \frac{w^2}{2(p-1)}
	\right) \rho_\tau \psi^2 
	+ \cR_0,  
\end{aligned}
\end{equation}
where 
\begin{equation}\label{eq:R0cal}
	\cR_0:= \int  \frac{w^2}{p-1} \rho \psi \psi_\tau  
	- \int \frac{2|w|^{p+1}}{p+1} \rho \psi \psi_\tau. 
\end{equation}
On the other hand, by taking the computation
\begin{equation}\label{eq:hatwc}
\begin{aligned}
	|\widehat \nabla w|^2 
	&= 
	|\nabla w-(\partial_n w \nabla' g,0)|^2 \\
	&= |\nabla w|^2 - 2 (\partial_n w) \nabla'w\cdot \nabla'g 
	+ (\partial_n w)^2 |\nabla'g|^2 
\end{aligned}
\end{equation}
into account, we set 
\begin{equation}\label{eq:I123def}
\begin{aligned}
	\frac{d}{d\tau} 
	\int \frac{ |\widehat \nabla w|^2}{2} \rho \psi^2 
	&= \frac{1}{2} \frac{d}{d\tau}
	\int |\nabla w|^2 \rho \psi^2 
	- \frac{d}{d\tau} 
	\int (\partial_n w) \rho \psi^2 \nabla' w \cdot \nabla' g  \\
	&\quad 
	+ \frac{1}{2} \frac{d}{d\tau} 
	\int (\partial_n w)^2 \rho \psi^2 |\nabla' g|^2  \\
	&=:\frac{1}{2} \frac{d I_1}{d\tau} - \frac{d I_2}{d\tau} 
	+ \frac{1}{2} \frac{d I_3}{d\tau}. 
\end{aligned}
\end{equation}
We compute the derivatives of $I_1$, $I_2$ and $I_3$ 
in the following way: 
\begin{enumerate}
\item
If a term contains the derivative of $\psi$, keep the term as it is. 
\item
If a term contains the spatial derivative(s) of $w_\tau$, 
perform integration by parts to remove the spatial derivative(s). 
\item
If a term does not contain $w_\tau$ but contains 
the second order spatial derivatives of $w$, 
perform integration by parts for lowering the order. 
\end{enumerate}

For $I_1$, integration by parts and \eqref{eq:wpsi0} show that 
\[
\begin{aligned}
	\frac{1}{2}\frac{d I_1}{d\tau}  
	&= -\frac{d}{d\tau} \left( 
	\int (w\psi \nabla w \cdot \nabla \psi) \rho
	+ 
	\frac{1}{2}\int w \psi^2 \nabla \cdot (\rho \nabla w)
	\right) \\
	&= -\int (w\psi \nabla w \cdot \nabla \psi)_\tau \rho
	-\frac{1}{2}\int w_\tau \psi^2 \nabla \cdot (\rho\nabla w) \\
	&\quad 
	-\int w\psi \psi_\tau \nabla \cdot (\rho\nabla w) 
	-\frac{1}{2}\int w \psi^2 \nabla \cdot (\rho\nabla w_\tau) \\
	&\quad 
	-\frac{1}{2}\int w \psi^2 \nabla\cdot (\rho_\tau \nabla w) 
	-\int w\rho_\tau \psi \nabla w \cdot \nabla \psi. 
\end{aligned}
\]
Integrating by parts twice and \eqref{eq:wpsi0} yield 
\[
\begin{aligned}
	- \frac{1}{2} \int w \psi^2 \nabla \cdot (\rho\nabla w_\tau) 
	&= -\frac{1}{2}\int w_\tau \psi^2 \nabla \cdot(\rho\nabla w) 
	- \int w_\tau \rho \psi \nabla w \cdot \nabla \psi \\
	&\quad + \int w \rho \psi \nabla w_\tau \cdot \nabla \psi  
	+ \frac{1}{2} \int_{\partial\Omega(\tau)} 
	w_\tau (\partial_\nu w)\rho \psi^2  dS. 
\end{aligned}
\]
In addition, we see that 
\[
	-\frac{1}{2}\int w \psi^2 \nabla\cdot (\rho_\tau \nabla w) 
	=\frac{1}{2} \int |\nabla w|^2 \rho_\tau \psi^2 
	+ \int w \rho_\tau \psi \nabla w \cdot \nabla \psi. 
\]
The above computations imply that 
\begin{equation}\label{eq:I1esb}
\begin{aligned}
	\frac{1}{2}\frac{d I_1}{d\tau}  
	&= 
	-\int w_\tau \psi^2 \nabla \cdot (\rho\nabla w)
	+\frac{1}{2} \int_{\partial\Omega(\tau)} w_\tau \rho \psi^2 
	\nabla w\cdot \nu dS \\
	&\quad 
	+ \frac{1}{2} \int |\nabla w|^2 \rho_\tau \psi^2 
	+ \cR_1, 
\end{aligned}
\end{equation}
where 
\[
\begin{aligned}
	\cR_1 &:= 
	-\int (w\psi \nabla w \cdot \nabla \psi)_\tau \rho 
	-\int w\psi \psi_\tau \nabla \cdot (\rho\nabla w) \\
	&\quad  
	- \int w_\tau \rho \psi \nabla w \cdot \nabla \psi 
	+ \int w \rho \psi \nabla w_\tau \cdot \nabla \psi. 
\end{aligned}
\]
For $\cR_1$, 
expanding the first term 
and integrating by parts in the second term, 
we obtain  
\begin{equation}\label{eq:R1cal}
	\cR_1 = 
	-2 \int w_\tau \rho  \psi \nabla w \cdot \nabla \psi 
	+ \int |\nabla w|^2 \rho \psi \psi_\tau.  
\end{equation}

By the argument of \cite[Proposition 2.1]{GK87} 
with $\Phi(\partial \Omega \cap B_R)\subset \R^n_+$, 
we can see that 
\begin{equation}\label{eq:I1bou}
\begin{aligned}
	&\frac{1}{2}\int_{\partial\Omega(\tau)} w_\tau 
	(\partial_\nu w) \rho \psi^2 dS \\
	&= -\frac{1}{4}\int_{
	e^{\tau/2} ( \Phi(\partial\Omega\cap B_R)-\tilde \xi )} 
	(\partial_\nu w)^2 \rho \psi^2  (\nu\cdot \eta) dS. 
\end{aligned}
\end{equation}
Indeed, 
from the boundary conditions in \eqref{eq:hatueq} and \eqref{eq:weq}, 
it follows that 
\[
	0=\hat u_t= (\tilde t-t)^{-\frac{1}{p-1}-1} 
	\left( \nabla w\cdot \frac{y}{2}+w_\tau \right), 
\]
and so $w_\tau=-\nabla w\cdot (\eta/2)$ 
on $e^{\tau/2} ( \Phi(\partial\Omega\cap B_R)-\tilde \xi )$. 
By the boundary condition in \eqref{eq:weq}, we also have 
$\nabla w=(\nabla w\cdot \nu)\nu= (\partial_\nu w) \nu$. 
Thus,  
\begin{equation}\label{eq:wtauboun}
	w_\tau = -\frac{1}{2} (\partial_\nu w) (\nu\cdot \eta), \quad 
	\eta\in e^{\tau/2} ( \Phi(\partial\Omega\cap B_R)-\tilde \xi ). 
\end{equation}
Recall that $\psi(\eta,\tau)=0$ for 
$\eta\in e^{\tau/2} ( \Phi(\Omega\cap \partial B_R)-\tilde \xi )$
by the proof of \eqref{eq:wpsi0}. 
Then \eqref{eq:I1bou} follows. 
For later use, we note that, on the boundary,  
\begin{equation}\label{eq:wpro}
	\nabla' w=(\partial_\nu w)\nu', \quad 
	\partial_n w=(\partial_\nu w)\nu_n, 
\end{equation}
which follow from $\nabla w=(\partial_\nu w)\nu$.

We next consider $I_2$. Since 
\[
	-I_2 = 
	\int w \psi^2 \nabla'(\rho \partial_n w) \cdot \nabla' g 
	+2\int w(\partial_n w)\rho \psi \nabla'\psi\cdot \nabla'g 
	+\int w(\partial_n w)\rho \psi^2 \Delta'g, 
\]
we have 
\[
\begin{aligned}
	-\frac{d I_2}{d\tau} 
	&=
	\int w_\tau \psi^2  \nabla'(\rho\partial_n w)\cdot \nabla'g 
	+\int w_\tau (\partial_n w) \rho \psi^2 \Delta'g \\
	&\quad 
	+\int w \psi^2 \nabla'(\rho \partial_n w_\tau)\cdot \nabla'g 
	+\int w (\partial_n w_\tau) \rho \psi^2 \Delta'g \\
	&\quad 
	+\int w \psi^2 \nabla'(\rho_\tau \partial_n w)\cdot \nabla'g 
	+\int w (\partial_n w) \rho_\tau \psi^2 \Delta'g \\
	&\quad 
	+\int w \psi^2 \nabla'(\rho \partial_n w)\cdot \nabla'g_\tau 
	+\int w (\partial_n w) \rho \psi^2 \Delta'g_\tau \\
	&\quad 
	+2\int w \psi \psi_\tau \nabla'(\rho\partial_n w) \cdot \nabla'g
	+2\int ( w (\partial_n w) \rho \psi \nabla'\psi \cdot \nabla'g)_\tau \\
	&\quad 
	+2\int w (\partial_n w) \rho \psi \psi_\tau \Delta'g,
\end{aligned}
\]
where the 8th term in the right-hand side requires 
the 3rd derivative of $f\in C^{2+\alpha}_0 (\R^{n-1})$. 
But the computations here and below can be justified 
by the standard approximation procedure.
Again by integrating by parts twice, we can see that 
\begin{equation}\label{eq:inttwi}
\begin{aligned}
	&\int w \psi^2 \nabla'(\rho \partial_n w_\tau)\cdot \nabla'g 
	+\int w (\partial_n w_\tau) \rho \psi^2 \Delta'g \\
	&= 
	\int w_\tau \psi^2 \partial_n (\rho \nabla' w)\cdot \nabla'g 
	-\int_{\partial \Omega(\tau)}
	w_\tau \rho \psi^2 \nu_n \nabla'w\cdot \nabla'g dS \\
	&\quad 
	-2\int w(\partial_n w_\tau) \rho \psi \nabla'\psi\cdot \nabla'g
	+ 2 \int w_\tau \rho \psi (\partial_n \psi) \nabla'w\cdot \nabla'g. 
\end{aligned}
\end{equation}
Moreover, we have 
\[
\begin{aligned}
	&\int w \psi^2 \nabla'(\rho_\tau \partial_n w)\cdot \nabla'g 
	+\int w (\partial_n w) \rho_\tau \psi^2 \Delta'g \\
	&\quad 
	+\int w \psi^2 \nabla'(\rho \partial_n w)\cdot \nabla'g_\tau 
	+\int w (\partial_n w) \rho \psi^2 \Delta'g_\tau \\
	&= 
	-\int (\partial_n w) \rho_\tau \psi^2 \nabla' w \cdot \nabla'g 
	-2 \int w(\partial_n w)\rho_\tau \psi \nabla'\psi \cdot \nabla'g \\
	&\quad 
	 -\int (\partial_n w) \rho \psi^2 \nabla' w \cdot \nabla'g_\tau 
	-2 \int w(\partial_n w)\rho \psi \nabla'\psi \cdot \nabla'g_\tau. 
\end{aligned}
\]
These computations show that 
\begin{equation}\label{eq:I2esb}
\begin{aligned}
	-\frac{d I_2}{d\tau} &= 
	\int w_\tau \psi^2 \nabla'(\rho\partial_n w)\cdot \nabla'g  
	+ \int w_\tau \psi^2 \partial_n( \rho\nabla'w) \cdot \nabla'g  \\
	&\quad +\int w_\tau (\partial_n w) \rho \psi^2 \Delta'g 
	-\int_{\partial \Omega(\tau)}
	w_\tau \rho \psi^2 \nu_n \nabla'w\cdot \nabla'g dS \\
	&\quad 
	-\int (\partial_n w) \rho_\tau \psi^2 \nabla' w \cdot \nabla'g 
	+ \cR_2, 
\end{aligned}
\end{equation}
where 
\[
\begin{aligned}
	\cR_2&:= 
	- 2\int w(\partial_n w_\tau) \rho \psi \nabla'\psi\cdot \nabla'g
	+ 2 \int w_\tau \rho \psi (\partial_n \psi) \nabla'w\cdot \nabla'g \\
	&\quad 
	-2 \int w(\partial_n w)\rho_\tau \psi \nabla'\psi \cdot \nabla'g 
	 -\int (\partial_n w) \rho \psi^2 \nabla' w \cdot \nabla'g_\tau \\
	&\quad 
	-2 \int w(\partial_n w)\rho \psi \nabla'\psi \cdot \nabla'g_\tau 
	+2\int w \psi \psi_\tau \nabla'(\rho\partial_n w) \cdot \nabla'g \\
	&\quad 
	+2\int ( w (\partial_n w) \rho \psi \nabla'\psi \cdot \nabla'g)_\tau 
	+2\int w (\partial_n w) \rho \psi \psi_\tau \Delta'g. 
\end{aligned}
\]
For $\cR_2$, 
integrating by parts in the sixth term and 
expanding the seventh term, we obtain 
\begin{equation}\label{eq:R2cal}
\begin{aligned}
	\cR_2 &= 
	2 \int w_\tau \rho \psi (\partial_n \psi) \nabla'w\cdot \nabla'g 
	-\int (\partial_n w) \rho \psi^2 \nabla' w \cdot \nabla'g_\tau \\
	&\quad 
	- 2\int (\partial_n w) \rho \psi \psi_\tau \nabla' w \cdot \nabla'g 
	+2\int  w_\tau (\partial_n w) \rho \psi \nabla'\psi \cdot \nabla'g. 
\end{aligned}
\end{equation}
From the same computations as in the proof of \eqref{eq:I1bou}, 
it follows that 
\begin{equation}\label{eq:I2bou}
\begin{aligned}
	&-\int_{\partial \Omega(\tau)}
	 w_\tau \rho \psi^2 \nu_n \nabla'w\cdot \nabla'g dS  \\
	&= \frac{1}{2} 
	\int_{e^{\tau/2} ( \Phi(\partial\Omega\cap B_R)-\tilde \xi )}
	(\partial_\nu w)^2 \rho \psi^2 (\nu\cdot \eta) \nu_n \nabla'g \cdot \nu' dS. 
\end{aligned}
\end{equation}
Indeed, by \eqref{eq:wpro}, 
we have 
$\nabla'w \cdot \nabla'g =(\partial_\nu w)(\partial_{\nu'} g)$ 
on $e^{\tau/2} ( \Phi(\partial\Omega\cap B_R)-\tilde \xi )$. 
This together with \eqref{eq:wtauboun} shows the above relation.

We examine $I_3$. Again by integration by parts, we have 
\[
\begin{aligned}
	\frac{1}{2}\frac{d I_3}{d\tau}  
	&= 
	-\frac{d}{d\tau} \left( 
	\frac{1}{2} \int w \psi^2 \partial_n (\rho\partial_n w) |\nabla' g|^2 
	+\int w (\partial_n w) \rho \psi (\partial_n \psi) |\nabla' g|^2 
	\right) \\
	&= 
	-\frac{1}{2} \int w_\tau \psi^2 \partial_n(\rho \partial_n w) |\nabla' g|^2 
	-\frac{1}{2} \int w \psi^2 \partial_n(\rho \partial_n w_\tau) |\nabla' g|^2 \\
	&\quad 
	-\frac{1}{2} \int w \psi^2 \partial_n(\rho_\tau \partial_n w) |\nabla' g|^2 
	-\int w \psi^2 \partial_n(\rho \partial_n w) \nabla' g\cdot \nabla' g_\tau \\
	&\quad 
	-\int w \psi \psi_\tau  \partial_n(\rho\partial_n w)|\nabla' g|^2
	-\int (w (\partial_n w) \rho \psi (\partial_n \psi)  |\nabla' g|^2)_\tau. 
\end{aligned}
\]
In the same manner as in \eqref{eq:inttwi}, we see that 
\[
\begin{aligned}
	& -\frac{1}{2} \int w_\tau \psi^2 \partial_n(\rho \partial_n w) |\nabla' g|^2 
	-\frac{1}{2} \int w \psi^2 \partial_n(\rho \partial_n w_\tau) |\nabla' g|^2 \\
	&= 
	- \int w_\tau \psi^2 \partial_n(\rho \partial_n w) |\nabla' g|^2 
	+\frac{1}{2} \int_{\partial \Omega(\tau)} 
	 w_\tau(\partial_n w) \rho \psi^2 |\nabla' g|^2 \nu_n dS \\
	&\quad 
	- \int w_\tau (\partial_n w) \rho \psi (\partial_n \psi) |\nabla' g|^2 
	+ \int w(\partial_n w_\tau) \rho \psi (\partial_n \psi) |\nabla' g|^2 
\end{aligned}
\]
and 
\[
\begin{aligned}
	& -\frac{1}{2} \int w \psi^2 \partial_n(\rho_\tau \partial_n w) |\nabla' g|^2 
	-\int w \psi^2 \partial_n(\rho \partial_n w) \nabla' g\cdot \nabla' g_\tau \\
	& = 
	\frac{1}{2} \int (\partial_n w)^2 \rho_\tau \psi^2 |\nabla' g|^2 
	+\int (\partial_n w)^2 \rho \psi^2 \nabla' g\cdot \nabla' g_\tau  \\
	&\quad 
	+\int w(\partial_n w) \rho_\tau \psi (\partial_n \psi) |\nabla' g|^2 
	+2 \int w(\partial_n w) \rho \psi (\partial_n \psi ) 
	\nabla' g \cdot \nabla' g_\tau. 
\end{aligned}
\]
Then we have 
\begin{equation}\label{eq:I3esb}
\begin{aligned}
	\frac{1}{2}\frac{d I_3}{d\tau}  
	&=
	- \int w_\tau \psi^2 \partial_n(\rho \partial_n w) |\nabla' g|^2 
	+ \frac{1}{2} \int (\partial_n w)^2 \rho_\tau \psi^2 |\nabla' g|^2 \\
	&\quad 
	+\frac{1}{2} \int_{\partial \Omega(\tau)} 
	w_\tau(\partial_n w) \rho \psi^2 |\nabla' g|^2 \nu_n dS 
	+\cR_3, 
\end{aligned}
\end{equation}
where 
\[
\begin{aligned}
	\cR_3&:= 
	- \int w_\tau (\partial_n w) \rho \psi (\partial_n \psi) |\nabla' g|^2 
	+ \int w(\partial_n w_\tau) \rho \psi (\partial_n \psi) |\nabla' g|^2  \\
	&\quad 
	+\int (\partial_n w)^2 \rho \psi^2 \nabla' g\cdot \nabla' g_\tau  
	+\int w(\partial_n w) \rho_\tau \psi (\partial_n \psi) |\nabla' g|^2 \\
	&\quad 
	+2 \int w(\partial_n w) \rho \psi (\partial_n \psi ) 
	\nabla' g \cdot \nabla' g_\tau 
	-\int w \psi \psi_\tau  \partial_n(\rho\partial_n w)|\nabla' g|^2 \\
	&\quad 
	-\int (w (\partial_n w) \rho \psi (\partial_n \psi)  |\nabla' g|^2)_\tau. 
\end{aligned}
\]
For $\cR_3$, 
integrating by parts in the sixth term and 
expanding the seventh term, we obtain 
\begin{equation}\label{eq:R3cal}
\begin{aligned}
	\cR_3 &= 
	- 2 \int w_\tau (\partial_n w) \rho \psi (\partial_n \psi) |\nabla' g|^2  
	+\int (\partial_n w)^2 \rho \psi^2 \nabla' g\cdot \nabla' g_\tau \\
	&\quad 
	+ \int  (\partial_n w)^2 \rho \psi \psi_\tau  |\nabla' g|^2. 
\end{aligned}
\end{equation}
By \eqref{eq:wtauboun} and \eqref{eq:wpro}, we see that 
\begin{equation}\label{eq:I3bou}
\begin{aligned}
	&\frac{1}{2} \int_{\partial \Omega(\tau)} 
	w_\tau (\partial_n w) \rho \psi^2 |\nabla' g|^2 \nu_n dS\\
	&= 
	-\frac{1}{4} 
	\int_{e^{\tau/2} ( \Phi(\partial\Omega\cap B_R)-\tilde \xi )}
	(\partial_\nu w)^2 \rho \psi^2 (\nu\cdot \eta) 
	\nu_n^2 |\nabla'g|^2  dS. 
\end{aligned}
\end{equation}

By combining \eqref{eq:I0cal}, 
\eqref{eq:I123def}, \eqref{eq:I1esb}, \eqref{eq:I1bou}, 
\eqref{eq:I2esb}, \eqref{eq:I2bou}, 
\eqref{eq:I3esb} and \eqref{eq:I3bou}, and then by \eqref{eq:hatwc}, 
we obtain 
\begin{equation}\label{eq:cEderimid}
	\frac{d}{d\tau} \cE(\tau) 
	= 
	\cJ - \cB + \sum_{i=0}^3 \cR_i 
	+ \int \left( \frac{ |\widehat \nabla w|^2}{2} 
	- \frac{|w|^{p+1}}{p+1} + \frac{w^2}{2(p-1)}
	\right) \rho_\tau \psi^2, 
\end{equation}
where $\cB$ is given in the statement of this lemma and 
\[
\begin{aligned}
	\cJ&:= \int w_\tau \psi^2 
	\bigg( \frac{w\rho }{p-1}  
	- w |w|^{p-1} \rho
	-\nabla \cdot (\rho\nabla w) 
	+\nabla'(\rho\partial_n w)\cdot \nabla'g  \\
	&\qquad  
	+ \partial_n( \rho\nabla'w) \cdot \nabla'g
	+ \rho \partial_n w \Delta'g
	- \partial_n(\rho \partial_n w) |\nabla' g|^2
	\bigg). 
\end{aligned}
\]
From \eqref{eq:weq} and $\nabla' (\partial_n w ) = \partial_n \nabla'w $, 
it follows that 
\[
\begin{aligned}
	\cJ &= 
	\int w_\tau \psi^2 \bigg( 
	-w_\tau \rho -\frac{1}{2}\rho \eta\cdot \nabla w
	-\nabla \rho\cdot \nabla w 
	+\partial_n \rho \nabla' w\cdot \nabla' g  \\
	&\qquad 
	+\partial_n w \nabla' \rho \cdot \nabla' g
	- \partial_n w (\partial_n \rho )|\nabla'g|^2
	\bigg). 
\end{aligned}
\]
By using \eqref{eq:rhocalc}, we have 
\[
\begin{aligned}
	&-\frac{1}{2}\rho \eta\cdot \nabla w
	-\nabla \rho\cdot \nabla w 
	+\partial_n \rho \nabla' w\cdot \nabla' g 
	= \frac{1}{2} (\partial_n w) \rho (g(\eta',\tau) -g(0,\tau)), \\
	&\partial_n w \nabla' \rho \cdot \nabla' g
	- \partial_n w (\partial_n \rho )|\nabla'g|^2
	= -\frac{1}{2} (\partial_n w) \rho \eta' \cdot \nabla'g. 
\end{aligned}
\]
Thus, 
\[
	\cJ =
	-\int w_\tau^2 \rho \psi^2 
	+\frac{1}{2}\int w_\tau (\partial_n w ) \rho \psi^2 
	(  g(\eta',\tau)-g(0,\tau)-\eta'\cdot \nabla' g ). 
\]
Substituting this into \eqref{eq:cEderimid} and 
combining \eqref{eq:R0cal}, \eqref{eq:R1cal}, \eqref{eq:R2cal} 
and \eqref{eq:R3cal} 
yield the desired equality. 
The proof is complete. 
\end{proof}

We next estimate the terms $\cB$ and $\cR$ in \eqref{eq:derivE}. 

\begin{lemma}\label{lem:boundary}
$\cB(\tau) \geq0$. 
\end{lemma}

\begin{proof}
In \eqref{eq:R1def}, 
we have  $\cB=0$, 
since the domain of integration is far from the boundary. 
Thus, we consider the case \eqref{eq:R2def}. 
Since $\Phi(\partial \Omega \cap B_R)\subset \R^n_+$ and 
$\tilde \xi_n = \Phi_n(\tilde x)
= \tilde x_n-f(\tilde x') \geq 0$ 
for $\tilde x\in \overline{\Omega_{R/4}}$
by the choice of $f$, we see that $\nu\cdot \eta\geq 0$ 
for $\eta\in e^{\tau/2} ( \Phi(\partial\Omega\cap B_R)-\tilde \xi )$. 
In addition, since $|\nu'|\leq 1$, 
we have 
\[
\begin{aligned}
	&1-2\nu_n \partial_{\nu'}g +|\nabla'g|^2  \nu_n^2 \\
	&=
	( 1-(\partial_{\nu'}g) \nu_n )^2  
	+( |\nabla'g|^2 - (\partial_{\nu'}g)^2 ) \nu_n^2 
	\geq0. 
\end{aligned}
\]
Then the lemma follows. 
\end{proof}

\begin{lemma}\label{lem:ceRmid}
There exists $C>0$ such that 
\[
	\cR(\tau)\leq 
	\frac{1}{2} \int w_\tau^2  \psi^2 \rho
	+ C\tilde \cR(\tau) 
\]
for $\tau=-\log(\tilde t-t)$ with $-1/2<t<\tilde t\leq 0$, where 
\[
	\tilde \cR(\tau):= 
	\int (w^2+|\nabla w|^2+|w|^{p+1}) 
	e^{-\frac{|\eta|^2}{32}} e^{-\frac{\tau}{2}} 
	\chi_{[0,1]}\left( 
	\frac{4}{R} |\Psi(\tilde \xi+e^{-\frac{\tau}{2}}\eta)-\Psi(\tilde \xi)|  \right). 
\]
\end{lemma}

\begin{proof}
We only consider the case \eqref{eq:R2def}, 
since the case \eqref{eq:R1def} is simpler. 
By \eqref{eq:backg} and the choice of $f$, we have 
\begin{equation}\label{eq:gderes}
\begin{aligned}
	&
	\|\nabla' g(\cdot,\tau)\|_{L^\infty(\R^{n-1})} \leq \frac{1}{2}, 
	\quad 
	\|(\nabla')^2 g(\cdot,\tau)\|_{L^\infty(\R^{n-1})} \leq C e^{-\frac{\tau}{2}}, \\
	&\| \nabla' g_\tau(\cdot,\tau) \|_{L^\infty(\R^{n-1})} \leq 
	C |\eta'| e^{-\frac{\tau}{2}}, 
\end{aligned}
\end{equation}
where $C>0$ is independent of $\tau$. 
By Cauchy's inequality and \eqref{eq:gderes}, we obtain 
\[
\begin{aligned}
	\cR&\leq 
	\frac{1}{2} \int w_\tau^2 \rho \psi^2  
	+C \int (w^2+|\nabla w|^2+|w|^{p+1}) 
	( \rho ( |\nabla \psi|^2 +|\psi_\tau|) + |\rho_\tau|\psi^2)  \\
	&\quad + C \int |\nabla w|^2 \rho \psi^2 
	(|\eta'| e^{-\frac{\tau}{2}} + |\eta'|^4 e^{-\tau}). 
\end{aligned}
\]
From \eqref{eq:rhodef} and \eqref{eq:gderes}, it follows that 
\begin{equation}\label{eq:rhoR2dec}
	\rho(\eta,\tau) 
	\leq \exp\left( 
	-\frac{1}{4} 
	\left(  |\eta'|^2 
	+\frac{1}{2} \eta_n^2  
	 - (g(\eta',\tau)-g(0,\tau))^2 \right) 
	\right) 
	\leq e^{-\frac{|\eta|^2}{8}}. 
\end{equation}
By \eqref{eq:rhotau} and \eqref{eq:gderes}, 
and then by \eqref{eq:rhoR2dec}, we also have 
\begin{equation}\label{eq:rhotauee}
	|\rho_\tau| 
	\leq C |\eta'|^2 e^{-\frac{\tau}{2}} 
	(|\eta_n|+|\eta'|) \rho 
	\leq C e^{-\frac{|\eta|^2}{16}} e^{-\frac{\tau}{2}}. 
\end{equation}
These inequalities together with 
$\tau> -\log(\tilde t+1/2)> 0$ show that 
\begin{equation}\label{eq:cRcSmid}
\begin{aligned}
	\cR 
	&\leq
	\frac{1}{2} \int w_\tau^2 \rho \psi^2  \\
	&\quad 
	+ C \int (w^2+|\nabla w|^2+|w|^{p+1}) 
	( |\nabla \psi|^2 +|\psi_\tau| + \psi^2 e^{-\frac{\tau}{2}} )
	e^{-\frac{|\eta|^2}{16}}. 
\end{aligned}
\end{equation}

From \eqref{eq:psietadef} and \eqref{eq:gderes}, 
it follows that 
\begin{equation}\label{eq:naps}
\begin{aligned}
	|\nabla \psi|
	&= \frac{4e^{-\tau/2}}{R} 
	\left| \varphi'\left( 
	\frac{4}{R} e^{-\frac{\tau}{2}}
	|(\eta',\eta_n+g(\eta',\tau)-g(0,\tau))|   \right)
	\right| \\
	&\quad 
	\times \frac{
	|(\eta'+(\eta_n+g(\eta')-g(0))\nabla'g, \eta_n+g(\eta')-g(0))|
	}{|(\eta',\eta_n+g(\eta')-g(0))|} \\
	&\leq 
	Ce^{-\frac{\tau}{2}} 
	\chi_{(\frac{1}{2},1)}\left(
	\frac{4}{R} e^{-\frac{\tau}{2}}|(\eta',\eta_n+g(\eta',\tau)-g(0,\tau))|  \right)
\end{aligned}
\end{equation}
and 
\[
\begin{aligned}
	|\psi_\tau| 
	&= \frac{4e^{-\tau/2}}{R} 
	\left| \varphi'\left( 
	\frac{4}{R} e^{-\frac{\tau}{2}}
	|(\eta',\eta_n+g(\eta',\tau)-g(0,\tau))|   \right)
	\right| \\
	&\quad 
	\times \left| 
	-\frac{1}{2} |(\eta',\eta_n+g(\eta')-g(0))| 
	+\partial_\tau ( |(\eta',\eta_n+g(\eta')-g(0))| )
	\right| \\
	&\leq 
	Ce^{-\frac{\tau}{2}} |\varphi'| 
	\left( 
	\frac{1}{2} (|\eta'|^2 + 2\eta_n^2 + 2(g(\eta')-g(0))^2 ) 
	+ |\eta'|^2 e^{-\frac{\tau}{2}} 
	\right) \\
	&\leq 
	C|\eta|^2 e^{-\frac{\tau}{2}}
	\chi_{(\frac{1}{2},1)}\left( 
	\frac{4}{R} e^{-\frac{\tau}{2}}|(\eta',\eta_n+g(\eta',\tau)-g(0,\tau))|  \right), 
\end{aligned}
\]
where $g(\eta'):=g(\eta',\tau)$ and $g(0):=g(0,\tau)$. 
Then by $e^{-\tau/2}|(\eta',\eta_n+g(\eta',\tau)-g(0,\tau))|
= |\Psi(\tilde \xi+e^{-\tau/2}\eta)-\Psi(\tilde \xi)|$ in \eqref{eq:Psieta}, 
we have 
\begin{equation}\label{eq:nabtaupsi}
\begin{aligned}
	& ( |\nabla \psi|^2 +|\psi_\tau| + \psi^2 e^{-\frac{\tau}{2}} )
	e^{-\frac{|\eta|^2}{16}} \\
	&\leq 
	C e^{-\frac{|\eta|^2}{32}} e^{-\frac{\tau}{2}} 
	\chi_{[0,1]}\left( 
	\frac{4}{R} |\Psi(\tilde \xi+e^{-\tau/2}\eta)-\Psi(\tilde \xi)|  \right). 
\end{aligned}
\end{equation}
The lemma follows from \eqref{eq:nabtaupsi} and \eqref{eq:cRcSmid}. 
\end{proof}

We are now in a position to prove Proposition \ref{pro:31}. 

\begin{proof}[Proof of Proposition \ref{pro:31}]
By Lemmas \ref{lem:cederivR2}, \ref{lem:boundary}, \ref{lem:ceRmid}, we see that 
\[
	\cE(\tau) 
	+ \frac{1}{2} \int_{\tau'}^\tau \int_{\Omega(\sigma)} 
	w_\sigma^2 \rho \psi^2 d\eta d\sigma 
	\leq 
	\cE(\tau') + C\int_{\tau'}^\tau \tilde \cR(\sigma) d\sigma 
\]
for $\tau'=-\log(\tilde t-t')$ and $\tau=-\log(\tilde t-t)$ with 
$-1/2<t'<t<\tilde t\leq 0$. 
Note that this inequality holds for both \eqref{eq:R1def} and \eqref{eq:R2def}. 
The change of variables and the same computations as in Lemma \ref{lem:Ebdd} yield  
\begin{equation}\label{eq:bwucal}
\begin{aligned}
	\int_{\tau'}^\tau \tilde \cR(\sigma) d\sigma  
	&= 
	\int_{t'}^t (\tilde t-s)^{\frac{p+1}{p-1}-\frac{1}{2}}  
	\int_{\Omega_R} 
	\left( \frac{u^2}{\tilde t-s} + |\nabla u|^2 + |u|^{p+1} \right) \\
	&\qquad \times 
	(\tilde t-s)^{-\frac{n}{2}}  
	e^{-\frac{|x-\tilde x|^2}{32(\tilde t-s)}} dx ds \\
	&\leq 
	C(M+M^p)^2 \int_{t'}^t (\tilde t-s)^{-\frac{1}{2}} ds \\
	&\leq 
	C(M+M^p)^2 (\tilde t-t')^\frac{1}{2}. 
\end{aligned}
\end{equation}
Thus, 
\begin{equation}\label{eq:31conch}
	\cE(\tau) 
	+ \frac{1}{2} \int_{\tau'}^\tau \int_{\Omega(\sigma)} 
	w_\sigma^2 \rho \psi^2 d\eta d\sigma 
	\leq 
	\cE(\tau') + C(M+M^p)^2 (\tilde t-t')^\frac{1}{2}. 
\end{equation}

On the other hand, by 
\[
	w_{\tau}= (\tilde{t}-t)^\frac{1}{p-1}
	\left(-\frac{\hat u}{p-1}-\frac{(\xi-\tilde{\xi})\cdot \nabla \hat u}{2}
	+(\tilde t-t) \hat u_t 
	\right) 
\]
and the change of variables, we can see that 
\[
\begin{aligned}
	& \int_{\tau'}^\tau \int_{\Omega(\sigma)} 
	w_\sigma^2 \rho \psi^2 d\eta d\sigma \\
	&= 
	\int_{t'}^t 
	(\tilde{t}-s)^{\frac{2}{p-1} -\frac{n}{2}-1} 
	\int_{\Phi(\Omega_R)}
	\left| \frac{\hat u}{p-1}
	+\frac{(\xi-\Phi(\tilde x))\cdot\nabla \hat u}{2}
	- (\tilde t-s)\hat u_s \right|^2  \\
	&\qquad 
	\times 
	e^{-\frac{|\Psi(\xi)-\tilde x|^2}{4(\tilde t -s)}}
	\phi^2_{\tilde x,R/4}(\Psi(\xi)) d\xi ds. 
\end{aligned}
\]
This together with \eqref{eq:31conch} and $\cE(\tau)=E(t)$ deduces 
the desired inequality. The proof is complete. 
\end{proof}

%%%%%%%%%%%%%%%%%%%%%%%%%%%%%%%%%%%%%%%%%%%%%%%%
%%%%%%%%%%%%%%%%%%%%%%%%%%%%%%%%%%%%%%%%%%%%%%%%
%%%%%%%%%%%%%%%%%%%%%%%%%%%%%%%%%%%%%%%%%%%%%%%%
\section{$\eps$-regularity}\label{sec:epreg}
%%%%%%%%%%%%%%%%%%%%%%%%%%%%%%%%%%%%%%%%%%%%%%%%
%%%%%%%%%%%%%%%%%%%%%%%%%%%%%%%%%%%%%%%%%%%%%%%%
%%%%%%%%%%%%%%%%%%%%%%%%%%%%%%%%%%%%%%%%%%%%%%%%
The purpose of this section is to show the following 
$\eps$-regularity theorem:

\begin{theorem}\label{th:epsreg}
Let $n\geq3$, $p>p_S$ and 
$\Omega$ be any $C^{2+\alpha}$ domain 
in $\R^n$ with $0\in \overline{\Omega}$. 
Fix $0<R<1/2$ such that either \eqref{eq:R1def} or \eqref{eq:R2def} holds. 
Let $u$ be a solution of \eqref{eq:fujitaeq} satisfying \eqref{eq:Mdef2}. 
Then there exist constants 
$\eps_0$, $\delta_0$ and $\theta_0$ with 
$0<\eps_0, \theta_0<1$ and $0<\delta_0<R$ 
depending only on $n$, $p$, $M$, $\Omega$ and $R$
such that the following holds: 
If there exists $0<\delta\leq \delta_0$ such that 
\begin{equation}\label{eq:1scale}
	\delta^{\frac{4}{p-1}-n}
	\iint_{Q_\delta} (|\nabla u|^2 + |u|^{p+1}) dxdt 
	\leq \eps_0,
\end{equation}
then 
\begin{equation}\label{eq:epbdd}
	\|u\|_{L^\infty(Q_{\theta_0\delta})}\leq C(\theta_0 \delta)^{-\frac{2}{p-1}}. 
\end{equation}
Here $C$ is a positive constant 
depending only on $n$, $p$, $M$, $\Omega$ and $R$ 
and independent of $\eps_0$, $\delta$, $\delta_0$ and $\theta_0$. 
\end{theorem}

Among the contents of this section, 
we will cite only Theorem \ref{th:epsreg}, Remark \ref{rem:wepreg} 
and Lemma \ref{lem:32} 
for proving our main result in Section \ref{sec:Blim}.

\begin{remark}
Chou, Du and Zheng \cite[Theorems 2, 2']{CDZ07} 
proved $\eps$-regularity theorems 
for global-in-time solutions 
(or borderline solutions) of \eqref{eq:fujitaeq} 
in bounded convex domains. 
The theorems were applied for showing 
the decay of borderline solutions as $t\to\infty$, 
see \cite{So17} for alternative approach.
In \cite[Theorem 2]{CD10}, the assumption on convexity was removed, 
see also \cite[Proposition 4.2]{Du19}. 
By the nontrivial modifications of \cite{CD10,CDZ07}, 
we show the $\eps$-regularity for local-in-time solutions.

We note that the proofs of the $\eps$-regularity theorems 
in \cite{CD10,CDZ07} are based on 
a preliminary $\eps$-regularity result \cite[Lemma 3]{CDZ07}, 
where the time at which the regularity of the solution is concerned with 
should be contained in the interior 
of the time interval of the solution. 
Thus, it seems difficult to apply their argument 
near the final time of local-in-time solutions.  
To overcome this issue, 
we give a variant of \cite[Lemma 3]{CDZ07} 
with the aid of Blatt and Struwe \cite[Proposition 4.1]{BS15} 
and Giga and Kohn \cite[Theorems 2.1, 2.5]{GK89} to require 
only the estimate of solutions shortly before the reference time, 
see Lemma \ref{lem:epspre}. 
\end{remark}

\begin{remark}\label{rem:wepreg}
Theorem \ref{th:epsreg} remains to hold for weak solutions
satisfying an estimate of the form \eqref{eq:32} in Lemma \ref{lem:32}.
In particular, we may apply it to the blow-up limit 
of certain rescaled solutions in Section \ref{sec:Blim}.
\end{remark}

In this section, unless otherwise stated, $C$ denotes a constant 
depending only on $n$, $p$, $M$, $\Omega$ and $R$. 
Each $C$ may have different values also within the same line. 
We always assume either \eqref{eq:R1def} or \eqref{eq:R2def} 
and deal with each of the cases in a unified way. 
In addition, we always assume that 
$u$ is a solution of \eqref{eq:fujitaeq} satisfying \eqref{eq:Mdef2}.

We state a preliminary $\eps$-regularity result.

\begin{lemma}\label{lem:epspre}
There exist $\eps_1>0$ and $C>0$ depending only on $n$, $p$ and $\Omega$ 
such that 
the following holds: 
If there exists $0<R_1<R/4$ such that 
\begin{equation}\label{eq:asspre}
	(r/2)^{\frac{4}{p-1}-n} 
	\int_{t_1-(r/2)^2}^{t_1-(r/4)^2} 
	\int_{\Omega_{r/2}(x_1)} 
	|u(x,t)|^{p+1} dxdt \leq \eps_1
\end{equation}
for any $(x_1,t_1)$ and $r>0$ satisfying 
$Q_r(x_1,t_1)\subset Q_{2R_1}$, 
then $\|u\|_{L^\infty(Q_{R_1/4})}\leq CR_1^{-2/(p-1)}$. 
\end{lemma}

\begin{proof}
Let $\eps_1>0$ be a constant chosen later. 
Throughout this proof, $C$ depends only on $n$, $p$ and $\Omega$. 
Set $v(y,s):= R_1^{2/(p-1)} u(R_1 y, R_1^2 s)$. 
Since $R/R_1>4$ and $-1/R_1^2<-16$, we can check that $v$ satisfies 
$v_t=\Delta v+|v|^{p-1}v$ in $Q_4'$, where 
$\Omega':=R_1^{-1}\Omega$, 
$Q_r'(x_1,t_1):=(\Omega'\cap B_r(x_1))\times(t_1-r^2,t_1)$ and $Q_r':=Q_r'(0,0)$. 
From \eqref{eq:asspre} and the change of variables, it follows that  
if $Q_r'(x_1,t_1)\subset Q_2'$ and $(r/4)^2\leq -t_1$, then 
$Q_r'(x_1,t_1+(r/4)^2)\subset Q_2'$ and 
\begin{equation}\label{eq:r24p1}
\begin{aligned}
	&(r/2)^{\frac{4}{p-1}-n} \int_{t_1-(r/4)^2}^{t_1} 
	\int_{\Omega'\cap B_{r/2}(x_1)} |v(y,s)|^{p+1} dyds \\
	&\leq 
	(r/2)^{\frac{4}{p-1}-n} \int_{t_1+(r/4)^2-(r/2)^2}^{t_1+(r/4)^2-(r/4)^2} 
	\int_{\Omega'\cap B_{r/2}(x_1)} |v(y,s)|^{p+1} dyds \leq \eps_1. 
\end{aligned}
\end{equation}

Let $(\tilde x,\tilde t)\in Q'_{1/2}$. 
Set $\lambda:=(-\tilde{t})^{1/2}$ and
$\Omega'':=\lambda^{-1}(\Omega'-\tilde x)$. 
Then the rescaled function $\tilde v(x,t)
:=\lambda^{2/(p-1)} v(\lambda x+\tilde x, \lambda^2 t+\tilde{t})$ satisfies 
$\tilde v_t=\Delta \tilde v+|\tilde v|^{p-1} \tilde v$ 
in $(\Omega''\cap B_{2/\lambda}( -\tilde x/\lambda))
\times ( -4/\lambda^2, 0)$. 
If $B_{r/\lambda}( (x_1-\tilde x)/\lambda)\subset B_2$ 
and $((t_1-r^2)/\lambda^2, t_1/\lambda^2)\subset ( -4, 0)$, then 
$Q_r'(x_1, t_1 + \tilde t)\subset Q_2'$ and 
$(r/4)^2\leq -(t_1 + \tilde t)$. 
Therefore, \eqref{eq:r24p1} shows that 
\[
\begin{aligned}
	&(r/4\lambda)^{\frac{4}{p-1}-n}
	\int_{(t_1-(r/4)^2)/\lambda^2}^{t_1/\lambda^2} 
	\int_{ \Omega''\cap B_{r/4\lambda}((x_1-\tilde x)/\lambda) }
	|\tilde{v}(x,t)|^{p+1} dxdt \\
	&=(r/4)^{\frac{4}{p-1}-n} 
	\int_{t_1+\tilde t-(r/4)^2}^{t_1+\tilde t}
	\int_{\Omega' \cap B_{r/4}(x_1)}
	|v(y, s)|^{p+1} dyds 
	\leq C \eps_1. 
\end{aligned}
\]
Replacing $(x_1,t_1)$ and $r$ with $(\lambda x_1+\tilde x,\lambda^2 t_1)$ and $\lambda r$, 
respectively, we see that 
\[
	(r/4)^{\frac{4}{p-1}-n}
	\iint_{Q_{r/4}''(x_1,t_1)} |\tilde{v}(x,t)|^{p+1} dxdt 
	\leq C \eps_1
\]
for any $Q_r''(x_1,t_1) \subset Q_2''$, where 
$Q_r''(x_1,t_1):= (\Omega''\cap B_r(x_1))\times (t_1-r^2,t_1)$ 
and $Q_r'':=Q_r''(0,0)$. 
Hence $\|\tilde{v}\|_{M^{p+1,\mu_c}(Q_{1/2}'')} \leq C \eps_1$ 
with $\mu_c:=2(p+1)/(p-1)$, 
where $\|\cdot\|_{M^{p+1,\mu_c}(Q_{1/2}'')}$ 
is the parabolic Morrey norm on $Q_{1/2}''$.

Taking $\eps_1$ sufficiently small,
we may apply \cite[Proposition 4.1]{BS15} to see that $\|\tilde v\|_{L^\infty(Q_{1/4}'')} 
\le C \|\tilde v \|_{M^{p+1,\mu_c}(Q_{1/2}'')} \leq C\eps_1$. 
In particular, 
$\lambda^{2/(p-1)}|v(\lambda x+\tilde x, \lambda^2 t+\tilde{t})| 
\leq C\eps_1$ for $(x,t)\in Q_{1/4}''$. 
Letting $(x,t)\to (0,0)$ gives 
\[
	|v(\tilde x, \tilde t )| 
	\leq C\eps_1 \lambda^{-\frac{2}{p-1}} 
	= C\eps_1 (-\tilde t)^{-\frac{1}{p-1}} 
\]
for $(\tilde x,\tilde t)\in Q'_{1/2}$. 
Since $\Omega'$ is $C^{2+\alpha}$, by
applying \cite[Theorems 2.1, 2.5]{GK89} with 
$\eps_1$ replaced by a smaller constant if necessary, we obtain  
$\|v\|_{L^\infty(Q_{1/4}')} \leq C$, and hence 
$|u(x, t)| \leq C R_1^{-2/(p-1)}$ for $(x,t)\in Q_{R_1/4}$. 
The proof is complete. 
\end{proof}

In the rest of this section, we prove Theorem \ref{th:epsreg} 
by using Lemma \ref{lem:epspre} 
and estimates of $E$. 
First of all, note that we may replace $\delta$ with $A\delta$ 
in the assumption \eqref{eq:1scale}, 
where $A>1$ is a large constant and 
$0<\delta<1$ is a small constant depending on $A$. 
More precisely, we may assume 
\begin{equation}\label{eq:asKdel}
	\delta^{\frac{4}{p-1}-n}
	\iint_{Q_{A\delta}} (|\nabla u|^2 + |u|^{p+1}) dxdt 
	\leq A^{n-\frac{4}{p-1}} \eps_0. 
\end{equation}
Here we  take the constants $A$, $\eps_0$, $\delta$ and $\theta_0$ as 
\begin{equation}\label{eq:Kdelthsmall}
	A>3, \quad 0<\eps_0<1, \quad 
	0<\delta<\frac{R}{16A}<\frac{1}{16}, 
	\quad 0<\theta_0<\frac{1}{32A},
\end{equation}
which will be specified later, 
see \eqref{eq:choiceKedelt}. 
Set 
\[
	I_r=I_r(x_1,t_1):=
	(r/2)^{\frac{4}{p-1}-n}
	\int_{t_1-(r/2)^2}^{t_1-(r/4)^2} \int_{\Omega_{r/2}(x_1)}  
	|u|^{p+1} dxdt.
\]
In order to show Theorem  \ref{th:epsreg}, 
it suffices to check the following statement: 
\begin{equation}\label{eq:x1t1e1}
	I_r(x_1,t_1) \leq \eps_1
	\quad 
	\mbox{ for any }
	Q_r(x_1,t_1)\subset Q_{8A\theta_0 \delta}, 
\end{equation}
where $\eps_1$ is given in Lemma \ref{lem:epspre}. 
Note that 
\[
	0<r<8A\theta_0\delta<\frac{1}{4}\delta<\delta <
	\frac{R}{16}<\frac{1}{16}. 
\]
Indeed, 
once \eqref{eq:x1t1e1} is proved, 
Lemma \ref{lem:epspre} guarantees 
the desired $L^\infty$ bound of $u$: 
$\|u\|_{L^\infty(Q_{A\theta_0\delta})}\leq C(A\theta_0 \delta)^{-2/(p-1)}$. 
Therefore our temporary task is to estimate $I_r$. 

\begin{proposition}\label{pro:Irhtilh}
There exists a constant $C>0$ 
depending only on $n$, $p$, $M$, $\Omega$ and $R$ 
such that 
\[
\begin{aligned}
	I_r(x_1,t_1) 
	&\leq 
	C \overline{h} \left( 
	e^{-A^2/C} 
	+A^{n-\frac{4}{p-1}} \eps_0
	+ \delta  \right) \\
\end{aligned}
\]
for any $Q_r(x_1,t_1)\subset Q_{8A\theta_0 \delta}$, 
where $\overline{h}(s):=s+s^{1/(p+1)}$ for $s\geq0$. 
\end{proposition}

We prove this proposition by means of several lemmas.

\begin{lemma}
There exists $C>0$ such that  
\begin{equation}\label{eq:Ires1}
	I_r \leq 
	C \int_{t_1-(r/2)^2}^{ t_1-(r/4)^2} 
	(t_1-s)^\frac{2}{p-1} 
	\int_{\Omega_R}
	|u|^{p+1} K_{(x_1,t_1)}(x,s) \phi^2_{x_1,R/8} dxds 
\end{equation}
for any $Q_r(x_1,t_1)\subset Q_{8A\theta_0 \delta}$. 
\end{lemma}

\begin{proof}
For $(x,s)\in 
B_{r/2}(x_1) \times (t_1-(r/2)^2, t_1-(r/4)^2)$, we have 
$|x-x_1|\leq r/2<R/32$ and 
$t_1-(r^2/4)<s<t_1-(r^2/16)$. Thus, 
\[
\begin{aligned}
	K_{(x_1, t_1)}(x,s) \phi^2_{x_1,R/8}(x)
	&= ( t_1-s)^{-\frac{n}{2}} 
	e^{-\frac{|x-x_1|^2}{4( t_1-s)}} 
	\varphi^2 \left( \frac{8|x-x_1|}{R} \right) \\
	&\geq 
	(1/4)^{-\frac{n}{2}} r^{-n} 
	e^{-1} 
	\varphi^2 (1/4)
	\geq 
	C r^{-n}. 
\end{aligned}
\]
By $(t_1-s)^{2/(p-1)} \geq  16^{-2/(p-1)} r^{4/(p-1)}$ 
for $t_1-r^2/4< s<t_1-(r^2/16)$, the lemma follows. 
\end{proof}

To estimate the right-hand side of \eqref{eq:Ires1}, 
we prepare the following lemma 
by using Proposition \ref{pro:31}.

\begin{lemma}\label{lem:32}
There exists $C>0$ such that 
\begin{equation}\label{eq:32}
\begin{aligned}
	&\int_{t'}^t (\tilde t-s)^\frac{2}{p-1} 
	\int_{\Omega_R} |u(x,s)|^{p+1} 
	K_{(\tilde x, \tilde t)}(x,s) \phi_{\tilde x,\frac{R}{4}}^2(x) dxds  \\
	&\leq 
	C \left( \log\frac{\tilde t-t'}{\tilde t-t} \right)^\frac{1}{2}
	\left(  E_{(\tilde x,\tilde t)}(t';\phi_{\tilde x,\frac{R}{4}}) 
	- E_{(\tilde x,\tilde t)}(t;\phi_{\tilde x,\frac{R}{4}}) 
	+ C(\tilde t-t')^\frac{1}{2} \right)^\frac{1}{2} \\
	&\quad 
	+ C_p \left( 
	E_{(\tilde x,\tilde t)}(t';\phi_{\tilde x,\frac{R}{4}}) 
	+ C (\tilde t-t')^\frac{1}{2} \right) 
	\log\frac{\tilde t-t'}{\tilde t-t}  
	+ C (\tilde t-t')^\frac{1}{2}
\end{aligned}
\end{equation}
for any $\tilde x\in \overline{\Omega_{R/4}}$ and 
$-1/2<t'<t<\tilde t\leq 0$, 
where $C_p:=2(p+1)/(p-1)$. 
\end{lemma}

\begin{proof}
Define $w$, $\cE$, $\rho$ and $\psi$ 
by \eqref{eq:backw}, \eqref{eq:cEdef}, 
\eqref{eq:rhodef} and \eqref{eq:psietadef}, respectively. 
Then the left-hand side of the desired inequality equals 
$\iint |w|^{p+1} \rho \psi^2 $, where 
we write $\int(\cdots)=\int_{\Omega(\tau)}(\cdots) d\eta$ 
and $\iint (\cdots)
=\int_{\tau'}^\tau \int_{\Omega(\sigma)} (\cdots) d\eta d\sigma$
unless otherwise stated. 
From \eqref{eq:weq} and \eqref{eq:cEdef}, it follows that 
\[
\begin{aligned}
	&\frac{1}{2}\frac{d}{d\tau} \int w^2 \rho \psi^2 
	=
	\int w w_\tau \rho \psi^2 
	+\frac{1}{2} \int w^2 \rho_\tau \psi^2 
	+\int w^2 \rho \psi \psi_\tau 
	-2\cE(\tau) + 2\cE(\tau) \\
	&=-2\cE(\tau) +\frac{p-1}{p+1} \int |w|^{p+1} \rho \psi^2 
	-\frac{1}{2} \int w \rho \psi^2 \nabla w \cdot \eta 
	+\int w \rho \psi^2 \Delta w \\
	&\quad 
	-2\int w \rho \psi^2 \nabla'(\partial_n w)\cdot \nabla'g 
	+\int w (\partial_n^2 w) \rho  \psi^2 |\nabla' g|^2 \\
	&\quad 
	-\int w (\partial_n w) \rho \psi^2 \Delta' g 
	+\int |\nabla w|^2 \rho \psi^2 
	-2 \int (\partial_n w) \rho \psi^2 \nabla'w\cdot \nabla'g \\
	&\quad 
	+\int (\partial_n w)^2 \rho \psi^2  |\nabla'g|^2 
	+\frac{1}{2} \int w^2 \rho_\tau \psi^2
	+\int w^2 \rho \psi\psi_\tau. 
\end{aligned}
\]

Integrating by parts gives 
\[
	\frac{1}{2}\frac{d}{d\tau} \int w^2 \rho \psi^2 
	=-2\cE(\tau) +\frac{p-1}{p+1} \int |w|^{p+1} \rho \psi^2 
	+\tilde \cR_1+\tilde \cR_2, 
\]
where 
\[
\begin{aligned}
	&\begin{aligned}
	\tilde \cR_1&:= 
	-\frac{1}{2}\int w \rho \psi^2 \nabla w \cdot \eta
	-\int w \psi^2 \nabla w \cdot \nabla \rho 
	+2\int w (\partial_n w) \psi^2 \nabla'\rho \cdot \nabla' g \\
	&\quad 
	+\int w (\partial_n w) \rho \psi^2 \Delta' g
	-\int w(\partial_n w) (\partial_n \rho) \psi^2  |\nabla'g|^2
	+\frac{1}{2} \int w^2 \rho_\tau \psi^2, 
	\end{aligned} \\
	&\begin{aligned}
	\tilde \cR_2 &:= 
	-2\int w \rho \psi \nabla w \cdot \nabla \psi 
	+4 \int w (\partial_n w) \rho \psi \nabla'\psi \cdot \nabla'g \\
	&\quad 
	-2\int w (\partial_n w) \rho  \psi (\partial_n \psi)  |\nabla' g|^2 
	+\int w^2 \rho \psi \psi_\tau. 
	\end{aligned}
\end{aligned}
\]

Since $w\nabla w=\nabla (w^2)/2$ and $w\partial_n w=\partial_n(w^2)/2$, 
integrating by parts again shows that 
\[
\begin{aligned}
	\tilde \cR_1 &=
	\int w^2 \psi^2 \bigg( \frac{\eta}{4}\cdot \nabla\rho 
	+\frac{n}{4} \rho + \frac{1}{2} \Delta \rho 
	-\partial_n(\nabla' \rho)\cdot \nabla'g \\ 
	&\qquad 
	-\frac{1}{2} (\partial_n \rho) \Delta'g 
	+\frac{1}{2} (\partial_n^2 \rho) |\nabla' g|^2 \bigg) 
	+\frac{1}{2} \int w^2 \rho \psi \nabla \psi \cdot\eta \\
	&\quad 
	+\int w^2 \psi \nabla \psi \cdot \nabla \rho 
	-2 \int w^2 \psi (\partial_n \psi) \nabla' \rho \cdot \nabla'g 
	-\int w^2 \rho \psi (\partial_n \psi)  \Delta'g \\
	&\quad 
	+\int w^2 (\partial_n \rho) \psi(\partial_n \psi ) |\nabla'g|^2 
	+\frac{1}{2} \int w^2 \rho_\tau \psi^2. 
\end{aligned}
\]
From \eqref{eq:rhocalc}, \eqref{eq:gderes}, \eqref{eq:rhoR2dec} 
and direct computations, it follows that 
\[
\begin{aligned}
	& \frac{\eta}{4}\cdot \nabla\rho 
	+\frac{n}{4} \rho + \frac{1}{2} \Delta \rho 
	-\partial_n(\nabla' \rho)\cdot \nabla'g 
	-\frac{1}{2} (\partial_n \rho) \Delta'g 
	+\frac{1}{2} (\partial_n^2 \rho) |\nabla' g|^2 \\
	&=\frac{1}{8} (\eta_n + g(\eta')-g(0))
	(g(\eta')-g(0)-\eta'\cdot \nabla' g) \rho \\
	&\quad 
	+\frac{1}{8} (\eta_n+g(\eta')-g(0))^2 (1-2|\nabla'g|^2) \rho \\
	&\geq 
	- \frac{1}{8} |\eta_n + g(\eta')-g(0)|
	|g(\eta')-g(0)-\eta'\cdot \nabla' g| \rho \\
	&\geq 
	- C |\eta|^3 e^{-\frac{\tau}{2}} e^{-\frac{|\eta|^2}{8}}
	\geq 
	- C e^{-\frac{\tau}{2}} e^{-\frac{|\eta|^2}{16}}, 
\end{aligned}
\]
where $g(\eta'):=g(\eta',\tau)$ and $g(0):=g(0,\tau)$. 
The remainder terms in $\tilde \cR_1$ can be estimated 
by using \eqref{eq:rhocalc}, \eqref{eq:gderes}, \eqref{eq:rhoR2dec} 
and \eqref{eq:rhotauee}. 
Then by $\tau> -\log(\tilde t+1/2)> 0$, we obtain 
\[
\begin{aligned}
	\tilde \cR_1 & \geq 
	-C \int w^2 \psi^2 
	e^{-\frac{\tau}{2}} e^{-\frac{|\eta|^2}{16}}
	-C \int w^2 |\nabla \psi| e^{-\frac{|\eta|^2}{16}} (1+e^{-\frac{\tau}{2}}) \\
	&\geq 
	-C \int w^2 ( |\nabla \psi| + \psi^2 e^{-\frac{\tau}{2}}) 
	e^{-\frac{|\eta|^2}{16}}. 
\end{aligned}
\]
On the other hand, by Cauchy's inequality, \eqref{eq:gderes} 
and \eqref{eq:rhoR2dec}, we see that 
\[
	\tilde \cR_2  \geq 
	-C \int (w^2+|\nabla w|^2) (|\nabla \psi|+|\psi_\tau|) 
	e^{-\frac{|\eta|^2}{8}}, 
\]
and so by \eqref{eq:nabtaupsi} with \eqref{eq:naps}, 
we obtain 
\[
	\tilde \cR_1 + \tilde \cR_2
	\geq 
	-C \int (w^2+|\nabla w|^2) 
	( |\nabla \psi| + |\psi_\tau|+ \psi^2 e^{-\frac{\tau}{2}}) 
	e^{-\frac{|\eta|^2}{16}} 
	\geq 
	-C \tilde \cR, 
\]
where $\tilde \cR=\tilde \cR(\tau)$ 
is given in Lemma \ref{lem:ceRmid}. 

By $\cE(\tau)=E(t)$ and  Proposition \ref{pro:31}, 
we obtain 
\[
\begin{aligned}
	\frac{1}{2}\frac{d}{d\tau} \int w^2 \rho \psi^2 
	&\geq -2\cE(\tau) +\frac{p-1}{p+1} \int |w|^{p+1} \rho \psi^2 
	-C\tilde \cR \\
	&\geq 
	-2(E(t')+C (\tilde t-t')^\frac{1}{2})
	+\frac{p-1}{p+1} \int |w|^{p+1} \rho \psi^2 
	-C\tilde \cR. 
\end{aligned}
\]
Set $C_p:=2(p+1)/(p-1)$. Then, 
\[
\begin{aligned}
	\int |w|^{p+1} \rho \psi^2 
	\leq 
	\frac{C_p}{4}\frac{d}{d\tau} \int w^2 \rho \psi^2 
	+ C_p (E(t')+C(\tilde t-t')^\frac{1}{2})
	+C\tilde \cR. 
\end{aligned}
\]
Integrating this inequality over $\sigma\in (\tau',\tau)$ 
with $\tau'=-\log(\tilde t-t')$ and $\tau=-\log(\tilde t-t)$,
we have  
\begin{equation}\label{eq:wpkappa}
\begin{aligned}
	\iint |w|^{p+1} \rho \psi^2   
	&\leq 
	\frac{C_p}{4} (\cK(\tau)-\cK(\tau')) \\
	&\quad 
	+ C_p (E(t')+C (\tilde t-t')^\frac{1}{2}) 
	\log\frac{\tilde t-t'}{\tilde t-t} 
	+C\int_{\tau'}^\tau \tilde \cR d\sigma, 
\end{aligned}
\end{equation}
where 
\[
	\cK(\tau):=\int w^2(\eta,\tau) \rho(\eta,\tau) \psi^2(\eta,\tau) d\eta. 
\]

We estimate $|\cK(\tau)-\cK(\tau')|$. 
By \eqref{eq:rhoR2dec} and \eqref{eq:rhotauee}, we have 
\[
\begin{aligned}
	|\cK(\tau)-\cK(\tau')| 
	&=\left| \int_{\tau'}^\tau \frac{d \cK}{d\sigma} d\sigma \right| 
	=\left| \iint 
	( 2ww_\sigma \rho \psi^2  
	+ w^2 \rho_\sigma \psi^2
	+ 2w^2 \rho \psi \psi_\sigma)  \right|  \\
	&\leq 
	2 \iint |w||w_\sigma| \rho \psi^2  
	+C\int_{\tau'}^\tau \tilde \cR(\sigma)d\sigma. 
\end{aligned}
\]
The H\"older inequality and \eqref{eq:31conch} with $\cE(\tau)=E(t)$ yield 
\[
\begin{aligned}
	\iint |w||w_\sigma| \rho \psi^2 
	&\leq
	\left( \iint w^2 \rho \psi^2  \right)^\frac{1}{2} 
	\left( \iint w_\sigma^2 \rho \psi^2  \right)^\frac{1}{2} \\
	&\leq 
	\sqrt{2} \left( \iint w^2 \rho \psi^2 \right)^\frac{1}{2} 
	( E(t') - E(t) + C (\tilde t-t')^\frac{1}{2} )^\frac{1}{2}. 
\end{aligned}
\]
Computations similar to \eqref{eq:bwucal} give  
\[
\begin{aligned}
	\iint w^2 \rho \psi^2  
	&\leq 
	\int_{t'}^t (\tilde t-s)^{\frac{2}{p-1}-1} 
	\int_{\Omega_R}  |u|^2 
	(\tilde t-s)^{-\frac{n}{2}}
	e^{-\frac{|x-\tilde x|^2}{8(\tilde t-s)}} \phi_{\tilde x,R/4}^2 
	dxds \\
	&\leq 
	C \int_{t'}^t (\tilde t-s)^{-1} ds
	\leq 
	C \log\frac{\tilde t-t'}{\tilde t-t}. 
\end{aligned} 
\]
Thus, 
\[
	|\cK(\tau)-\cK(\tau')| \leq 
	C \left( \log\frac{\tilde t-t'}{\tilde t-t} \right)^\frac{1}{2}
	(  E(t') - E(t) + C (\tilde t-t')^\frac{1}{2} )^\frac{1}{2}
	+ C\int_{\tau'}^\tau \tilde \cR(\sigma)d\sigma. 
\]

The above estimates show that 
\[
\begin{aligned}
	\iint |w|^{p+1} \rho \psi^2   
	&\leq 
	C \left( \log\frac{\tilde t-t'}{\tilde t-t} \right)^\frac{1}{2}
	(  E(t') - E(t) + C (\tilde t-t')^\frac{1}{2} )^\frac{1}{2} \\
	&\quad 
	+ C_p (E(t')+C (\tilde t-t')^\frac{1}{2}) 
	\log\frac{\tilde t-t'}{\tilde t-t} 
	+C\int_{\tau'}^\tau \tilde \cR (\sigma)d\sigma, 
\end{aligned}
\]
From \eqref{eq:bwucal}, 
it follows that 
$\iint |w|^{p+1} \rho \psi^2 d\eta d\sigma$
is bounded by the right-hand side of the desired inequality. 
Then the lemma follows. 
\end{proof}

We estimate the right-hand side of \eqref{eq:Ires1}. 
Assume $Q_r(x_1,t_1)\subset Q_{8A\theta_0 \delta}$. 
Then by $x_1\in \overline{\Omega_{R/4}}$ and 
$-1/2<t_1-(r^2/4)<t_1-(r^2/16)<t_1\leq 0$, 
we see that Lemma \ref{lem:32} 
(with $\phi_{x_1,R/4}$ replaced by $\phi_{x_1,R/8}$) 
gives 
\begin{equation}\label{eq:Irlemg}
\begin{aligned}
	I_r
	&\leq 
	C 
	\left(  E_{(x_1,t_1)} \left( t_1-\frac{r^2}{4} \right) 
	- E_{(x_1,t_1)} \left( t_1-\frac{r^2}{16} \right)  
	+ C r \right)^\frac{1}{2} \\
	&\quad 
	+ C_p \left( 
	E_{(x_1,t_1)}\left(t_1-\frac{r^2}{4};\phi_{x_1,R/8} \right) 
	+ C r \right) \log 4 
	+ C r 
\end{aligned}
\end{equation}
for $Q_r(x_1,t_1)\subset Q_{8A\theta_0 \delta}$. 
We derive an upper bound of $E_{(x_1,t_1)}$
and a lower bound of $E_{(x_1,t_1)}$.
 
\begin{lemma}\label{lem:upperpre}
There exists $C>0$ independent of $t_1$ such that 
\[
	E_{(x_1,t_1)}\left( 
	t_1-\frac{r^2}{4};\phi_{x_1,R/8} \right) 
	\leq 
	E_{(x_1,t_1)}\left( -4\delta^2;\phi_{x_1,R/8} \right) 
	+ C \delta 
\]
for $Q_r(x_1,t_1)\subset Q_{8A\theta_0 \delta}$. 
\end{lemma}

\begin{proof}
By $x_1\in \overline{\Omega_{R/4}}$, 
$-1/2<-4\delta^2<t_1-r^2/4<t_1\leq 0$ and 
Proposition \ref{pro:31} 
(with $\phi_{x_1,R/4}$ replaced by $\phi_{x_1,R/8}$), 
we see that 
\[
	E_{(x_1,t_1)}\left( 
	t_1-\frac{r^2}{4};\phi_{x_1,R/8} \right) 
	\leq 
	E_{(x_1,t_1)}( -4\delta^2;\phi_{x_1,R/8} ) 
	+ C (t_1 +4\delta^2)^\frac{1}{2}. 
\]
This implies the desired inequality. 
\end{proof}

\begin{lemma}\label{lem:lowerpre}
There exists $C>0$ independent of $t_1$ such that 
\[
	E_{(x_1,t_1)}\left(t_1-\frac{r^2}{16};
	\phi_{x_1,R/8} \right)
	\geq 
	-C \delta 
\]
for $Q_r(x_1,t_1)\subset Q_{8A\theta_0 \delta}$. 
\end{lemma}

\begin{proof}
By \eqref{eq:wpkappa} and \eqref{eq:bwucal}, 
there exists a constant $C'>0$ such that 
\begin{equation}\label{eq:taukapp}
\begin{aligned}
	&\int_{\tau'}^\tau \int_{\Omega(\sigma)} 
	|w|^{p+1} \rho \psi^2  d\eta d\sigma \\
	&\leq 
	\frac{C_p}{4} \cK(\tau) 
	+ C_p (E_{(\tilde x,\tilde t)}(t')+C' (\tilde t-t')^\frac{1}{2})  
	\log\frac{\tilde t-t'}{\tilde t-t} 
	+C' (\tilde t-t')^\frac{1}{2}
\end{aligned}
\end{equation}
for any $\tilde x\in \overline{\Omega_{R/4}}$, 
$\tau'=-\log(\tilde t-t')$ and $\tau=-\log(\tilde t-t)$ with 
$-1/2<t'<t<\tilde t\leq 0$. 
We claim that 
\[
	E_{(\tilde x,\tilde t)}(t')+C' (\tilde t-t')^\frac{1}{2} \geq 0
\]
for any $\tilde x\in \overline{\Omega_{R/4}}$ and $-1/2<t'<\tilde t\leq 0$. 
To obtain a contradiction, we suppose that 
there exist 
$\tilde x_0\in \overline{\Omega_{R/4}}$ and $-1/2<t_0'<\tilde t_0\leq 0$ such that 
$E_{(\tilde x_0,\tilde t_0)}(t_0')+C' (\tilde t_0-t_0')^{1/2}< 0$. 
Then by \eqref{eq:taukapp}, we have 
\[
\begin{aligned}
	&\int_{\tau_0'}^\tau \int_{\Omega(\sigma)} 
	|w|^{p+1} \rho \psi^2  d\eta d\sigma \\
	&\leq 
	\frac{C_p}{4} \cK(\tau) 
	+ C_p (E_{(\tilde x_0,\tilde t_0)}(t_0')+C' (\tilde t_0-t_0')^\frac{1}{2})  
	\log\frac{\tilde t_0-t_0'}{\tilde t_0-t} 
	+C' (\tilde t_0-t_0')^\frac{1}{2}
\end{aligned}
\]
for any $\tau=-\log(\tilde t_0-t)$ with $t_0'<t<\tilde t_0$, 
where $\tau_0':=-\log(\tilde t_0-t_0')$. 
Therefore, there exists $t_0'<t_*<\tilde t_0$ such that 
\[
	\int_{\tau_0'}^\tau \int_{\Omega(\sigma)} 
	|w|^{p+1} \rho \psi^2  d\eta d\sigma 
	\leq \frac{C_p}{4} \int_{\Omega(\tau)}  w^2 \rho \psi^2 d\eta - 1
\]
for any $\tau=-\log(\tilde t_0-t)$ with $t_*<t<\tilde t_0$. 
From the H\"older inequality and \eqref{eq:rhoR2dec}, 
it follows that 
\[
	\int_{\tau_0'}^\tau 
	\left( \int_{\Omega(\sigma)} w^2 \rho \psi^2 d\eta \right)^\frac{p+1}{2} d\sigma
	\leq 
	C \int_{\Omega(\tau)}  w^2 \rho \psi^2 d\eta - C^{-1}
\]
for any $-\log(\tilde t_0-t_*)<\tau<\infty$. 
This integral inequality contradicts the fact that $\tau$ varies from 
$-\log(\tilde t_0-t_*)$ to $\infty$. 
Hence the claim holds. 
Since $x_1\in \overline{\Omega_{R/4}}$ and $-1/2<t_1-(r^2/16)<t_1\leq 0$, 
the claim shows the desired inequality. 
\end{proof}

By using the above lemmas, we can estimate $I_r$ by using 
$E_{(x_1,t_1)}( -4\delta^2)$.

\begin{lemma}\label{lem:IrEH}
There exists $C>0$ such that 
\[
	I_r \leq 
	C h\left(
	E_{(x_1,t_1)}( -4\delta^2;\phi_{x_1,R/8} )
	+ C\delta \right)
\]
for any $Q_r(x_1,t_1)\subset Q_{8A\theta_0 \delta}$, 
where $h(s):=s+s^{1/2}$ for $s\geq0$. 
\end{lemma}

\begin{proof}
By combining \eqref{eq:Irlemg} and 
Lemmas \ref{lem:upperpre} and \ref{lem:lowerpre}, 
we see that 
\[
\begin{aligned}
	I_r&\leq 
	C \left(  E_{(x_1,t_1)}( -4\delta^2 ) 
	+ C  \delta + C r \right)^\frac{1}{2} \\
	&\quad 
	+ C_p \left( E_{(x_1,t_1)}( -4\delta^2 ) 
	+ C \delta + C r \right) \log 4
	+ C r  \\
	&\leq 
	Ch \left(  E_{(x_1,t_1)}( -4\delta^2 ) 
	+ 
	C (\delta + r ) 
	\right)
\end{aligned}
\]
for $Q_r(x_1,t_1)\subset Q_{8A\theta_0 \delta}$, 
where $h(s):=s+s^{1/2}$ for $s\geq0$. 
By using $r<\delta$, we obtain 
the desired inequality 
for any $Q_r(x_1,t_1)\subset Q_{8A\theta_0 \delta}$. 
\end{proof}

We have estimated the right-hand side of \eqref{eq:Ires1}, 
and then obtained Lemma \ref{lem:IrEH}. 
To prove Proposition \ref{pro:Irhtilh}, 
we estimate $E_{(x_1,t_1)}( -4\delta^2)$ 
by using a functional defined by  
\[
\begin{aligned}
	J_0(t)
	&=J_0(t; x_1,t_1,R) \\
	&:=
	\int_{\Omega_R}  (|\nabla u(x, t)|^2 + |u(x, t)|^{p+1}) 
	K_{(x_1,t_1)}(x, t) 
	\phi^2_{x_1,R/8}(x) dx. 
\end{aligned} 
\]

\begin{lemma}\label{lem:EJres}
There exists $C>0$ such that 
\[
	E_{(x_1,t_1)} ( -4\delta^2;\phi_{x_1,R/8} ) 
	\leq 
	C \tilde h( \delta^{\frac{4}{p-1}+2} J_0(t) ) 
	+ C \delta
\]
for any $-9\delta^2\leq t\leq -4\delta^2$ 
and $Q_r(x_1,t_1)\subset Q_{8A\theta_0 \delta}$, 
where $\tilde h(s):=s+s^{2/(p+1)}$ for $s\geq0$.
\end{lemma}

\begin{proof}
We first claim that 
\[
	E_{(x_1,t_1)} ( -4\delta^2 )
	\leq 
	E_{(x_1,t_1)} ( t )
	+ C \delta
\]
for $-9\delta^2\leq t\leq -4\delta^2$. 
This inequality clearly holds for $t=-4\delta^2$,
and hence it suffices to consider the case 
$-9\delta^2 \leq t < -4\delta^2$. 
Since $x_1\in \overline{\Omega} \cap B_{R/4}$ and 
$-1/2< -9\delta^2\leq t<-4\delta^2<t_1\leq 0$, 
Proposition \ref{pro:31} shows the claim. Indeed, 
\[
	E_{(x_1,t_1)}(-4\delta^2) 
	\leq 
	E_{(x_1,t_1)}( t) 
	+ C (t_1- t)^\frac{1}{2} \leq 
	E_{(x_1,t_1)}( t) 
	+ C \delta. 
\]

From the definition of $E$, it follows that 
\[
\begin{aligned}
	E_{(x_1,t_1)}(t) 
	&\leq 
	\frac{1}{2} (t_1- t)^\frac{p+1}{p-1}
	\int_{\Omega_R} |\nabla u|^2 
	K_{(x_1,t_1)} \phi^2_{x_1,R/8} dx \\
	&\quad  
	+ (t_1- t)^\frac{2}{p-1}
	\int_{\Omega_R} |u|^2 K_{(x_1,t_1)}
	\phi^2_{x_1,R/8} dx \\
	&\leq 
	C \delta^\frac{2(p+1)}{p-1}
	\int_{\Omega_R} |\nabla u|^2 
	K_{(x_1,t_1)} \phi^2_{x_1,R/8} dx \\
	&\quad  
	+C \delta^\frac{4}{p-1}
	\int_{\Omega_R} |u|^2 K_{(x_1,t_1)}
	\phi^2_{x_1,R/8} dx. 
\end{aligned}
\]
By the H\"older inequality, $\phi_{x_1,R/8}\leq 1$ and 
$\int_{\R^n} K dx =(4\pi)^{n/2}$, 
we have 
\[
\begin{aligned}
	E_{(x_1,t_1)}(t) 
	&\leq 
	C \delta^{\frac{4}{p-1}+2}
	\int_{\Omega_R} |\nabla u|^2 
	K_{(x_1,t_1)} \phi^2_{x_1,R/8} dx \\
	&\quad  
	+C\left( \delta^{\frac{4}{p-1}+2} 
	\int_{\Omega_R} |u|^{p+1} 
	K_{(x_1,t_1)} 
	\phi^{p+1}_{x_1,R/8} dx \right)^\frac{2}{p+1} \\
	&\leq 
	C \tilde h\left( 
	\delta^{\frac{4}{p-1}+2} \int_{\Omega_R} 
	(|\nabla u|^2 + |u|^{p+1}) 
	K_{(x_1,t_1)}
	\phi^2_{x_1,R/8} dx
	\right), 
\end{aligned}
\]
where $\tilde h(s):=s+s^{2/(p+1)}$ for $s\geq0$. 
Hence the desired inequality follows. 
\end{proof}

We next estimate $J_0(t; x_1,t_1,R)$ uniformly for 
$x_1$ and $t_1$.

\begin{lemma}\label{lem:varK}
There exists $C>0$ such that 
\[
	J_0(t; x_1,t_1,R) 
	\leq
	C J_0(t; 0,20\delta^2,2R) 
\]
for any $-9\delta^2\leq t\leq -4\delta^2$ and 
$Q_r(x_1,t_1)\subset Q_{8A\theta_0 \delta}$. 
\end{lemma}

\begin{proof}
By the definition of $J_0$, it suffices to show that 
\begin{align}
	&\label{eq:GGes}
	K_{(x_1,t_1)}(x,t) \leq C K_{(0,20\delta^2)}(x,t), \\
	&\label{eq:phiphies}
	\phi_{x_1,R/8}^2(x) \leq C \phi_{0,R/4}^2(x), 
\end{align}
for $x\in\Omega_R$ and $-9\delta^2\leq t\leq -4\delta^2$. 
Since $(x_1,t_1)\in Q_{8A\theta_0 \delta}$ and $A\theta_0<1/32$, 
we have 
\[
\begin{aligned}
	\frac{K_{(x_1,t_1)}(x,t)}{K_{(0,20\delta^2)}(x,t)}
	&=\left( 
	\frac{20\delta^2-t}{t_1-t}
	\right)^\frac{n}{2}
	\exp\left( 
	- \frac{|x-x_1|^2}{4(t_1-t)} 
	+ \frac{|x|^2}{4(20\delta^2-t)}
	\right) \\
	&\leq 
	\left( 
	\frac{20\delta^2-t}{-(8A\theta_0 \delta)^2-t}
	\right)^\frac{n}{2}
	\exp\left( 
	- \frac{|x-x_1|^2}{36\delta^2} 
	+ \frac{|x|^2}{96\delta^2}
	\right) \\
	&\leq 
	C  \exp\left( 
	- \frac{|x-x_1|^2}{36\delta^2} 
	+ \frac{|x|^2}{96\delta^2}
	\right). 
\end{aligned}
\]
If $|x|\geq 10\delta$, then 
\[
	|x-x_1| \geq |x| - |x_1|
	\geq 10\delta - 8A\theta_0 \delta \geq 10\delta - \delta =9\delta, 
\]
and so 
\[
	\frac{|x|}{|x-x_1|}
	\leq \frac{|x-x_1|+|x_1|}{|x-x_1|}
	\leq 1+\frac{8A\theta_0 \delta}{9\delta} \leq \frac{10}{9}. 
\]
Hence, if $|x|\geq 10\delta$, then 
\[
	\frac{K_{(x_1,t_1)}(x,t)}{K_{(0,20\delta^2)}(x,t)}
	\leq 
	C \exp\left( 
	- \frac{|x-x_1|^2}{36\delta^2} 
	+ \frac{25|x-x_1|^2}{24\cdot 81\delta^2}
	\right) \leq C. 
\]
On the other hand, if $|x|< 10\delta$, then 
\[
	\frac{K_{(x_1,t_1)}(x,t)}
	{K_{(0,20\delta^2)}(x,t)}
	\leq 
	C \exp\left( 
	- \frac{|x-x_1|^2}{36\delta^2} 
	+ \frac{25}{24}
	\right) 
	\leq Ce^\frac{25}{24}. 
\]
Therefore, we obtain \eqref{eq:GGes}.

To check \eqref{eq:phiphies}, we prove 
$\varphi(8|x-x_1|/R) \leq C \varphi(4|x|/R)$. 
In the case $|x-x_1|\leq R/8$, we have 
\[
	|x|\leq |x-x_1|+|x_1| \leq \frac{R}{8} + 8A\theta_0 \delta 
	\leq \frac{3}{16} R. 
\]
Then by the choice of $\varphi$ in the first part of 
Subsection \ref{subsec:dcv}, we see that
$\varphi(4|x|/R) \geq \varphi(3/4)>0$ for $|x-x_1|\leq R/8$. 
Thus, 
\[
	\varphi(8|x-x_1|/R) \leq 1 \leq 
	\frac{\varphi(4|x|/R)}{\varphi(3/4)}. 
\]
In the case $|x-x_1|\geq R/8$, we see that 
$\varphi(8|x-x_1|/R) = 0 \leq \varphi(4|x|/R)$. 
Hence \eqref{eq:phiphies} holds. 
\end{proof}

By Lemmas \ref{lem:EJres} and \ref{lem:varK}, 
to obtain a bound of $E_{(x_1,t_1)}(-4\delta^2)$, 
it suffices to estimate 
$J_0(t; 0,20\delta^2,2R)$ at some $t=\hat t$.

\begin{lemma}\label{lem:enprefin}
There exists 
$-9\delta^2 \leq \hat t \leq -4\delta^2$ 
such that 
\[
	\delta^{\frac{4}{p-1}+2} J_0(\hat t; 0,20\delta^2,2R) 
	\leq 
	C A^{n-\frac{4}{p-1}} \eps_0 + Ce^{-A^2/C}. 
\]
\end{lemma}

\begin{proof}
Due to the mean value theorem, it suffices to prove 
\[
	\delta^{\frac{4}{p-1}}
	\int_{-9\delta^2}^{-4\delta^2} J_0(s; 0,20\delta^2,2R) ds 
	\leq 
	C A^{n-\frac{4}{p-1}} \eps_0 + Ce^{-A^2/C}.  
\]
To see this, we show that 
\begin{align}
	&\label{eq:innJ0}
	\begin{aligned}
	J_1
	&:=\delta^{\frac{4}{p-1}}
	\int_{-9\delta^2}^{-4\delta^2}
	\int_{\Omega_R\cap B_{A\delta}} (|\nabla u|^2 + |u|^{p+1}) 
	K_{(0,20\delta^2)} \phi_{0,R/4}^2 dx ds  \\
	&\leq 
	C A^{n-\frac{4}{p-1}} \eps_0, 
	\end{aligned}
	\\
	&\label{eq:outJ0}
	\begin{aligned}
	J_2
	&:=\delta^{\frac{4}{p-1}}
	\int_{-9\delta^2}^{-4\delta^2}
	\int_{\Omega_R\setminus B_{A\delta}} (|\nabla u|^2 + |u|^{p+1}) 
	K_{(0,20\delta^2)} \phi_{0,R/4}^2 dx ds  \\
	&\leq 
	Ce^{-A^2/C}. 
	\end{aligned}
\end{align}

As for \eqref{eq:innJ0}, we recall our assumption \eqref{eq:asKdel}, that is, 
\[
	\delta^{\frac{4}{p-1}-n}
	\int_{-(A\delta)^2}^0 
	\int_{\Omega_{A\delta}} (|\nabla u|^2 + |u|^{p+1}) dxds 
	\leq A^{n-\frac{4}{p-1}} \eps_0. 
\]
Since $A>3$, we have 
\[
	\delta^{\frac{4}{p-1}-n}
	\int_{-9\delta^2}^{-4\delta^2}
	\int_{\Omega_{A\delta}}
	(|\nabla u|^2 + |u|^{p+1}) dxds 
	\leq A^{n-\frac{4}{p-1}} \eps_0. 
\]
Therefore we have
\[
\begin{aligned}
	J_1
	&\leq 
	\delta^{\frac{4}{p-1}}
	\int_{-9\delta^2}^{-4\delta^2}
	\int_{\Omega_{A\delta}} (|\nabla u|^2 + |u|^{p+1}) 
	(20\delta^2-s)^{-\frac{n}{2}} 
	e^{-\frac{|x|^2}{4(20\delta^2-s)}} dxds \\
	&\leq 
	C \delta^{\frac{4}{p-1}-n}
	\int_{-9\delta^2}^{-4\delta^2}
	\int_{\Omega_{A\delta}} (|\nabla u|^2 + |u|^{p+1}) dxds 
	\leq 
	C A^{n-\frac{4}{p-1}} \eps_0,
\end{aligned}
\]
which proves \eqref{eq:innJ0}.
We next show \eqref{eq:outJ0}. 
For $x\in \Omega_R\setminus B_{A\delta}$ and 
$-9\delta^2\leq s\leq -4\delta^2$, 
we see that 
\[
\begin{aligned}
	\frac{K_{(0,20\delta^2)}(x,s)}
	{K_{(0,21\delta^2)}(x,s)}
	&= 
	\left( \frac{21\delta^2-s}{20\delta^2-s} 
	\right)^\frac{n}{2} 
	\exp\left( -\frac{|x|^2}{4} 
	\frac{\delta^2}{(20\delta^2-s)(21\delta^2-s)}\right) \\
	&\leq C e^{-A^2/C}. 
\end{aligned}
\]
This together with the 
same computations as in Lemma \ref{lem:Ebdd}, we obtain 
\[
\begin{aligned}
	J_2
	&\leq 
	C e^{-A^2/C} 
	\int_{-9\delta^2}^{-4\delta^2}
	\delta^{\frac{4}{p-1}}
	\int_{\Omega_R\setminus B_{A\delta}} (|\nabla u|^2 + |u|^{p+1}) \\
	&\qquad 
	\times K_{(0,21\delta^2)} \phi_{0,R/4}^2 dxds \\
	&\leq 
	C e^{-A^2/C} \delta^{\frac{4}{p-1}} 
	\int_{-9\delta^2}^{-4\delta^2}
	(21\delta^2 -s )^{-\frac{p+1}{p-1}} ds
	\leq C e^{-A^2/C}, 
\end{aligned}
\]
and so \eqref{eq:outJ0} follows. 
As stated in the beginning of the proof, 
the lemma follows. 
\end{proof}

We are now in a position to prove Proposition \ref{pro:Irhtilh}. 

\begin{proof}[Proof of Proposition \ref{pro:Irhtilh}]
By combining Lemmas \ref{lem:EJres}, \ref{lem:varK} and \ref{lem:enprefin}, 
we obtain 
\[
	E_{(x_1,t_1)} ( -4\delta^2;\phi_{x_1,R/8} ) 
	\leq 
	C \tilde h( A^{n-\frac{4}{p-1}} \eps_0 + e^{-A^2/C}  ) 
	+ C \delta. 
\]
for any $Q_r(x_1,t_1)\subset Q_{8A\theta_0 \delta}$.  
This together with Lemma \ref{lem:IrEH} shows 
\[
	I_r \leq 
	C h\left(
	C \tilde h( A^{n-\frac{4}{p-1}} \eps_0 + e^{-A^2/C}  ) 
	+ C\delta \right) 
\]
for any $Q_r(x_1,t_1)\subset Q_{8A\theta_0 \delta}$. 
Since $h\circ \tilde h + h \leq C \overline{h}$ 
with $\overline{h}(s):=s+s^{1/(p+1)}$ for $s\geq0$, 
the desired inequality follows. 
\end{proof}

Finally, we complete the proof of Theorem \ref{th:epsreg}. 

\begin{proof}[Proof of Theorem \ref{th:epsreg}]
As explained just before Proposition \ref{pro:Irhtilh}, 
it suffices to prove that the statement \eqref{eq:x1t1e1} holds 
under the assumption \eqref{eq:asKdel}. 
Let $A$, $\eps_0$, $\delta_0$ and $\theta_0$ satisfy \eqref{eq:Kdelthsmall} 
with $\delta$ replaced by $\delta_0$. 
Let $Q_r(x_1,t_1)\subset Q_{8A\theta_0 \delta_0}$. 
Then by Proposition \ref{pro:Irhtilh}, we have 
\[
	I_r(x_1,t_1) 
	\leq 
	C \overline{h} \left( 
	e^{-A^2/ C} \right) 
	+ C \overline{h} \left( 
	A^{n-\frac{4}{p-1}} \eps_0 \right)  
	+ C \overline{h} ( \delta_0), 
\]
where $C>0$ is a constant depending only on 
$n$, $p$, $M$, $\Omega$ and $R$. 
By taking \eqref{eq:Kdelthsmall} into account, 
we choose constants 
$A$, $\eps_0$, $\delta_0$ and $\theta_0$ satisfying 
the following conditions: 
\begin{equation}\label{eq:choiceKedelt}
\left\{ 
\begin{aligned}
	&A>3 
	&&\mbox{ with }
	C \overline{h}\left( e^{-A^2/C} \right)< \frac{\eps_1}{3}, \\
	&0<\eps_0<1
	&&\mbox{ with } 
	C \overline{h} \left( A^{n-\frac{4}{p-1}} \eps_0 \right) 
	<\frac{\eps_1}{3}, \\
	&0<\delta_0<\frac{R}{16A} 
	&&\mbox{ with }
	C \overline{h} ( \delta_0 )
	<\frac{\eps_1}{3}, \\
	&0<\theta_0<\frac{1}{32A}.  &&
\end{aligned}
\right.
\end{equation}
Here $\eps_1$ is given in Lemma \ref{lem:epspre}. 
Remark that $\eps_0$, $\delta_0$ and $\theta_0$ can be chosen independently. 

Finally, we assume that there exists $0<\delta\leq \delta_0$ such that 
\eqref{eq:asKdel} holds. 
Then we obtain $I_r(x_1,t_1)\leq \eps_1$ 
for any $Q_r(x_1,t_1)\subset Q_{8A\theta_0\delta}$. 
Hence the statement \eqref{eq:x1t1e1} holds 
under the assumption \eqref{eq:asKdel}, 
and so Lemma \ref{lem:epspre} shows that 
$\|u\|_{L^\infty(Q_{A\theta_0\delta})}\leq C(A\theta_0 \delta)^{-2/(p-1)}$. 
Then we can conclude 
that the original assumption \eqref{eq:1scale} implies \eqref{eq:epbdd}. 
The proof of Theorem \ref{th:epsreg} is complete. 
\end{proof}

%%%%%%%%%%%%%%%%%%%%%%%%%%%%%%%%
%%%%%%%%%%%%%%%%%%%%%%%%%%%%%%%%
\section{Proof of main theorem}\label{sec:Blim}
%%%%%%%%%%%%%%%%%%%%%%%%%%%%%%%%
%%%%%%%%%%%%%%%%%%%%%%%%%%%%%%%%
We first prove a localized statement of Theorem \ref{th:main}. 
At the level of strategy, 
the proof is based on \cite{ESS03} and 
\cite[Theorem 9.8]{Tsbook} 
for the Navier-Stokes equations, 
and \cite[Chapter 8]{LWbook} and 
\cite{St88,Wa08} for the harmonic map heat flow. 
Indeed, we use compactness, backward uniqueness and unique continuation.  
However, the analysis in each of the steps seems to be more involved.

\begin{theorem}[Localized statement]\label{th:critical}
Let $n\geq3$, $p>p_S$, $\Omega$ be any $C^{2+\alpha}$ 
domain in $\R^n$ 
and $u$ be a classical $L^{q_c}$-solution of \eqref{eq:main} 
with $u_0\in L^{q_c}(\Omega)$. 
If the maximal existence time $T>0$ is finite 
and $u$ has a blow-up point $a\in \overline{\Omega}$, 
then for any $r>0$, 
\[
	\limsup_{t\to T}\|u(\cdot,t)\|_{L^{q_c}(\Omega_r(a))}=\infty. 
\]
\end{theorem}

Note that if $\Omega$ is bounded, then 
Theorem \ref{th:main} immediately follows from this theorem. 
The unbounded case will be considered in Subsection \ref{subsec:pr}.

%%%%%%%%%%%%%%%%%%%%%%%%%%%%%%%%%%%%%%%%%%%%%%%%
%%%%%%%%%%%%%%%%%%%%%%%%%%%%%%%%%%%%%%%%%%%%%%%%
%%%%%%%%%%%%%%%%%%%%%%%%%%%%%%%%%%%%%%%%%%%%%%%%
\subsection{Existence of blow-up limit}\label{subsec:4uni}
%%%%%%%%%%%%%%%%%%%%%%%%%%%%%%%%%%%%%%%%%%%%%%%%
%%%%%%%%%%%%%%%%%%%%%%%%%%%%%%%%%%%%%%%%%%%%%%%%
%%%%%%%%%%%%%%%%%%%%%%%%%%%%%%%%%%%%%%%%%%%%%%%%
To prove Theorem \ref{th:critical}, we set up notation, 
and then we define rescaled solutions and take the limit. 
Let $u$ be a classical $L^{q_c}$-solution of \eqref{eq:main}. 
Assume that $a\in \overline{\Omega}$ is a blow-up point of $u$. 
By translation and scaling, we may assume that  $a=0$ 
and $u$ satisfies 
\[
\left\{ 
\begin{aligned}
	&u_t=\Delta u+|u|^{p-1}u
	&&\mbox{ in }\Omega_{R_1} \times(-1,0), \\
	&u=0
	&&\mbox{ on }(\partial \Omega \cap B_{R_1})\times(-1,0), 
\end{aligned}
\right.
\]
for some $0<R_1<1$. 
Here we assume that $t=0$ is the blow-up time. 
Suppose, contrary to Theorem \ref{th:critical}, that 
\begin{equation}\label{eq:M_1}
	\sup_{-1<t<0} \|u(\cdot, t)\|_{L^{q_c}(\Omega_{R_2})} \le M
\end{equation}
for some $M>0$ and $0<R_2<1$. 
Fix $0<R<\min\{R_1,R_2\}$ so small 
that \eqref{eq:R1def} and \eqref{eq:R2def} hold. 
Note that $u$ satisfies \eqref{eq:fujitaeq} and \eqref{eq:Mdef2}.

From the parabolic regularity estimates 
in Lemma \ref{lem:parareg}, 
it follows that 
\[
	\|u\|_{W^{1,l}(-1/4,0; L^r(\Omega_{R/2}))} 
	\le C(M+M^p)
\]
for $1\leq l<\infty$ and $1\leq r\le q_c/p$.
Hence after a redefinition of a zero set in the time interval,
we may assume $u \in C([-1/4,0];L^r(\Omega_{R/2}))$.
This together with the uniform bound \eqref{eq:Mdef2} gives 
\begin{equation}\label{eq:u0Mbdd}
	u \in C_\weak([-1/4,0];L^{q_c}(\Omega_{R/2})). 
\end{equation}
By the contraposition of Theorem \ref{th:epsreg}, there exist 
$\eps_0>0$ and $0<\delta_k<1/2$ 
with $\delta_k\to0$ as $k\to \infty$ such that 
\begin{equation}
\label{est:lower}
	{\delta_k}^{\frac{4}{p-1}-n}
	\iint_{Q_{\delta_k}} (|\nabla u|^2 + |u|^{p+1}) dxdt 
	> \eps_0.
\end{equation}

To take a blow-up limit of $u$, 
we define rescaled solutions and derive the corresponding equations. 
In the case \eqref{eq:R1def}, 
we define 
\[
	u_k(x,t):=\delta_k^\frac{2}{p-1} u(\delta_k x,\delta_k^2 t). 
\]
Then, 
\[
\left\{ 
\begin{aligned}
	&(u_k)_t - \Delta u_k 
	= |u_k|^{p-1} u_k 
	&&\mbox{ in } \delta_k^{-1} \Omega_R\times (-\delta_k^{-2},0),   \\
	&u_k=0&&\mbox{ on }
	\delta_k^{-1}(\partial\Omega\cap B_R)\times(-\delta_k^{-2},0). 
\end{aligned}
\right. 
\]
In the case \eqref{eq:R2def}, we define 
\begin{equation}\label{eq:u_k}
	u_k(x,t):=\delta_k^\frac{2}{p-1} \hat u(\delta_k x,\delta_k^2 t), 
	\quad 
	f_k(x'):=\delta_k^{-1} f(\delta_k x'), 
\end{equation}
where $\hat u$ is defined by \eqref{eq:uhatdef}. 
Then $u_k$ satisfies  
\begin{equation}\label{eq:ukresceq}
\left\{ 
\begin{aligned}
	&(u_k)_t - A_k u_k 
	= |u_k|^{p-1} u_k 
	&&\mbox{ in } \delta_k^{-1} \Phi(\Omega_R)\times (-\delta_k^{-2},0),   \\
	&u_k=0&&\mbox{ on }
	\delta_k^{-1}\Phi(\partial\Omega\cap B_R)\times(-\delta_k^{-2},0), 
\end{aligned}
\right. 
\end{equation}
where 
\[
	A_k u_k(x,t):=
	\Delta u_k - 2\nabla' (\partial_{x_n} u_k )
	\cdot \nabla' f_k 
	+(\partial_{x_n}^2 u_k) |\nabla' f_k|^2 
	-(\partial_{x_n}  u_k) \Delta' f_k. 
\]

In what follows, we take a limit of $u_k$ and study its properties. 
Since the case \eqref{eq:R1def} is easier to handle 
than the case \eqref{eq:R2def}, we focus on \eqref{eq:R2def}. 
For $\rho>0$, we set $B_\rho^+:=\R^n_+ \cap B_\rho$. 
To take a limit of $u_k$, 
we give estimates of rescaled solutions uniformly for $k\geq k_\rho$. 
Here $k_\rho\geq 1$ is taken so that 
$0<\delta_k<1/2$, 
$\Psi(B_{3\rho \delta_k}^+)\subset \Omega_{R/2}$ and
$\Psi(B_{3\rho \delta_k})\subset B_{R/2}$ 
for all $k\geq k_\rho$. 
Since $\nabla'f(0)=0$, $f\in C^{2+\alpha}_0(\R^{n-1})$ 
and $\delta_k\to0$ as $k\to\infty$, 
we have 
\begin{align}
	&\label{eq:nabfkuni}
	|\nabla' f_k(x')|
	= |\nabla' f(\delta_k x')|
	\to0, \\
	&\label{eq:delfkuni}
	\|\Delta' f_k\|_{L^\infty(\R^{n-1})}
	\leq 
	\delta_k \| \Delta' f\|_{L^\infty(\R^{n-1})} \to0. 
\end{align}

We first give uniform estimates of $u_k$ and $\nabla u_k$. 

\begin{lemma}\label{lem:unifuknabuk}
Assume \eqref{eq:R2def}. Let $\rho>0$. 
Then there exists $C>0$ such that 
the rescaled functions $u_k$ satisfy 
\[
\begin{aligned}
	&\sup_{-1<t<0}\|u_k(\cdot,t)\|_{L^{q_c}(B_\rho^+)} 
	\leq M, \\
	&\sup_{-1<t<0}
	\|\nabla u_k(\cdot,t)\|_{L^{q_*,\infty}(B_\rho^+)} 
	\leq  C(M+M^p), 
\end{aligned}
\]
for all $k\geq k_\rho$, 
where $C$ depends on $R$ and is independent of $\rho$ and $k$. 
\end{lemma}

\begin{proof}
By the change of variables and the choice of $k_\rho$, 
we can easily see the inequality for $u_k$. 
From the following computation
\[
\begin{aligned}
	|\nabla u_k(x,t)| &= 
	\delta_k^{\frac{2}{p-1}+1}
	| \nabla u(\Psi(\delta_k x),\delta_k^2 t)
	+\partial_{x_n} u(\Psi(\delta_k x),\delta_k^2 t)
	\nabla' f(\delta_k x') |\\
	&\leq 
	\delta_k^{\frac{2}{p-1}+1}
	( 1+\|\nabla' f\|_{L^\infty(\R^{n-1})} )
	| \nabla u(\Psi(\delta_k x),\delta_k^2 t) |
\end{aligned}
\]
and Proposition \ref{pro:gradoricri}, it follows that 
\[
\begin{aligned}
	\|\nabla u_k(\cdot,t)\|_{L^{q_*,\infty}(B_\rho^+)}
	&\leq 
	( 1+\|\nabla' f\|_{L^\infty(\R^{n-1})} )
	\| \nabla u(\cdot,\delta_k^2 t) \|_{L^{q_*,\infty}(\Omega_{R/2})}  \\
	&\leq 
	C(M+M^p)
\end{aligned}
\]
for $-1<t<0$. Then the lemma follows. 
\end{proof}

We next give a uniform parabolic regularity estimate.

\begin{lemma}\label{lem:unifnab2uk}
Assume \eqref{eq:R2def}. 
Let $1\leq l<\infty$, $1\leq r\leq q_c/p$ and $\rho>0$. 
Then there exists $C>0$ such that 
the rescaled functions $u_k$ satisfy 
\[
	\| (u_k)_t \|_{L^l(-1,0; L^r(B_\rho^+))} 
	+ \|\nabla^2 u_k\|_{L^l(-1,0; L^r(B_\rho^+))} 
	\leq  C(M+M^p)
\]
for all $k\geq k_\rho$, 
where $C$ depends on $\rho$ and $R$ and is independent of $k$. 
\end{lemma}

\begin{proof}
By direct computations, we have 
\[
\begin{aligned}
	&\|(u_k)_t\|_{L^l(-1,0; L^r(B_\rho^+))} 
	+ \|\nabla^2 u_k\|_{L^l(-1,0; L^r(B_\rho^+))}  \\
	&= 
	\delta_k^{\frac{2}{p-1}+2-\frac{n}{r}-\frac{2}{l}}
	( \|\hat u_t\|_{L^l(-\delta_k^2,0; L^r(B_{\rho\delta_k}^+))} 
	+ \|\nabla^2 \hat{u}\|_{L^l(-\delta_k^2,0; L^r(B_{\rho\delta_k}^+))} ).
\end{aligned}
\]
This together with Lemma \ref{lem:pararegd} shows the lemma. 
\end{proof}

Finally in this subsection, 
we show the existence of a blow-up limit. 
Let $\rho>0$. 
To avoid technicalities due to the corner of $\partial B^+_{\rho}$, 
we prepare a smooth domain $\cB_\rho$ 
satisfying $B_{\rho/2}^+\subset \cB_\rho \subset B_{\rho}^+$.

\begin{lemma}\label{lem:exbul}
Assume \eqref{eq:R2def}. 
There exist a subsequence still denoted by $u_k$
and a function $\overline{u}$ defined on $\R^n_+\times[-1,0]$ 
satisfying the following statements: 
\begin{itemize}
\item[(i)]
$u_k \rightarrow \overline{u}$
strongly in $L^\infty(-1,0;W^{1,2}(\cB_\rho))$ for each $\rho>0$ as $k\to\infty$.
\item[(ii)]
$u_k \rightarrow \overline{u}$
strongly in $L^\infty(-1,0;L^{p+1}(\cB_\rho))$ for each $\rho>0$ as $k\to\infty$. 
\item[(iii)]
$\|\overline{u} \|_{L^\infty(-1,0;L^{q_c}(\R^n_+))} \le 
M$. 
\item[(iv)]
$\|\nabla \overline{u} \|_{L^\infty(-1,0;L^{q_*,\infty}(\R^n_+))} 
\leq  C(M+M^p)$. 
\end{itemize}
Here $M$ is the constant in \eqref{eq:M_1} 
and $C>0$ is a constant depending on $R$. 
\end{lemma}

\begin{proof}
Let $1<r<\min\{2,q_c/p\}$. 
By a consequence of an Aubin-Lions type compactness result 
(see Lemma \ref{lem:WLWcom}), we have 
\[
	W^{1,5}(-1,0;L^r(\cB_\rho)) \cap L^5(-1,0;W^{2,r}(\cB_\rho))
	\hookrightarrow 
	C([-1,0]; W^{1,r}(\cB_\rho))
\]
and this embedding is compact.
Therefore, the uniform bounds in Lemma \ref{lem:unifnab2uk} show that 
there exists a subsequence still denoted by $\{u_k \}$ satisfying 
$u_k\to\overline{u}$ in 
$C([-1,0];W^{1,r}(\cB_\rho))$ as $k\to\infty$. 
By Lemma \ref{lem:unifuknabuk}, 
the rescaled functions $u_k$ also satisfy the following uniform bounds: 
\begin{align}
	&\notag 
	\sup_{-1<t<0} 
	\|u_k(\cdot,t)\|_{L^{q_c}(\cB_\rho)} 
	\leq  M, \\
	&\label{eq:unifq*}
	\sup_{-1<t<0} \|\nabla u_k(\cdot,t)\|_{L^{q_*,\infty}(\cB_\rho)} 
	\leq  C(M+M^p), 
\end{align}
for all $k\geq k_\rho$, 
where $C$ depends on $R$ and is independent of $\rho$ and $k$. 
Hence the standard interpolations give
\[
\begin{aligned}
	&\begin{aligned}
	\|u_k-\overline{u}\|_{L^\infty(-1,0;L^{p+1}_x)} 
	&\le
	\|u_k-\overline{u}\|_{L^\infty(-1,0;L^{r}_x)}^{\theta_1}
	\|u_k-\overline{u}\|_{L^\infty(-1,0;L^{q_c}_x)}^{1-\theta_1}
	\\
	&\le
	CM\|u_k-\overline{u}\|_{L^\infty(-1,0;L^{r}_x)}^{\theta_1}, 
	\end{aligned} \\
	&\begin{aligned}
	\|\nabla(u_k-\overline{u})\|_{L^\infty(-1,0;L^{2}_x)} 
	&\le
	\|\nabla(u_k-\overline{u})\|_{L^\infty(-1,0;L^{r}_x)}^{\theta_2}
	\|\nabla(u_k-\overline{u})\|_{L^\infty(-1,0;L^{q_*,\infty}_x)}^{1-\theta_2}
	\\
	&\le
	C(M+M^p)\|\nabla(u_k-\overline{u})\|_{L^\infty(-1,0;L^{r}_x)}^{\theta_2}, 
	\end{aligned}
\end{aligned}
\]
where $L^r_x:=L^r(\cB_\rho)$ 
and $L^{q_*,\infty}_x:=L^{q_*,\infty}(\cB_\rho)$. 
These imply (i) and (ii).
By the lower semicontinuity of the weak limit, we see that 
\[
	\|\overline{u}(\cdot,t)\|_{L^{q_c}(\cB_\rho)} 
	\le \liminf_{k\to \infty} \|u_k(\cdot,t)\|_{L^{q_c}(\cB_\rho)} 
	\le M
\]
for $-1<t<0$. Thus, letting $\rho\to\infty$ gives (iii). 
Since the constant $C$ in \eqref{eq:unifq*} 
is independent of $\rho$, we also obtain  (iv). 
The proof is complete. 
\end{proof}

\begin{remark}\label{rem:weakt}
By Lemma \ref{lem:unifnab2uk}, up to subsequence still denoted by $u_k$, 
we can check that  $(u_k)_t \to \overline{u}_t$
weakly in $L^r(\cB_\rho\times (-1,0))$ for each $\rho>0$ as $k\to\infty$. 
\end{remark}

%%%%%%%%%%%%%%%%%%%%%%%%%%%%%%%%%%%%%%%%%%%%%%%%%%
%%%%%%%%%%%%%%%%%%%%%%%%%%%%%%%%%%%%%%%%%%%%%%%%%%
%%%%%%%%%%%%%%%%%%%%%%%%%%%%%%%%%%%%%%%%%%%%%%%%%%
\subsection{Blow-up analysis}
%%%%%%%%%%%%%%%%%%%%%%%%%%%%%%%%%%%%%%%%%%%%%%%%%%
%%%%%%%%%%%%%%%%%%%%%%%%%%%%%%%%%%%%%%%%%%%%%%%%%%
%%%%%%%%%%%%%%%%%%%%%%%%%%%%%%%%%%%%%%%%%%%%%%%%%%
We continue to study the case \eqref{eq:R2def}. 
For \eqref{eq:R1def}, see Remark \ref{rem:R1bu}. 
By using Lemma \ref{lem:exbul}, 
we show several properties of the blow-up limit $\overline{u}$. 
For $\tilde x \in \overline{\R^n_+}$ and $-1<t<\tilde t\leq 0$, 
define a global weighted energy by 
\[
	\overline{E}_{(\tilde x,\tilde t)}(t) :=
	(\tilde t-t)^{\frac{p+1}{p-1}} 
	\int_{\R^n_+} 
	\left( \frac{|\nabla \overline{u}|^2}{2}  
	-\frac{|\overline{u}|^{p+1}}{p+1} 
	+\frac{ \overline{u}^2}{2(p-1)(\tilde t-t) } \right) 
	K_{(\tilde x,\tilde t)}(x,t) dx. 
\]
Note that the same computations as in Lemma \ref{lem:Ebdd} 
together with Lemma \ref{lem:exbul} (iii) and (iv) yield 
\begin{equation}\label{eq:Ebarbdd}
	|\overline{E}_{(\tilde x,\tilde t)}(t)|
	\leq C(M+M^p)^2 
\end{equation} 
for any $\tilde x\in \overline{\R^n_+}$ and 
$-1<t<\tilde t\leq  0$.

We first show that $\overline{E}$ is a scaling limit of $E$.

\begin{lemma}\label{lem:monobul}
Assume \eqref{eq:R2def}. 
Then, for each 
$\tilde x \in \overline{\R^n_+}$ and $-1<t<\tilde t\leq  0$,
\[
	E_{(\delta_k\tilde x,\delta_k^2\tilde{t})}
	(\delta_k^2 t;\phi_{\delta_k\tilde x,R/4})
	\to 
	\overline{E}_{(\tilde x,\tilde t)}(t)
	\quad \mbox{ as }k\to \infty. 
\]
\end{lemma}

\begin{proof}
Fix $\tilde x \in \overline{\R^n_+}$ and $-1<t<\tilde t\leq  0$. 
We take $k$ so large that 
$\tilde x\in \delta_k^{-1} B_{R/4}$ 
and $-(1/2)\delta_k^{-2}<t<0$. 
By the definition of $E$ in \eqref{eq:Edeft} 
and the change of variables, we have 
\[
\begin{aligned}
	& E_{(\delta_k \tilde x,\delta_k^2 \tilde t)}
	(\delta_k^2 t; \phi_{\delta_k\tilde x,R/4}) \\
	&= 
	(\tilde t-t)^{\frac{p+1}{p-1}-\frac{n}{2}} 
	\int_{\delta_k^{-1} \Phi(\Omega_R)} 
	\left( \frac{|\widehat \nabla u_k|^2}{2}  
	-\frac{|u_k|^{p+1}}{p+1} 
	+\frac{|u_k(x,t)|^2}{2(p-1)(\tilde t-t)} \right) \\
	&\qquad \times 
	\exp\left( 
	- \frac{|\delta_k^{-1} \Psi(\delta_k x)-\tilde x|^2}{4(\tilde t-t)}
	\right)
	\varphi^2 \left( \frac{4}{R} 
	|\Psi(\delta_k x)-\delta_k \tilde x| \right)  dx, 
\end{aligned}
\]
where $\widehat \nabla u_k
:=(\nabla'u_k - (\partial_n u_k) \nabla' f_k, \partial_n u_k)$
and $f_k(x')=\delta_k^{-1} f(\delta_k x')$, see \eqref{eq:u_k}. 
In what follows, 
we focus on the convergence of the first term, since 
the other terms can be handled in the same manner.
Namely, we prove that 
\begin{equation}\label{eq:EdelkbarE}
\begin{aligned}
	&\int_{\delta_k^{-1} \Phi(\Omega_R)} 
	|\widehat \nabla u_k|^2 
	e^{ - \frac{|\delta_k^{-1} \Psi(\delta_k x)-\tilde x|^2}
	{4(\tilde t-t)} } \varphi_k^2(x) dx 
	- \int_{\R^n_+} 
	|\nabla \overline{u}|^2 
	e^{- \frac{|x-\tilde x|^2}{4(\tilde t-t)}} dx \\
	&=: 
	J_1^k + J_2^k + J_3^k + J_4^k\to 0 
\end{aligned}
\end{equation}
as $k\to\infty$, where 
$\varphi_k(x):=
\varphi(4\delta_k|\delta_k^{-1}\Psi(\delta_k x)-\tilde x|/R)$ and 
\[
\begin{aligned}
	&J_1^k: = 
	\int_{\delta_k^{-1} \Phi(\Omega_R)} 
	(|\widehat \nabla u_k|^2 - |\nabla \overline{u}|^2 )
	e^{ - \frac{|\delta_k^{-1} \Psi(\delta_k x)-\tilde x|^2}
	{4(\tilde t-t)} }
	\varphi_k^2 dx, \\
	&J_2^k: =  \int_{\delta_k^{-1} \Phi(\Omega_R)}
	|\nabla \overline{u}|^2 
	e^{ - \frac{|\delta_k^{-1} \Psi(\delta_k x)-\tilde x|^2}
	{4(\tilde t-t)} } 
	( \varphi_k^2 -1 ) dx, \\
	&J_3^k: =  \int_{\delta_k^{-1} \Phi(\Omega_R)}
	|\nabla \overline{u}|^2 
	\left( 
	e^{ - \frac{|\delta_k^{-1} \Psi(\delta_k x)-\tilde x|^2}
	{4(\tilde t-t)} } 
	- e^{- \frac{|x-\tilde x|^2}{4(\tilde t-t)}}  
	\right) dx, \\
	&J_4^k: = 
	- \int_{\R^n_+\setminus \delta_k^{-1} \Phi(\Omega_R)} 
	|\nabla \overline{u}|^2 
	e^{- \frac{|x-\tilde x|^2}{4(\tilde t-t)}} dx. 
\end{aligned}
\]
Remark that the resultant convergence in \eqref{eq:EdelkbarE} 
is pointwise for $\tilde x$, $t$ and $\tilde t$, 
but this does not cause any problems for the proof of this lemma. 

Let us estimate $J_1^k$. 
For $R_1'>0$, we estimate and set 
\[
\begin{aligned}
	|J_1^k| &\leq 
	\left(  \int_{B_{R_1'}} 
	+ \int_{\R_+^n \setminus B_{R_1'}} \right) 
	||\widehat \nabla u_k|^2 - |\nabla \overline{u}|^2|
	e^{ - \frac{|\delta_k^{-1} \Psi(\delta_k x)-\tilde x|^2}
	{4(\tilde t-t)} }
	\varphi_k^2 \chi_{\delta_k^{-1} \Phi(\Omega_R)} dx
	\\
	&=: J_{\text{in}}^k + J_{\text{out}}^k. 
\end{aligned}
\]
We consider $J_{\text{out}}^k$. 
From \eqref{eq:R2def} and 
$|f(\delta_k x')| \leq 2^{-1} \delta_k |x'|$, 
it follows that 
\begin{equation}\label{eq:dkmdkx}
	|\delta_k^{-1}\Psi(\delta_k x)-\tilde x|^2 
	= |x+(0,\delta_k^{-1}f(\delta_k x')) 
	-\tilde x|^2 \\
	\geq \frac{1}{8}|x|^2 - |\tilde x|^2. 
\end{equation}
Thus, 
\[
\begin{aligned}
	J_{\text{out}}^k &\leq 
	e^{ \frac{|\tilde x|^2}{8(\tilde t-t)} }
	e^{ - \frac{(R_1')^2}{64(\tilde t-t)} }
	\int_{\R_+^n} 
	(|\widehat \nabla u_k|^2 + |\nabla \overline{u}|^2)
	e^{ - \frac{|\delta_k^{-1} \Psi(\delta_k x)-\tilde x|^2}
	{8(\tilde t-t)} }
	\varphi_k^2 \chi_{\delta_k^{-1} \Phi(\Omega_R)} dx \\
	&\leq 
	e^{ \frac{|\tilde x|^2}{8(\tilde t-t)} }
	e^{ - \frac{(R_1')^2}{64} }
	\int_{\delta_k^{-1} \Phi(\Omega_R)} 
	|\widehat \nabla u_k|^2 
	e^{ - \frac{|\Psi(\delta_k x)-\delta_k\tilde x|^2}
	{8\delta_k^2 (\tilde t-t)} }
	\varphi_k^2 dx \\
	&\quad 
	+e^{ \frac{|\tilde x|^2}{4(\tilde t-t)} }
	e^{ - \frac{(R_1')^2}{64} }
	\int_{\R_+^n} |\nabla \overline{u}|^2 
	e^{ - \frac{|x|^2}{512(\tilde t-t)} } dx
\end{aligned}
\]
for $-1<t<\tilde t\leq  0$. 
The second term in the right-hand side can be estimated by 
computations similar to \eqref{eq:nabucalpp} 
with the aid of Lemma \ref{lem:exbul} (iv). 
As for the first term, 
we go back to the original variables. 
Then by $\supp \varphi^2 ( 4 |\cdot-\delta_k\tilde x| /R ) 
\subset B_{R/4}(\delta_k \tilde x)$, 
$\tilde x\in \delta_k^{-1} B_{R/4}$, 
$-(1/2)\delta_k^{-2}<t<\tilde t\leq 0$ 
and computations similar to \eqref{eq:nabucalpp} 
with \eqref{eq:M_1}, we see that 
\[
\begin{aligned}
	&\int_{\delta_k^{-1} \Phi(\Omega_R)} 
	|\widehat \nabla u_k(x,t)|^2 
	e^{ - \frac{|\Psi(\delta_k x)-\delta_k\tilde x|^2}
	{8\delta_k^2 (\tilde t-t)} }
	\varphi_k^2(x) dx \\
	&= \delta_k^{\frac{2(p+1)}{p-1}-n} \int_{\Omega_R} 
	|\nabla u(y, \delta_k^2 t)|^2 
	e^{ - \frac{|y-\delta_k\tilde x|^2}
	{8\delta_k^2 (\tilde t-t)} }
	\varphi^2 \left( \frac{4}{R} 
	|y-\delta_k\tilde x| \right) dy \\
	&\leq 
	\delta_k^{\frac{2(p+1)}{p-1}-n} \int_{\Omega_{R/2}} 
	|\nabla u(y, \delta_k^2 t)|^2 
	e^{ - \frac{|y-\delta_k\tilde x|^2}
	{8\delta_k^2 (\tilde t-t)} } dy 
	\leq 
	C(M+M^p)^2 (\tilde t-t)^{\frac{n}{2}-\frac{p+1}{p-1}}, 
\end{aligned}
\]
where $C>0$ is independent of $k$. 
Thus, 
\[
	J_{\text{out}}^k \leq 
	C e^{ \frac{|\tilde x|^2}{4(\tilde t-t)} }
	(\tilde t-t)^{\frac{n}{2}-\frac{p+1}{p-1}}
	e^{ - \frac{(R_1')^2}{64} }. 
\]

We examine $J_{\text{in}}^k$. 
By using  $\widehat \nabla u_k
=(\nabla'u_k - (\partial_n u_k) \nabla' f_k, \partial_n u_k)$ and 
$f_k(x')=\delta_k^{-1} f(\delta_k x')$, 
we have 
\[
\begin{aligned}
	&|\widehat \nabla u_k|
	=|\nabla u_k - 
	( (\partial_n u_k) \nabla' f(\delta_k x'), 0)|
	\leq (1+|\nabla' f(\delta_k x')|) |\nabla u_k|, \\ 
	&|\widehat \nabla u_k -\nabla \overline{u}|^2
	\leq C|\nabla u_k-\nabla \overline{u}|^2 
	+ C |\nabla u_k|^2 |\nabla' f(\delta_k x')|^2. 
\end{aligned}
\]
These estimates together with  the H\"older inequality show that 
\[
\begin{aligned}
	J_{\text{in}}^k 
	&\leq 
	\int_{B_{R_1'}} 
	|\widehat \nabla u_k + \nabla \overline{u}|
	|\widehat \nabla u_k - \nabla \overline{u}|
	\chi_{\delta_k^{-1} \Phi(\Omega_R)} dx \\ 
	&\leq 
	\left( \int_{B_{R_1'}} 
	( (1+|\nabla' f(\delta_k x')|) |\nabla u_k| + |\nabla \overline{u}| 
	)^2 \chi dx \right)^\frac{1}{2} \\
	&\quad \times 
	\left( C \int_{B_{R_1'}} 
	|\nabla u_k-\nabla \overline{u}|^2 \chi dx
	+ C \int_{B_{R_1'}}  |\nabla u_k|^2 |\nabla' f(\delta_k x')|^2
	\chi dx \right)^\frac{1}{2}
\end{aligned}
\]
where $\chi:=\chi_{\delta_k^{-1} \Phi(\Omega_R)}$. 
By $f\in C^{2+\alpha}_0(\R^{n-1})$, 
\eqref{eq:unifq*} and Lemma \ref{lem:exbul} (iv), 
the first integral in the right-hand side is 
bounded by a constant. 
In addition, Lemma \ref{lem:exbul} (i) guarantees that 
the second integral converges to $0$ as $k\to \infty$. 
Therefore,
from $\nabla'f(0)=0$ and 
the Lebesgue dominated convergence theorem, it follows that 
\[
	\limsup_{k\to\infty} J_{\text{in}}^k
	\leq 
	C \limsup_{k\to\infty} 
	\left( \int_{B_{R_1'}} |\nabla u_k|^2 
	|\nabla' f(\delta_k x')|^2
	\chi_{\delta_k^{-1} \Phi(\Omega_R)} dx \right)^\frac{1}{2}=0. 
\]
Hence the above estimates for $J_{\text{in}}^k$ 
and $J_{\text{out}}^k$ show that 
\[
	\limsup_{k\to\infty} |J_1^k| \leq 
	C e^{ \frac{|\tilde x|^2}{4(\tilde t-t)} }
	(\tilde t-t)^{\frac{n}{2}-\frac{p+1}{p-1}}
	e^{ - \frac{(R_1')^2}{64} } \to 0
\]
as $R_1'\to\infty$, and so $\lim_{k\to\infty} J_1^k=0$.

We consider $J_2^k$. 
For $R_2'>0$, we note that 
\[
	|\delta_k^{-1} \Psi(\delta_k x) - \tilde x|
	\leq |x| + \frac{1}{2} |x'| + |\tilde x| 
	\leq \frac{3}{2} R_2' + |\tilde x| 
\]
if $x\in B_{R_2'}$. 
This together with \eqref{eq:dkmdkx}, 
$0\leq \varphi_k\leq 1$ and $\varphi'\leq 0$ shows that 
\[
\begin{aligned}
	\limsup_{k\to\infty} |J_2^k| 
	&\leq 
	\limsup_{k\to\infty} 
	\left( 1- \varphi^2 \left( \frac{4\delta_k}{R} 
	\left(\frac{3}{2} R_2' + |\tilde x|\right) \right) \right)
	\int_{B_{R_2'}^+} 
	|\nabla \overline{u}|^2  dx \\
	&\quad 
	+ 
	e^\frac{|\tilde x|^2}{4(\tilde t-t)} 
	\int_{\R^n_+\setminus B_{R_2'}^+} 
	|\nabla \overline{u}|^2 
	e^{ - \frac{|x|^2}{32(\tilde t-t)} }  dx 
\end{aligned}
\]
for $-1<t<\tilde t\leq  0$. 
Hence the inner part converges to $0$ as $k \rightarrow \infty$, and 
then letting $R_2'\to\infty$ yields $\lim_{k\to\infty} J_2^k =0$.

As for $J_3^k$, by \eqref{eq:dkmdkx}, we have 
\[
\begin{aligned}
	|J_3^k| 
	&\leq 
	\int_{B_{R_3'}^+} 
	|\nabla \overline{u}|^2 
	\left| 
	e^{ - \frac{|\delta_k^{-1} \Psi(\delta_k x)-\tilde x|^2}
	{4(\tilde t-t)} } 
	- e^{- \frac{|x-\tilde x|^2}{4(\tilde t-t)}}  
	\right| dx \\
	&\quad 
	+ \int_{\R^n_+\setminus B_{R_3'}^+} 
	|\nabla \overline{u}|^2 
	\left( 
	e^{ - \frac{|x|^2}{32(\tilde t-t)} } 
	e^\frac{|\tilde x|^2}{4(\tilde t-t)} 
	+ e^{- \frac{|x|^2}{8(\tilde t-t)}}  
	e^\frac{|\tilde x|^2}{4(\tilde t-t)} 
	\right) dx. 
\end{aligned}
\]
The outer part can be handled in the same way as above. 
We can also check that 
the inner part converges to $0$ as $k\to\infty$ by 
\[
	\delta_k^{-1} \Psi(\delta_k x)
	= x+(0,\delta_k^{-1} f(\delta_k x'))
	= x+(0,\nabla' f(\theta_{x'} \delta_k x'))
	\to x 
\]
for some $0\leq \theta_{x'}\leq 1$ as $k\to\infty$. 
Therefore $\lim_{k\to\infty} J_3^k =0$. 
By Lemma \ref{lem:exbul} (iv), we can easily see that 
$\lim_{k\to\infty} J_4^k =0$. 
Hence we obtain \eqref{eq:EdelkbarE}. 
\end{proof}

We next estimate an integral concerning $|\overline{u}|^{p+1}$ 
by using $\overline{E}_{(\tilde x,\tilde t)}$. 
More precisely, we prove an analog of Lemma \ref{lem:32} 
for the blow-up limit $\overline{u}$.

\begin{lemma}
Assume \eqref{eq:R2def}. 
Then there exists $C>0$ such that 
\[
\begin{aligned}
	&\int_{t'}^t (\tilde t-s)^{\frac{p+1}{p-1}} 
	\int_{\R^n_+} |\overline{u}|^{p+1} K_{(\tilde x,\tilde t)}(x,s) dx ds \\
	&\leq 
	C \left( \log\frac{\tilde t-t'}{\tilde t-t} \right)^\frac{1}{2}
	(  \overline{E}_{(\tilde x,\tilde t)}(t') 
	- \overline{E}_{(\tilde x,\tilde t)}(t) 
	+ C(\tilde t-t')^\frac{1}{2} )^\frac{1}{2} \\
	&\quad 
	+ C_p (\overline{E}_{(\tilde x,\tilde t)}(t') 
	+ C(\tilde t-t')^\frac{1}{2}) 
	\log\frac{\tilde t-t'}{\tilde t-t}  
	+ C(\tilde t-t')^\frac{1}{2}
\end{aligned}
\]
for $\tilde x \in \overline{\R^n_+}$ and 
$-1<t'<t<\tilde t\leq 0$. 
Here $C_p:=2(p+1)/(p-1)$ and 
$h(s):=s+s^{1/2}$ for $s\geq 0$. 
\end{lemma}

\begin{proof}
Define $\tilde \varphi\in C^\infty([0,\infty))$ so that  
$\tilde \varphi' \leq 0$, 
$0\leq \tilde \varphi\leq 1$, 
$\tilde \varphi(z)=0$ for $z>2$ and 
$\tilde \varphi(z)=1$ for $0\leq z<1$. 
For $\tilde R>1$, set $\tilde \varphi_{\tilde R}(z):= \tilde \varphi(z/\tilde R)$. 
For $\tilde x \in \overline{\R^n_+}$ and 
$-1<t<\tilde t\leq 0$, 
we define 
\[
\begin{aligned}
	E_k(t)&:= 
	(\tilde t-t)^{\frac{p+1}{p-1}-\frac{n}{2}} 
	\int_{\delta_k^{-1} \Phi(\Omega_R)} 
	\left( \frac{|\widehat \nabla u_k(\xi,t)|^2}{2}  
	-\frac{|u_k|^{p+1}}{p+1} 
	+\frac{(\tilde t-t)^{-1} |u_k|^2}{2(p-1)} \right) \\
	&\qquad \times 
	e^{ - \frac{|\Psi_k(\xi)-\Psi_k(\tilde \xi_k)|^2}{4(\tilde t-t)} }
	\varphi^2 \left( \frac{4\delta_k}{R} |\Psi_k(\xi)-\Psi_k(\tilde \xi_k)| \right) 
	\tilde \varphi_{\tilde R}^2(|\xi|) d\xi, 
\end{aligned}
\]
where 
$\Psi_k(\xi):= \delta_k^{-1} \Psi(\delta_k \xi)$, 
$\Phi_k(x):= \delta_k^{-1} \Phi(\delta_k x)$ and 
$\tilde \xi_k:=\Phi_k(\tilde x)$. 
We note that $\Psi_k(\tilde \xi_k)=\tilde x$ and that 
the right-hand side of this definition coincides with 
$E_{(\delta_k \tilde x,\delta_k^2 \tilde t)}
(\delta_k^2 t; \phi_{\delta_k\tilde x,R/4})$ 
if we replace $\tilde \varphi_{\tilde R}$ with $1$.

By using the backward similarity variables 
$\eta :=(\xi-\tilde \xi_k)/(\tilde t-t)^{1/2}$ and 
$\tau:=-\log(\tilde t-t)$, we define 
\[
\begin{aligned}
	&w_k(\eta,\tau):= e^{-\frac{1}{p-1}\tau} 
	u_k( \tilde \xi_k +e^{-\frac{1}{2}\tau} \eta, \tilde t-e^{-\tau}), \\
	&g_k(\eta',\tau):= 
	e^{\frac{1}{2}\tau} f_k( \tilde \xi_k'+e^{-\frac{1}{2}\tau} \eta' ). 
\end{aligned}
\]
Then $w_k$ satisfies  
\[
	\left\{ 
	\begin{aligned}
	&(w_k)_\tau + \frac{1}{2} \eta\cdot \nabla w_k
	+\frac{1}{p-1} w_k -A_k w_k - |w_k|^{p-1}w_k =0, \\
	&\quad \eta\in 
	\Omega_k(\tau):=e^{\tau/2}(\Phi_k(\delta_k^{-1}\Omega_R)-\tilde x), \ 
	\tau\in  (-\log(\tilde t+1), \infty), \\
	&w_k=0, \quad \eta\in e^{\tau/2}
	(\Phi_k(\delta_k^{-1}(\partial\Omega\cap B_R))-\tilde \xi_k), 
	\end{aligned}
	\right.
\]
where 
$A_k w_k:=
	\Delta w_k 
	- 2\nabla' (\partial_n w_k )
	\cdot \nabla' g_k 
	+(\partial_n^2 w_k) |\nabla' g_k|^2 
	-(\partial_n w_k) \Delta' g_k$. 
Setting $\widehat \nabla w_k
:=(\nabla' w_k-(\partial_n w_k) \nabla' g_k, \partial_n w_k)$ and 
\[
\begin{aligned}
	&\cE_k(\tau):= 
	\int_{\Omega_k(\tau)} 
	\left( 
	\frac{ |\widehat \nabla w_k|^2}{2}
	-\frac{|w_k|^{p+1}}{p+1} 
	+\frac{w_k^2}{2(p-1)} \right) 
	\rho_k \psi_k^2 d\eta, \\
	&\begin{aligned}
	&\rho_k=\rho_k(\eta,\tau) :=\exp\bigg( 
	-\frac{1}{4} 
	( |\eta|^2 + 2(g_k(\eta',\tau)-g_k(0,\tau))\eta_n \\
	&\qquad + (g_k(\eta',\tau)-g_k(0,\tau))^2 ) \bigg), 
	\end{aligned} \\
	&\psi_k=\psi_k(\eta,\tau):= 
	\varphi\left( 
	\frac{4|\Psi_k(\tilde \xi_k+e^{-\tau/2}\eta)-\Psi_k(\tilde \xi_k)|}{R} 
	\right) 
	\tilde \varphi_{\tilde R}(| \tilde \xi_k + e^{-\frac{\tau}{2}}\eta| ),
\end{aligned}
\]
we can check that $E_k(t)=\cE_k(\tau)$ with $\tau=-\log(\tilde t-t)$. 
In addition, we can observe that 
the above situation is almost the same as in Section \ref{sec:pre}, 
except for the appearance of $\tilde \varphi_{\tilde R}$.

By the choice of $\tilde \varphi_{\tilde R}$ and $\tilde R>1$, we have 
\[
\begin{aligned}
	&|\nabla ( \tilde \varphi_{\tilde R}(| \tilde \xi_k + e^{-\frac{\tau}{2}}\eta| )) |
	\leq 
	C|\tilde \varphi_{\tilde R}'(| \tilde \xi_k + e^{-\frac{\tau}{2}}\eta| ) | e^{-\frac{\tau}{2}}
	\leq 
	C e^{-\frac{\tau}{2}}, \\
	&| \partial_\tau( \tilde \varphi_{\tilde R}(| \tilde \xi_k + e^{-\frac{\tau}{2}}\eta| )) |
	\leq 
	C|\tilde \varphi_{\tilde R}'(| \tilde \xi_k + e^{-\frac{\tau}{2}}\eta| ) | |\eta| 
	e^{-\frac{\tau}{2}}
	\leq 
	C|\eta|e^{-\frac{\tau}{2}}, 
\end{aligned}
\]
where $C>0$ is independent of $k$ and $\tilde R$. 
Then the same computations 
as in Section \ref{sec:pre} show that 
there exists a constant $C>0$ independent of $k$ and $\tilde R$ satisfying 
\[
\begin{aligned}
	\cE_k(\tau) + \frac{1}{2} \int_{\tau'}^\tau \int_{\Omega_k(\sigma)} 
	(w_k)_\sigma^2 \rho_k \psi_k^2 d\eta d\sigma
	&\leq  \cE_k(\tau')
	+ C (\tilde t-t')^\frac{1}{2}. 
\end{aligned}
\]
Therefore, the same computations as in Lemma \ref{lem:32} yield 
\[
\begin{aligned}
	&\int_{t'}^t (\tilde t-s)^{\frac{p+1}{p-1}-\frac{n}{2}} 
	\int_{\delta_k^{-1} \Phi(\Omega_R)} |u_k|^{p+1}  
	\exp\left( 
	- \frac{|\delta_k^{-1} \Psi(\delta_k x)-\tilde x|^2}{4(\tilde t-s)}
	\right)
	\\
	&\qquad \times 
	\varphi^2 \left( \frac{4}{R} |\Psi(\delta_k x)-\delta_k \tilde x| \right) 
	\tilde \varphi^2_{\tilde R}(|x|) dx ds \\
	&\leq 
	C \left( \log\frac{\tilde t-t'}{\tilde t-t} \right)^\frac{1}{2}
	(  E_k(t') - E_k(t) 
	+ C (\tilde t-t')^\frac{1}{2} )^\frac{1}{2} \\
	&\quad 
	+ C_p (E_k(t') 
	+ C (\tilde t-t')^\frac{1}{2}) 
	\log\frac{\tilde t-t'}{\tilde t-t}  
	+ C (\tilde t-t')^\frac{1}{2}, 
\end{aligned}
\]
where $C_p:=2(p+1)/(p-1)$ and $C>0$ is independent of $k$ and $\tilde R$. 
By Lemma \ref{lem:exbul} (ii), 
the same argument as in Lemma \ref{lem:monobul} and letting $k\to\infty$, 
we have 
\[
\begin{aligned}
	&\int_{t'}^t (\tilde t-s)^{\frac{p+1}{p-1}} 
	\int_{\R^n_+} |\overline{u}|^{p+1} K_{(\tilde x,\tilde t)}(x,s)
	\tilde \varphi_{\tilde R}^2(|x|) dx ds \\
	&\leq 
	C \left( \log\frac{\tilde t-t'}{\tilde t-t} \right)^\frac{1}{2}
	(  \overline{E}_{\tilde \varphi_{\tilde R}}(t') 
	- \overline{E}_{\tilde \varphi_{\tilde R}}(t) 
	+ C(\tilde t-t')^\frac{1}{2} )^\frac{1}{2} \\
	&\quad 
	+ C_p (\overline{E}_{\tilde \varphi_{\tilde R}}(t') 
	+ C(\tilde t-t')^\frac{1}{2}) 
	\log\frac{\tilde t-t'}{\tilde t-t}  
	+ C (\tilde t-t')^\frac{1}{2}, 
\end{aligned}
\]
where 
\[
	\overline{E}_{\tilde \varphi_{\tilde R}}(t) \\
	:=
	(\tilde t-t)^{\frac{p+1}{p-1}} 
	\int_{\R^n_+} 
	\left( \frac{|\nabla \overline{u}|^2}{2}  
	-\frac{|\overline{u}|^{p+1}}{p+1} 
	+\frac{\overline{u}^2}{2(p-1)(\tilde t-t)} \right) 
	K_{(\tilde x,\tilde t)} \tilde \varphi_{\tilde R}^2  dx. 
\]
Letting $\tilde R\to\infty$ 
with the aid of the monotone convergence theorem to the left-hand side 
and the Lebesgue dominated convergence theorem to the right-hand side, 
we obtain the desired inequality. 
\end{proof}

\begin{remark}\label{rem:wepregv}
As stated in Remark \ref{rem:wepreg}, 
the $\eps$-regularity (Theorem \ref{th:epsreg}) is also valid 
for $\overline{u}$, 
since $\overline{u}$ satisfies 
the analog of Lemma \ref{lem:32}. 
\end{remark}

We next show a monotonicity estimate of $\overline{E}_{(0,0)}$. 

\begin{lemma}\label{lem:monoubar}
Assume \eqref{eq:R2def}. 
Then there exists $C>0$ such that 
\[
\begin{aligned}
	&\overline{E}_{(0,0)}(t')
	- \overline{E}_{(0,0)}(t) 
	\geq 
	C^{-1} \frac{ (-t)^{4/(p-1)}}{(-t')^{2p/(p-1)}} \\
	&\quad \times 
	\left( 
	\int_{t'}^t  
	\int_{\R^n_+}
	\left( 
	\frac{\overline{u}}{p-1}
	+\frac{x\cdot\nabla \overline{u}}{2} + s \overline{u}_s
	\right)  
	K_{(0,0)}(x,s) \zeta(x,s) dxds
	\right)^2 
\end{aligned}
\]
for any $-1<t'<t<0$ and 
$\zeta\in C^\infty_0(\R^n_+\times(-1,0))$ with $|\zeta|\leq 1$. 
\end{lemma}

\begin{proof}
Let $-1<t'<t<0$. 
We take $k$ so large that $-(1/2)\delta_k^{-2}< t'<t<0$. 
From \eqref{eq:419} and the change of variables, 
it follows that 
\begin{equation}
\label{est:427}
	E_{(0,0)}(\delta_k^2 t;\phi_{0,R/4})
	+\frac{1}{2} \cI 
	\leq E_{(0,0)}(\delta_k^2 t';\phi_{0,R/4})
	+ C \delta_k (-t')^\frac{1}{2}, 
\end{equation}
where 
\[
\begin{aligned}
	\cI &:= \int_{t'}^{t} 
	(-s)^{\frac{2}{p-1} -\frac{n}{2}-1} 
	\int_{ \delta_k^{-1} \Phi(\Omega_R)} 
	\left| \frac{u_k}{p-1}
	+\frac{x\cdot\nabla u_k}{2}
	+s(u_k)_s \right|^2  \\
	&\qquad 
	\times 
	\exp\left( 
	-\frac{|\delta_k^{-1}\Psi(\delta_k x)|^2}{4( -s)}
	\right)
	\varphi^2\left( \frac{4}{R} |\Psi(\delta_k x)| \right)
	dx ds. 
\end{aligned}
\]
Let $\zeta\in C^\infty_0(\R^n_+\times(-1,0))$ 
satisfy $|\zeta|\leq 1$. Then, 
\[
\begin{aligned}
	\cI\geq 
	\iint_{\delta_k^{-1} \Phi(\Omega_R) \times (t',t)}
	|F(u_k)|^2 \zeta^2(x,s) d\mu_k(x,s). 
\end{aligned}
\]
Here $F(u_k)$ and $\mu_k$ are defined by 
\[
\begin{aligned}
	&F(u_k):=\frac{u_k}{p-1}
	+\frac{x\cdot\nabla u_k}{2}
	+ s (u_k)_s, \\
	&d\mu_k(x,s):= 
	(-s)^{\frac{2}{p-1} -1} 
	K_{(0,0)}(\delta_k^{-1}\Psi(\delta_k x),s)
	\varphi^2\left( \frac{4}{R} |\Psi(\delta_k x)| \right) dxds. 
\end{aligned}
\]

By \eqref{eq:dkmdkx} with $\tilde x=0$ and $\varphi\leq 1$, 
we compute that 
\[
\begin{aligned}
	\mu_k(\delta_k^{-1} \Phi(\Omega_R) \times (t',t))
	&\leq 
	C\int_{t'}^t
	(-s)^{\frac{2}{p-1} -1} 
	\int_{\R^n}
	K_{(0,0)}\left( \frac{x}{\sqrt{8}},s \right) dxds \\
	&\leq 
	C(-t')^\frac{2}{p-1}, 
\end{aligned}
\]
where $C$ depends only on $n$ and $p$ 
and is independent of $k$. 
Then by Jensen's inequality, we see that 
\[
\begin{aligned}
	\cI &\geq 
	\frac{1}{\mu_k(\delta_k^{-1} \Phi(\Omega_R) 
	\times (t',t)) }
	\left( \iint_{\delta_k^{-1} \Phi(\Omega_R) 
	\times (t',t)}
	F(u_k) \zeta  d\mu_k(x,s)
	\right)^2 \\
	&\geq 
	C^{-1} (-t')^{-\frac{2}{p-1}}
	\left( \iint_{\delta_k^{-1} \Phi(\Omega_R)
	 \times (t',t)}
	F(u_k) \zeta  d\mu_k(x,s)
	\right)^2. 
\end{aligned}
\]
From this and \eqref{est:427}, it follows that 
\[
\begin{aligned}
	&E_{(0,0)}(\delta_k^2 t;\phi_{0,R/4})
	+
	C^{-1} (-t')^{-\frac{2}{p-1}}
	\left( \iint_{\delta_k^{-1} \Phi(\Omega_R) 
	\times (t',t)}
	F(u_k) \zeta  d\mu_k
	\right)^2 \\
	&\leq E_{(0,0)}(\delta_k^2 t';\phi_{0,R/4})
	+ C \delta_k (-t')^\frac{1}{2}. 
\end{aligned}
\]

From Remark \ref{rem:weakt}, Lemma \ref{lem:exbul} 
and computations similar to the derivation of \eqref{eq:EdelkbarE}, 
it follows that  
\[
	\iint_{\delta_k^{-1} \Phi(\Omega_R) \times (t',t)}
	F(u_k) \zeta  d\mu_k(x,s) 
	\to 
	\int_{t'}^t 
	(-s)^{\frac{2}{p-1} -1} 
	\int_{\R^n_+}
	F(\overline{u})  
	K_{(0,0)} \zeta dxds
\]
as $k\to\infty$. 
This together with Lemma \ref{lem:monobul} and letting $k\to\infty$ implies that 
\[
\begin{aligned}
	&\overline{E}_{(0,0)}(t')-\overline{E}_{(0,0)}(t) \\
	&\geq 
	C^{-1} (-t')^{-\frac{2}{p-1}}
	\left( 
	\int_{t'}^t 
	(-s)^{\frac{2}{p-1} -1} 
	\int_{\R^n_+}
	F(\overline{u})  
	K_{(0,0)} \zeta dxds
	\right)^2 \\
	&\geq 
	C^{-1} (-t')^{-\frac{2p}{p-1}}
	(-t)^{\frac{4}{p-1}}
	\left( 
	\int_{t'}^t  
	\int_{\R^n_+}
	F(\overline{u})  
	K_{(0,0)} \zeta dxds
	\right)^2. 
\end{aligned}
\]
Then the desired inequality follows. 
\end{proof}

The monotonicity estimate gives the following equality 
based on the argument of \cite[Theorem 8.1]{St88}.

\begin{lemma}\label{lem:SS}
Assume \eqref{eq:R2def}. Then, 
\[
	\frac{\overline{u}(x,t)}{p-1} 
	+ \frac{x\cdot \nabla \overline{u}}{2} 
	+t\overline{u}_t 
	=0 \quad 
	\mbox{ a.e. in }\R^n_+\times (-1,0). 
\]
\end{lemma}

\begin{proof}
For $0<r<1/2$, we set 
\[
	\Phi(r)
	:=\int^{-r^2}_{-4r^2} 
	\overline{E}_{(0,0)}(t)\frac{dt}{-t}. 
\]
We note that \eqref{eq:Ebarbdd} guarantees that  $\Phi$ is well-defined.
By Lemmas \ref{lem:monobul} and \ref{lem:Ebdd}, 
the Lebesgue dominated convergence theorem and the change of variables, we have 
\[
	\Phi(r)
	=
	\lim_{k\to \infty}\int^{-r^2}_{-4r^2}
	E_{(0,0)}(\delta_k^2 t)\frac{dt}{-t} 
	= 
	\lim_{k\to \infty}\int^{-\delta_k^2 r^2}_{-4 \delta_k^2 r^2}
	E_{(0,0)}(s)
	\frac{ds}{-s}. 
\]

We claim that $\Phi(r)$ is independent of $r$. 
To show this, we set 
\[
	H(\tilde r) := \int_{- 4 \tilde r^2}^{- \tilde r^2}
	E_{(0,0)}(s)\frac{ds}{-s}
	= \int_{- 4}^{- 1}
	E_{(0,0)}(\tilde r^2 \lambda)\frac{d\lambda}{-\lambda} 
\]
for $0<\tilde r<1/2$ 
and check that $\lim_{\tilde r\to0} H(\tilde r)$ exists. 
Let $\eps>0$. 
Since Lemma \ref{lem:Ebdd} yields the boundedness of $H$, 
we see that $H_\eps:=\inf_{0<\tilde r<\eps} H(\tilde r)$ is finite. 
Then there exists $0<\tilde r_\eps<\eps$ such that 
$H(\tilde r_\eps) \leq H_\eps+\eps$. 
From Proposition \ref{pro:31} and the negativity of $\lambda$, 
it follows that 
\[
\begin{aligned}
	\limsup_{\tilde r\to0} H(\tilde r) &\leq 
	\limsup_{\tilde r\to0} \int_{- 4}^{- 1}
	( E_{(0,0)}(\tilde r_\eps^2 \lambda) 
	+C(M+M^p)^2 (-\tilde r_\eps^2 \lambda) )
	\frac{d\lambda}{-\lambda}  \\
	&= 
	H(\tilde r_\eps) + 3C(M+M^p)^2 \tilde r_\eps^2 
	\leq 
	H_\eps+\eps + 3C(M+M^p)^2 \eps^2. 
\end{aligned}
\]
Thus, letting $\eps\to0$ gives $\limsup_{\tilde r\to0} H(\tilde r) \leq 
\liminf_{\tilde r\to0} H(\tilde r)$. 
Hence the limit of $H(\tilde r)$ exists, and so
$\Phi(r) = \lim_{\tilde r\to0} H(\tilde r)$. 
Then the claim follows.

From the claim and Lemma \ref{lem:monoubar}, it follows that 
\[
\begin{aligned}
	0&= \int^{r_2}_{r_1} \Phi'(s) ds 
	=
	\int^{r_2}_{r_1}\frac{2}{s}
	( \overline{E}_{(0,0)}(-4s^2)
	-\overline{E}_{(0,0)}(-s^2) ) ds \\
	&\geq 
	C^{-1}\int^{r_2}_{r_1} s^\frac{9-5p}{p-1} 
	\left( 
	\int_{-4s^2}^{-s^2}  
	\int_{\R^n_+}
	\left( 
	\frac{\overline{u}}{p-1}
	+\frac{x\cdot\nabla \overline{u}}{2} + t \overline{u}_t
	\right)  
	K_{(0,0)} \zeta dxdt
	\right)^2  ds 
\end{aligned}
\]
for any $0<r_1<r_2<1/2$ and 
$\zeta\in C^\infty_0(\R^n_+\times(-1,0))$ with $|\zeta|\leq 1$. 
Thus, 
\[
\begin{aligned}
	\int_{-4s^2}^{-s^2}  
	\int_{\R^n_+}
	\left( 
	\frac{\overline{u}}{p-1}
	+\frac{x\cdot\nabla \overline{u}}{2} + t \overline{u}_t
	\right)  
	K_{(0,0)}(x,t) \zeta(x,t) dxdt = 0
\end{aligned}
\]
for a.e. $s\in (0,1/2)$. 
Hence by the fundamental lemma of the calculus of variations, 
the lemma follows. 
\end{proof}

Finally in this subsection, 
we prove a partial regularity result for $\overline{u}$ 
by using the $\eps$-regularity. 
For $t\in (-1,0)$, define 
the singular set of $\overline{u}(\cdot,t)$ by 
\[
	\Sigma(t):= 
	\{ x\in \R^n_+; \overline{u}(\cdot,t) \not\in L^\infty(B_\rho^+(x)) 
	\mbox{ for all }\rho>0 \}. 
\]
By modifying the argument of Wang \cite[Lemma 3.3]{Wa08}, 
we show that $\Sigma(t)$ consists of 
at most finitely many points.

\begin{lemma}\label{lem:parreg}
Assume \eqref{eq:R2def}. 
Then the singular set 
$\Sigma(t)$ consists of finitely many points for each 
$t\in (-1/4,0)$. 
More precisely, there exists a constant $C>0$ 
depending on $R$ such that 
the cardinality $\card(\Sigma(t))$ of $\Sigma(t)$ satisfies 
\[
	\card(\Sigma(t)) \leq 
	C(M+M^p)^{q_c}\eps_0^{-\frac{q_c}{2}}
\]
for each $t\in (-1/4,0)$. 
Here 
$\eps_0$ is given in Theorem \ref{th:epsreg} 
(see also Remark \ref{rem:wepregv}) 
and $C$ is independent of $t$ and $M$. 
\end{lemma}

\begin{proof}
Let $t_0\in (0,-1/4)$. 
Since $\card(\Sigma(t_0))=0$ if $\Sigma(t_0)= \emptyset$, 
it suffices to consider the case $\Sigma(t_0)\neq \emptyset$. 
Let $x_0\in \Sigma(t_0)$. 
We examine the $L^{q_c}$ norm of $\overline{u}$ near $(x_0,t_0)$.  
From the contraposition of Theorem \ref{th:epsreg} 
(see also Remark \ref{rem:wepregv}), 
it follows that 
\[
	\eps_0<(\rho/2)^{\frac{4}{p-1}-n} 
	\iint_{Q_{\rho/2}^+(x_0,t_0)} 
	( |\nabla \overline{u}|^2 + |\overline{u}|^{p+1} )  dxdt 
	\quad \mbox{ for }0<\rho\leq 2\delta_0.  
\]
Lemma \ref{lem:scaleL2Lqc} with translation gives 
\[
\begin{aligned}
	\iint_{Q_{\rho/2}^+(x_0,t_0)} 
	|\nabla \overline{u}|^2  dxdt 
	\leq 
	C\rho^{ n-\frac{4}{p-1}-\frac{4}{q_c} } 
	( 1 + M^{p-1} )^2 \|\overline{u} \|_{L^{q_c}(Q_\rho^+(x_0,t_0))}^2, 
\end{aligned}
\]
where $C>0$ is independent of $\rho$. 
The H\"older inequality also gives 
\[
\begin{aligned}
	&\iint_{Q_{\rho/2}^+(x_0,t_0)} 
	|\overline{u}|^{p+1}  dxdt \\
	&\leq 
	C\rho^{(n+2)(1-\frac{p+1}{q_c})} 
	\left( \iint_{Q_{\rho/2}^+(x_0,t_0)} 
	|\overline{u}|^{q_c} dxdt \right)^\frac{p-1}{q_c}
	\|\overline{u} \|_{L^{q_c}(Q_\rho^+(x_0,t_0))}^2 \\
	&\leq 
	C\rho^{(n+2)(1-\frac{p+1}{q_c}) + \frac{2(p-1)}{q_c}} 
	M^{p-1} \|\overline{u} \|_{L^{q_c}(Q_\rho^+(x_0,t_0))}^2. 
\end{aligned} 
\]
These estimates imply that 
\[
	\eps_0
	\leq 
	C (1+M^{p-1})^2 \rho^{-\frac{4}{q_c}} 
	\|\overline{u} \|_{L^{q_c}(Q_\rho^+(x_0,t_0))}^2. 
\]
Hence there exists a constant $C>0$ depending on $R$ 
and independent of $\rho$, $t_0$ and $M$ such that 
\[
	\eps_0^\frac{q_c}{2} 
	\leq 
	C (1+M^{p-1})^{q_c} \rho^{-2} 
	\int_{t_0-\rho^2}^{t_0} \int_{B_\rho^+(x_0)}
	|\overline{u} |^{q_c} dxdt 
	\quad \mbox{ for }0<\rho\leq 2\delta_0. 
\]

Let $S$ be any finite subset of $\Sigma(t_0)$. 
We write $S=\{x_1,\ldots,x_N\}$ and 
choose $0<\rho_0\leq \delta_0$ ($<1/2$) such that 
$\{ B_{\rho_0}^+(x_i)\}_{1\leq i\leq N}$ is pairwise disjoint. 
Then we have 
\[
	\eps_0^\frac{q_c}{2} 
	\leq 
	C (1+M^{p-1})^{q_c} \rho_0^{-2} 
	\int_{t_0-\rho_0^2}^{t_0} \int_{B_{\rho_0}^+(x_i)}
	|\overline{u} |^{q_c} dxdt
\]
for each $1\leq i\leq N$, and so we also have 
\[
\begin{aligned}
	N \eps_0^\frac{q_c}{2} 
	&\leq 
	C (1+M^{p-1})^{q_c} \rho_0^{-2} 
	\int_{t_0-\rho_0^2}^{t_0} 
	\left( \sum_{i=1}^N \int_{B_{\rho_0}^+(x_i)}
	|\overline{u} |^{q_c} dx \right) dt  \\
	&\leq 
	C (1+M^{p-1})^{q_c} \rho_0^{-2} 
	\int_{t_0-\rho_0^2}^{t_0} 
	\int_{\R^n_+} |\overline{u} |^{q_c} dx dt 
	\leq 
	C(M+M^p)^{q_c}, 
\end{aligned}
\]
where $C>0$ is independent of $\rho_0$, $t_0$, $M$ and $N$. 
Hence $\card(S)\leq C(M+M^p)^{q_c}\eps_0^{-q_c/2}$. 
Remark that the constant 
$C(M+M^p)^{q_c}\eps_0^{-q_c/2}$
does not depend on the choice of $S$.
Thus, the number of elements in any subset of $\Sigma(t_0)$ 
cannot exceed the constant, 
regardless of whether $\Sigma(t_0)$ 
contains an accumulation point or not. 
Therefore we can conclude that 
$\card(\Sigma(t_0))\leq C(M+M^p)^{q_c}\eps_0^{-q_c/2}$. 
The proof is complete. 
\end{proof}

\begin{remark}\label{rem:R1bu}
In the case \eqref{eq:R1def}, 
Lemma \ref{lem:exbul} also holds 
with $\cB_\rho$ and $\R^n_+$ replaced by 
$B_\rho$ and $\R^n$, respectively. 
In particular, the blow-up limit $\overline{u}$ exists. 
The analogs to the $\eps$-regularity (Theorem \ref{th:epsreg}), 
the equality in Lemma \ref{lem:SS} 
and the partial regularity (Lemma \ref{lem:parreg}) 
also hold for $\overline{u}$ in the case \eqref{eq:R1def}. 
\end{remark}

%%%%%%%%%%%%%%%%%%%%%%%%%%%%%%%%%%%%%%%%%%%%%%%%
%%%%%%%%%%%%%%%%%%%%%%%%%%%%%%%%%%%%%%%%%%%%%%%%
%%%%%%%%%%%%%%%%%%%%%%%%%%%%%%%%%%%%%%%%%%%%%%%%
\subsection{Proof of localized statement}
%%%%%%%%%%%%%%%%%%%%%%%%%%%%%%%%%%%%%%%%%%%%%%%%
%%%%%%%%%%%%%%%%%%%%%%%%%%%%%%%%%%%%%%%%%%%%%%%%
%%%%%%%%%%%%%%%%%%%%%%%%%%%%%%%%%%%%%%%%%%%%%%%%
We are now in a position to prove Theorem \ref{th:critical}. 

\begin{proof}[Proof of Theorem \ref{th:critical}]
We show this theorem by contradiction 
under the condition supposed in Subsection \ref{subsec:4uni}. 
We focus on the case \eqref{eq:R2def}, 
since \eqref{eq:R1def} is easier. 
From the lower bound \eqref{est:lower} 
together with \eqref{eq:uhatdef}, \eqref{eq:nabxihatu} and \eqref{eq:u_k}, 
it follows that 
\[
\begin{aligned}
	\eps_0 &< 
	\delta_k^{\frac{4}{p-1}-n} 
	\int_{-\delta_k^2}^0 
	\int_{\Phi(\Omega_{\delta_k})}
	( |\widehat \nabla \hat u |^2 
	+|\hat u|^{p+1} ) d\xi dt \\
	&=
	\int_{-1}^0 \int_{\delta_k^{-1}\Phi(\Omega_{\delta_k})}
	( |\widehat \nabla u_k|^2 
	+|u_k|^{p+1} ) dx dt, 
\end{aligned}
\]
where $\widehat \nabla \hat u
:= (\nabla'\hat u - (\partial_n\hat u) \nabla' f, \partial_n\hat u)$ 
and 
$\widehat \nabla u_k:=(\nabla'u_k - (\partial_n u_k) \nabla' f_k, \partial_n u_k)$.
Using $\|\nabla' f\|_{L^\infty(\R^{n-1})}\leq 1/2$ gives 
$\delta_k^{-1}\Phi(\Omega_{\delta_k}) \subset B_{\sqrt{2}}^+$. 
Hence by \eqref{eq:nabfkuni} 
and Lemma \ref{lem:exbul}, we see that  
\begin{equation}\label{est:lowerbar}
	\int_{-1}^0 \int_{B_{\sqrt{2}}^+} 
	(|\nabla \overline{u}|^2 + |\overline{u}|^{p+1}) dxdt 
	> \eps_0. 
\end{equation}
We will show that $\overline{u}\equiv 0$ a.e. in $\R^n_+\times(-1,0)$, 
contrary to \eqref{est:lowerbar}. 

By Lemma \ref{lem:SS}, it follows that 
\[
	\frac{d}{d\lambda} ( \lambda^\frac{2}{p-1} \overline{u}(\lambda x, \lambda^2 t) ) 
	= 
	2 \lambda^{\frac{2}{p-1}-1}\left( 
	\frac{\overline{u}(y, s)}{p-1} 
	+ \frac{y\cdot \nabla_y \overline{u}}{2} 
	+ s \overline{u}_s 
	\right)=0
\]
for $y=\lambda x$, $s=\lambda^2 t$ and $0<\lambda<1/\sqrt{-t}$ 
in the weak sense. Hence $\overline{u}$ is self-similar, and so
there exists a profile function $\overline{U} \in L^{q_c}(\R^n_+)$ 
such that  
\[
	\overline{u}(x,t)
	=(-t)^{-\frac{1}{p-1}} \overline{U}(z), 
	\quad 
	z:=\frac{x}{\sqrt{-t}},
\]
where $\overline{U}$ is a weak solution of the equation
\begin{equation}\label{eq:oUeqel}
	\Delta \overline{U}
	-\frac{1}{2}z \cdot \nabla \overline{U}
	-\frac{1}{p-1} \overline{U} 
	+|\overline{U}|^{p-1}\overline{U}=0, \quad z\in \R^n_+.
\end{equation}
We note that the method for proving 
self-similarity can also be found in \cite[Lemma 8.5.3]{LWbook} and \cite{St88,Wa08}.

We claim that there exist constants $\tilde R,C>1$ satisfying 
\begin{equation}\label{eq:asbu}
\left\{ 
\begin{aligned}
	&\overline{u}(\cdot,0)=0 &&\mbox{ a.e. in }\R^n_+, \\
	&|\overline{u}| \le C &&
	\mbox{ a.e. in } (\R^n_+ \setminus B_{\tilde R}) \times [-1/9,0].
\end{aligned}
\right.
\end{equation}
Recall from the proof of Lemma \ref{lem:exbul} that 
$u_k\to\overline{u}$ in 
$C([-1,0];W^{1,r}(\cB_\rho))$ 
for any $\rho>0$ with some $1<r<\min\{2,q_c/p\}$
as $k\to\infty$. 
Then, for $x_0 \in \R^n_+$ and $\delta>0$, we have 
\[
	\int_{B_1^+(x_0)}|\overline{u}(x,0)|^r dx 
	\leq 
	2^{r-1}\delta + 2^{r-1} \int_{B_1^+(x_0)} |u_k(x,0)|^r dx 
\]
for sufficiently large $k$. 
By returning to the original variables, we also have 
\[
\begin{aligned}
	\int_{B_1^+(x_0)} |u_k(x,0)|^r dx
	&=
	\delta_k^{\frac{2r}{p-1}-n}
	\int_{\Psi( B_{\delta_k}^+(\delta_k x_0) )}
	|u(x,0)|^r dx \\
	&\leq 
	\delta_k^{\frac{2r}{p-1}-n}
	\int_{\Psi( B^+_{(1+|x_0|) \delta_k}(0) )}|u(x,0)|^r dx \\
	&\leq 
	C(1+|x_0|)^{n-\frac{2r }{p-1}}
	\|u(\cdot,0)\|_{L^{q_c}(\Psi( B^+_{(1+|x_0|) \delta_k}(0) ) )}^r
	\to 0 
\end{aligned}
\]
as $k\to\infty$, where  
$u$ belongs to $C_\weak([-1/4,0];L^{q_c}(\Omega_{R/2}))$ 
in our situation, see \eqref{eq:u0Mbdd}. 
Hence we obtain the equality on the first line of \eqref{eq:asbu}.

We prove the inequality in \eqref{eq:asbu}. 
Let $\tilde \eps>0$. 
We claim that 
there exists a constant $\tilde R>1$ depending on $\tilde \eps$ and $M$ satisfying 
\begin{equation}
\label{est:ext}
	\iint_{Q_{1/2}^+(x_1,0)}
	(|\nabla \overline{u}|^{2}+|\overline{u}|^{p+1}) dxdt
	\le \tilde \eps 
\end{equation}
for any $x_1\in \R^n_+\setminus B_{\tilde{R}}^+(0)$, 
where $Q_{1/2}^+(x_1,0)= B_{1/2}^+(x_1)\times (-1/4,0)$. 
Indeed, for $\tilde \eps'>0$, 
we can choose $\tilde R'>1$ by Lemma \ref{lem:exbul} (iii) 
so large that 
\[
	\int^0_{-1}\int_{\R^n_+\setminus B_{\tilde R'-1}(0)}
	| \overline{u}|^{q_c} dxdt
	\le \tilde \eps'.
\]
From the H\"older inequality, it follows that 
\[
	\iint_{Q_1^+(x_1,0)}
	| \overline{u}|^{p+1} dxdt
	\leq 
	C \left( \iint_{Q_1^+(x_1,0)}
	|\overline{u}|^{q_c} dxdt \right)^\frac{p+1}{q_c} 
	\le C(\tilde \eps')^\frac{p+1}{q_c} 
\]
for any $x_1\in \R^n_+\setminus B_{\tilde R'}^+(0)$. 
On the other hand,  
Lemma \ref{lem:scaleL2Lqc} with $\rho=1$ yields 
\[
\begin{aligned}
	\iint_{Q_{1/2}^+(x_1,0)}
	|\nabla \overline{u}|^2 dxdt 
	&\leq 
	C (1 + M^{p-1})^2 
	\left( \iint_{Q_1^+(x_1,0)}
	|\overline{u}|^{q_c} dxdt \right)^\frac{2}{q_c}  \\
	&\leq 
	C (1 + M^{p-1})^2 (\tilde \eps')^\frac{2}{q_c}. 
\end{aligned}
\]
Thus choosing $\tilde \eps'$ small gives \eqref{est:ext}, 
and so the claim follows. 
Hence by translation 
and Theorem \ref{th:epsreg} (see also Remark \ref{rem:wepregv}), 
we see that  $|\overline{u}|\leq C\delta_0^{-2/(p-1)}$ 
in $B_{\delta_0}^+(x_1) \times [-1/9,0]$ for any 
$x_1\in \R^n_+ \setminus B_{\tilde R}^+(0)$, 
where $C$ and $\delta_0$ are constants independent of $x_1$. 
This proves the inequality in \eqref{eq:asbu}.

From \eqref{eq:asbu} 
and the backward uniqueness theorem \cite[Theorem 5.1]{ESS03}, 
it follows that 
\[
	\overline{u}\equiv 0 \quad \mbox{ on } 
	(\R^n_+ \setminus  B_{\tilde R})  \times [-1/9,0]. 
\]
In particular, $\overline{u}(x,-1/9)=9^{-1/(p-1)}\overline{U}(3x)=0$ 
for $x\in \R^n_+ \setminus  B^+_{\tilde R}$. 
Since $\overline{U}$ satisfies \eqref{eq:oUeqel} 
and is smooth except for a finite set by Lemma \ref{lem:parreg}, 
the unique continuation theorem 
for elliptic equations in a connected set 
(see \cite{Ar57} for instance) 
concludes that $\overline{U}\equiv 0$ in $\R^n_+$. 
This contradicts \eqref{est:lowerbar}, 
and hence the proof of the localized statement is complete. 
\end{proof}

%%%%%%%%%%%%%%%%%%%%%%%%%%%%%%%%%%%%%%%%%%%%%%%%
%%%%%%%%%%%%%%%%%%%%%%%%%%%%%%%%%%%%%%%%%%%%%%%%
%%%%%%%%%%%%%%%%%%%%%%%%%%%%%%%%%%%%%%%%%%%%%%%%
\subsection{Completion of proof}\label{subsec:pr}
%%%%%%%%%%%%%%%%%%%%%%%%%%%%%%%%%%%%%%%%%%%%%%%%
%%%%%%%%%%%%%%%%%%%%%%%%%%%%%%%%%%%%%%%%%%%%%%%%
%%%%%%%%%%%%%%%%%%%%%%%%%%%%%%%%%%%%%%%%%%%%%%%%
Finally, we prove Theorem \ref{th:main}.

\begin{proof}[Proof of Theorem \ref{th:main}]
If $\Omega$ is bounded, then Theorem \ref{th:main} 
immediately follows from Theorem \ref{th:critical}. 
In what follows, we consider the case where $\Omega$ is unbounded. 
To obtain a contradiction, suppose that 
\[
	\sup_{0<t<T} \|u(\cdot,t)\|_{L^{q_c}(\Omega)} \leq M
\]
for some $M>0$. Let $\eps>0$ and $a \in \Omega$. 
Then by the same argument as in the derivation of \eqref{est:ext}, 
there exists a constant $\tilde R>0$ depending on $\eps$, $a$ 
and $M$ satisfying 
\[
	\int_{T/2}^T \int_{\Omega_{1/2}(\tilde x)}
	( |u|^{p+1}+ |\nabla u|^{2}) dxdt \leq \eps 
\]
for any $\tilde x\in \Omega \setminus B_{\tilde R}(a)$. 
Therefore, in the same way as in the proof of \eqref{eq:asbu}, 
we see that $u$ is bounded 
on $(\Omega \setminus B_{\tilde R'}(a))\times (T/3,T)$ 
for some $\tilde R'>0$.
This implies that there exists at least one blow-up point 
$a' \in \overline{\Omega_{\tilde R'}(a)}$.
Hence Theorem \ref{th:critical} shows that 
\[
	\limsup_{t\to T}\|u(\cdot,t)\|_{L^{q_c}(\Omega_r(a'))}=\infty 
\]
for any $r>0$, a contradiction. 
The proof of Theorem \ref{th:main} is complete. 
\end{proof}

%%%%%%%%%%%%%%%%%%%%%%%%%%%%%%%%%%%%%%%%%%%%%%%%
%%%%%%%%%%%%%%%%%%%%%%%%%%%%%%%%%%%%%%%%%%%%%%%%
%%%%%%%%%%%%%%%%%%%%%%%%%%%%%%%%%%%%%%%%%%%%%%%%
\appendix
%%%%%%%%%%%%%%%%%%%%%%%%%%%%%%%%%%%%%%%%%%%%%%%%
%%%%%%%%%%%%%%%%%%%%%%%%%%%%%%%%%%%%%%%%%%%%%%%%
%%%%%%%%%%%%%%%%%%%%%%%%%%%%%%%%%%%%%%%%%%%%%%%%
\section{Regularity estimates}\label{sec:reges}
%%%%%%%%%%%%%%%%%%%%%%%%%%%%%%%%%%%%%%%%%%%%%%%%
%%%%%%%%%%%%%%%%%%%%%%%%%%%%%%%%%%%%%%%%%%%%%%%%
%%%%%%%%%%%%%%%%%%%%%%%%%%%%%%%%%%%%%%%%%%%%%%%%
We give some parabolic regularity estimates for solutions of \eqref{eq:fujitaeq} 
and a gradient estimate for 
the blow-up limit obtained in Lemma \ref{lem:exbul}. 
Let $n\geq3$, $p>p_S$ and 
$\Omega$ be any $C^{2+\alpha}$ domain 
in $\R^n$ with $0\in \overline{\Omega}$. 
Fix $R>0$ such that either 
\eqref{eq:R1def} or \eqref{eq:R2def} holds. 
Let $u$ satisfy \eqref{eq:fujitaeq} and \eqref{eq:Mdef2}. 
Define $\hat u$ by \eqref{eq:uhatdef}. 
For $\rho>0$, 
we take $0<\delta'<1/2$ such that 
$\Psi(B_{3\rho \delta'}^+)\subset \Omega_{R/2}$ and 
$\Psi(B_{3\rho \delta'})\subset B_{R/2}$. 
We first give parabolic regularity estimates. 
Remark that we mainly focus on the case \eqref{eq:R2def}, 
since \eqref{eq:R1def} is easier.

\begin{lemma}\label{lem:pararegd}
Assume \eqref{eq:R2def}. 
Let $1\le l<\infty$ and $1\le r\leq q_c/p$. 
Then there exists a constant $C>0$ such that
\[
	\|\hat u_t\|_{L^l(-\delta^2,0; L^r(B_{\rho \delta}^+))} 
	+
	\|\nabla^2 \hat u\|_{L^l(-\delta^2,0; L^r(B_{\rho \delta}^+))} 
	\leq  C\delta^{\frac{2}{l}+n(\frac{1}{r}-\frac{p}{q_c})}(M+M^p)
\]
for any $0<\delta<\delta'$, 
where $C$ depends on $R$ and $\rho$ 
and is independent of $\delta$. 
\end{lemma}

\begin{proof}
By \eqref{eq:uhatdef}, we have 
\[
	|\nabla^2 \hat u(x,t)| \leq C 
	(|(\nabla u)(\Psi(x),t)|
	+ |(\nabla^2 u)(\Psi(x), t)|), 
\]
where $C>0$ depends on $\|\nabla f\|_{L^\infty(\R^{n-1})}$ 
and $\|\nabla^2 f\|_{L^\infty(\R^{n-1})}$. 
Remark that the choice of $\delta'$ 
guarantees $\Psi(B_{\rho\delta}^+)\subset \Omega_{3R/4}$, 
and so Proposition \ref{pro:gradoricri} 
is applicable in $\Psi(B_{\rho\delta}^+)$. 
Then by $r\leq q_c/p<q_*$, 
the H\"older inequality 
in the Lorentz spaces (see \cite[Proposition 2.1]{KY99} for instance)
and Proposition \ref{pro:gradoricri}, we see that 
\begin{equation}\label{eq:delhu2}
\begin{aligned}
	&\|\nabla u ( \Psi(\cdot),  \cdot)
	\|_{L^l(-\delta^2,0; L^r(B_{\rho\delta}^+))}  \\
	&\leq 
	C\delta^{n(\frac{1}{r}-\frac{1}{q_*})} 
	\|\nabla u(\Psi(\cdot), \cdot)
	\|_{L^l(-\delta^2,0; L^{q_*,\infty}(B_{\rho\delta}^+))} \\
	&\leq 
	C\delta^{\frac{2}{l} + n(\frac{1}{r}-\frac{1}{q_*})}
	(M+M^p) 
	\leq 
	C\delta^{\frac{2}{l} + n(\frac{1}{r}-\frac{p}{q_c})}
	(M+M^p). 
\end{aligned}
\end{equation}

We estimate 
$\| \nabla^2 u(\Psi( \cdot),  \cdot)\|_{L^l(-\delta^2,0; L^r(B_{\rho\delta}^+))}$. 
Let $\phi\in C^\infty_0(\R^n)$ satisfy 
$0\leq \phi\leq 1$ in $\R^n$, 
$\phi=0$ in $\R^n\setminus B_{3\rho}$ and  
$\phi=1$ in $B_{2\rho}$. 
Set $\tilde \phi(x):= \phi(\delta^{-1} \Phi (x))$ 
and $v(x,t):=u(x,t) \tilde \phi(x)$. 
We prepare a $C^{2+\alpha}$ domain $\cD$ 
satisfying $\Psi(B_{3\rho\delta'}^+)\subset \cD \subset \Omega_R$ 
to avoid technicalities due to the 
corner of $\partial B_{3\rho\delta'}^+$. 
Then $v$ satisfies 
\[
\left\{ 
\begin{aligned}
	&v_t - \Delta v
	= \tilde \phi |u|^{p-1} u -2\nabla \tilde \phi \cdot \nabla u
	-u\Delta \tilde \phi 
	&&\mbox{ in }\cD\times (-1,0), \\
	&v=0
	&&\mbox{ on }\partial \cD\times (-1,0).
\end{aligned}
\right. 
\]
By the same computation as in the derivation of \eqref{eq:uGompu}, 
we have 
\[
\begin{aligned}
	u(x,t) 
	&=\int_{\cD} G_\cD(x,y,t+2\delta^2) 
	\tilde \phi(y) u(y,-2\delta^2) dy \\
	&\quad 
	+ \int_{-2\delta^2}^t \int_\cD G_\cD(x,y,t-s) 
	\tilde \phi |u|^{p-1} u  dyds \\
	&\quad 
	- \int_{-2\delta^2}^t \int_\cD G_\cD(x,y,t-s) 
	(2\nabla\tilde \phi \cdot \nabla u
	+u \Delta \tilde \phi ) dyds
\end{aligned}
\]
for $x\in \Psi(B_{\rho\delta}^+)$ and $-2\delta^2<t<0$, 
and so 
\[
\begin{aligned}
	|\nabla^2 u(x,t)| 
	&\leq C \int_{\R^n} K_2(x-y,t+2\delta^2) |u(y,-2\delta^2)| 
	\chi_{\Omega_R \cap \Psi(B_{3\rho\delta})}(y) dy \\
	&\quad 
	+ \left| \nabla^2 \int_{-2\delta^2}^t \int_\cD G_\cD(x,y,t-s) 
	\tilde \phi |u|^{p-1} u  dyds\right| \\
	&\quad 
	+ C \int_{-2\delta^2}^t \int_{\R^n} K_2 
	\left( 
	\frac{|u|}{\delta^2}+ \frac{|\nabla u|}{\delta} 
	\right) 
	\chi_{\Omega_R \cap 
	\Psi(  \overline{B_{3\rho\delta}} 
	\setminus B_{2\rho\delta} ) }  dyds \\
	&=:C V_1(x,t) + |\nabla^2 V_2(x,t)| + C V_3(x,t)
\end{aligned}
\]
for $x\in \Psi(B_{\rho\delta}^+)$ and $-2\delta^2<t<0$, 
where $K_2$ is defined by \eqref{eq:Kjdef}.

First, we estimate $V_1$. 
By $\|\nabla'f\|_{L^\infty(\R^{n-1})}\leq 1/2$ 
in \eqref{eq:R2def}, we have 
\[
	|\Psi(x)-\Psi(y)|^2 
	\geq 
	|x'-y'|^2 +  \frac{1}{2}(x_n-y_n)^2 - (f(x')-f(y'))^2 
	\geq \frac{1}{4}  |x-y|^2. 
\]
This together with the change of variables shows that 
\begin{equation}\label{eq:U1Psixtsc2}
\begin{aligned}
	&V_1(\Psi( x),  t) \\
	&=
	\int_{\R^n} K_2(\Psi( x)-\Psi(y), t+2\delta^2 )  |u(\Psi( y),-2\delta^2)| 
	\chi_{\Phi(\Omega_R) \cap B_{3\rho\delta}} dy \\
	&\leq 
	\int_{\R^n} K_2((x-y)/2, t+2\delta^2 )
	|u(\Psi(y),-2\delta^2)| 
	\chi_{\Phi(\Omega_R) \cap B_{3\rho\delta}} dy. 
\end{aligned}
\end{equation}
Then Young's inequality gives
\[
\begin{aligned}
	\| V_1(\Psi(\cdot), t)\|_{L^{q_c}(B_{\rho\delta}^+)} 
	&\leq 
	C (t+2\delta^2)^{-1}
	\| u(\Psi(\cdot),-2\delta^2) 
	\|_{L^{q_c}(\Phi(\Omega_R) \cap \Psi(B_{3\rho\delta}))} 
	\\
	&\leq 
	C \delta^{-2} M
\end{aligned}
\]
for $-\delta^2<t<0$. 
Therefore, the H\"older inequality shows that 
\begin{equation}\label{eq:Uk1concl2}
\begin{aligned}
	&\| V_1(\Psi( \cdot),  \cdot)
	\|_{L^l(-\delta^2,0; L^r(B_{\rho\delta}^+))} \\
	&\leq 
	C\delta^{n(\frac{1}{r}-\frac{1}{q_c})} 
	\| V_1(\Psi( \cdot), \cdot)
	\|_{L^l(-\delta^2,0; L^{q_c}(B_{\rho\delta}^+))} 
	\leq
	C\delta^{\frac{2}{l}+n(\frac{1}{r}-\frac{p}{q_c}) } M. 
\end{aligned}
\end{equation}

Let us next estimate $\nabla^2 V_2$. 
From the H\"older inequality, the change of variables 
$y=\Psi(x)$ with $dy=dx$ and the choice of $\cD$, 
it follows that 
\[
\begin{aligned}
	\| \nabla^2 V_2(\Psi( \cdot),  \cdot)
	\|_{L^l(-\delta^2,0; L^r(B_{\rho\delta}^+))}  
	&\leq 
	C \delta^{n(\frac{1}{r}-\frac{p}{q_c})}
	\| \nabla^2 V_2(\Psi( \cdot),  \cdot)
	\|_{L^l(-\delta^2,0; L^\frac{q_c}{p}(B_{\rho\delta}^+))} \\
	&\leq C \delta^{n(\frac{1}{r}-\frac{p}{q_c})} 
	\| \nabla^2 V_2\|_{L^l(-\delta^2,0; L^\frac{q_c}{p}(\cD))}. 
\end{aligned}
\]
We observe that $V_2$ is a solution of 
\[
\left\{ 
\begin{aligned}
	&(V_2)_t - \Delta V_2
	= \tilde \phi |u|^{p-1} u
	&&\mbox{ in }\cD\times (-2\delta^2,0), \\
	&V_2=0
	&&\mbox{ on }\partial \cD\times (-2\delta^2,0), \\
	&V_2(\cdot,-2\delta^2)=0 
	&&\mbox{ in }\cD. 
\end{aligned}
\right. 
\]
Since $\cD$ is a bounded $C^{2+\alpha}$ domain, 
$\cD$ is also a uniformly regular domain of class $C^2$. 
Therefore, we can apply 
the maximal regularity for inhomogeneous heat equations 
(see \cite[Remark 51.5]{QSbook2} and 
\cite[Theorem 7.11]{DHP03} for instance), and so 
\[
\begin{aligned}
	\| \nabla^2 V_2 \|_{L^l(-2\delta^2,0; L^\frac{q_c}{p}(\cD))}
	&\leq C \|\tilde \phi |u|^{p-1} u\|_{L^l(-2\delta^2,0; L^\frac{q_c}{p}(\cD))} \\
	&\leq 
	C \|u\|_{L^{pl}(-2\delta^2,0; L^{q_c}(\Omega_R))}^p 
	\leq 
	C \delta^{\frac{2}{l}} M^p. 
\end{aligned}
\]
Thus, 
\begin{equation}\label{eq:Uk2concl2}
	\| \nabla^2 V_2(\Psi( \cdot),  \cdot)
	\|_{L^l(-\delta^2,0; L^r(B_{\rho\delta}^+))}   
	\leq 
	C  \delta^{\frac{2}{l} + n(\frac{1}{r}-\frac{p}{q_c}) } M^p.  
\end{equation}

Finally, we consider $V_3$. 
Again in the same way as in \eqref{eq:U1Psixtsc2}, we have 
\[
\begin{aligned}
	V_3(\Psi(x), t)  
	&\leq 
	\int_{-2\delta^2}^t \int_{\R^n} K_2((x-y)/2, t-s ) \\
	&\qquad \times 
	(\delta^{-2}|u|+ \delta^{-1}|\nabla u|)(\Psi(y), s)
	\chi_{\Phi(\Omega_R) \cap ( \overline{B_{3\rho\delta}} 
	\setminus B_{2\rho\delta} ) }
	dy ds. 
\end{aligned}
\]
for $x\in B_{\rho\delta}^+$ and $-2\delta^2<t<0$. 
We observe that 
there exists $C>0$ satisfying 
\[
	K_2((x-y)/2, t-s)
	\leq 
	C\delta^{-n-2} 
\]
for $x\in B_{\rho\delta}^+$, 
$y\in \overline{B_{3\rho\delta}} \setminus B_{2\rho\delta}$ 
and $-2\delta^2<s<t<0$. 
Then by the change of variables, the H\"older inequality, 
$\Omega_R \cap  \Psi( B_{3\rho\delta}) \subset \Omega_{3R/4}$
and Proposition \ref{pro:gradoricri}, we see that 
\[
\begin{aligned}
	&V_3(\Psi(x), t) \\
	&\leq 
	C \delta^{-n-2}\int_{-2\delta^2}^t 
	\int_{ \Phi(\Omega_R) \cap  B_{3\rho\delta} } 
	(\delta^{-2}|u(\Psi(y),s)| 
	+ \delta^{-1}|\nabla u(\Psi(y), s)|) dy ds \\
	&\leq C \delta^{-n-2} 
	( M\delta^{n(1-\frac{1}{q_c})} 
	+(M+M^p) \delta^{n(1-\frac{1}{q_*})+1} ) 
	\leq 
	C \delta^{-\frac{np}{q_c}} (M+M^p)
\end{aligned}
\]
for $x\in B_{\rho\delta}^+$ and $-2\delta^2<t<0$. 
Hence we obtain 
\[
\begin{aligned}
	\|V_3(\Psi(\cdot), \cdot)
	\|_{L^l(-\delta^2,0; L^r(B_{\rho\delta}^+))}  
	\leq 
	C \delta^{\frac{2}{l} + n(\frac{1}{r}-\frac{p}{q_c})} (M+M^p). 
\end{aligned}
\]
By combining this inequality, 
\eqref{eq:delhu2}, \eqref{eq:Uk1concl2} and \eqref{eq:Uk2concl2}, 
we obtain the desired estimate for $\nabla^2 \hat u$. 
Then the desired estimate for $\hat u_t$ can be obtained by 
using the equation in \eqref{eq:hatueq}. 
\end{proof}

\begin{lemma}\label{lem:parareg}
Assume either \eqref{eq:R1def} or \eqref{eq:R2def}. 
Let $1\le l<\infty$ and $1\le r\leq q_c/p$. 
Then there exists a constant $C>0$ depending on $R$ such that
\[
	\|u_t\|_{L^l(-1/4,0; L^r(\Omega_{R/2}))} 
	+ \|\nabla^2 u\|_{L^l(-1/4,0; L^r(\Omega_{R/2}))} 
	\leq  C(M+M^p). 
\]
\end{lemma}

\begin{proof}
By easy modifications of Lemma \ref{lem:pararegd}, 
we can see that 
\[
	\|\nabla^2 u\|_{L^l(-1/4,0; L^r(\Omega_{R/2}))} \leq  C(M+M^p). 
\]
Then by the equation in \eqref{eq:fujitaeq}, 
we obtain the desired inequality. 
\end{proof}

Let us next show a gradient estimate for 
the blow-up limit $\overline{u}$ obtained in Lemma \ref{lem:exbul}. 
To estimate $\nabla \overline{u}$, 
we derive a localized integral equation for $\overline{u}$.

\begin{lemma}\label{lem:ubarint}
Assume \eqref{eq:R2def}. 
Let $\phi\in C^\infty_0(\R^n)$ satisfy 
$0\leq \phi\leq 1$ in $\R^n$, 
$\phi=0$ in $\R^n\setminus B_{3/5}$ and  
$\phi=1$ in $B_{4/5}$. 
Set $\phi_\rho(x):=\phi(x/\rho)$ for $0<\rho<1$. 
Then $\overline{u}$ satisfies 
\[
\begin{aligned}
	\overline{u} (x,t) &=\int_{\R^n_+} G_{\R^n_+}(x,y,\rho^2/4) 
	\phi_\rho(y) \overline{u} (y,t-\rho^2/4) dy  \\
	&\quad 
	+ \int_{t-\rho^2/4}^t \int_{\R^n_+} G_{\R^n_+}(x,y,t-s) 
	( \phi_\rho |\overline{u}|^{p-1} \overline{u}  
	+\overline{u} \Delta \phi_\rho ) dyds \\
	&\quad 
	+2 \int_{t-\rho^2/4}^t \int_{\R^n_+}  
	\nabla_y G_{\R^n_+}(x,y,t-s) \cdot 
	\nabla \phi_\rho(y)   \overline{u}(y,s) 
	dy ds 
\end{aligned}
\]
for a.e. $x\in B_{\rho/2}^+$ and $-\rho^2/4<t<0$. 
\end{lemma}

\begin{proof}
Let us convert our problem to the one in $\R^n_+$. 
Let $\psi\in C^\infty_0(\R^n)$ satisfy 
$0\leq \psi\leq 1$ in $\R^n$, 
$\psi=0$ in $\R^n\setminus B_{4R/5}$ and  
$\psi=1$ in $B_{3R/5}$. 
Set 
$\psi_k(x):=\psi(\Psi(\delta_k x))$ 
and  $v_k:=\phi_\rho \psi_k u_k $. 
Note that 
$\psi_k=0$ in $\R^n\setminus \delta_k^{-1}\Phi(B_{4R/5})$ 
and $\psi_k=1$  in $\delta_k^{-1}\Phi(B_{3R/5})$.  
Then by \eqref{eq:ukresceq}, we see that 
\[
\left\{ 
	\begin{aligned}
	&(v_k)_t-\Delta v_k = \phi_\rho \psi_k |u_k|^{p-1} u_k
	-\psi_k u_k \Delta \phi_\rho -2 \psi_k\nabla \phi_\rho\cdot \nabla u_k 
	+\cR_k\\
	&\quad \mbox{ in }\R^n_+\times(-\delta_k^{-2},0), \\
	&v_k=0 \quad \mbox{ on }\partial \R^n_+\times(-\delta_k^{-2},0), 
	\end{aligned}
\right.
\]
where 
\[
\begin{aligned}
	\cR_k &:= 
	\phi_\rho \psi_k 
	(  -2\nabla'(\partial_{x_n} u_k) \cdot \nabla'f_k 
	+(\partial_{x_n}^2 u_k) |\nabla'f_k|^2 
	-(\partial_{x_n} u_k) \Delta'f_k ) \\
	&\quad 
	- ( 
	2u_k \nabla \psi_k\cdot \nabla \phi_\rho 
	+ \phi_\rho u_k \Delta \psi_k + 2 \phi_\rho \nabla \psi_k \cdot \nabla u_k ). 
\end{aligned}
\]
Thus, 
\[
\begin{aligned}
	u_k(x,t) 
	&=\int_{\R^n_+} G_{\R^n_+}(x,y,\rho^2/4) 
	\rho_\rho(y) \psi_k(y) u_k(y,t-\rho^2/4) dy  \\
	&\quad 
	+ \int_{t-\rho^2/4}^t \int_{\R^n_+} G_{\R^n_+}(x,y,t-s) 
	\phi_\rho \psi_k |u_k|^{p-1} u_k  dyds \\
	&\quad 
	- \int_{t-\rho^2/4}^t \int_{\R^n_+}  
	G_{\R^n_+}
	( \psi_k u_k \Delta \phi_\rho +2 \psi_k\nabla \phi_\rho\cdot \nabla u_k 
	+ \cR_k) 
	dy ds \\
	&=: V_1^k(x,t) + V_2^k(x,t) + V_3^k(x,t)
\end{aligned}
\]
for $x\in B^+_{\rho/2}$, $-\rho^2/4<t<0$ 
and $k\geq k_\rho$, 
where $k_\rho$ is given by the first part of Subsection \ref{subsec:4uni}. 
Lemma \ref{lem:exbul} (ii) shows that 
$u_k(\cdot,t)$ converges to 
$\overline{u}(\cdot,t)$ 
in $L^1(B_{\rho/2}^+)$ for each $-\rho^2/4<t<0$ as $k \to \infty$.

We show that the right-hand side of the integral equation 
for a subsequence of $u_k$ still denoted by $u_k$ 
converges to the one in the following integral equation: 
\begin{equation}\label{eq:ubarinpre}
\begin{aligned}
	\overline{u} (x,t) &=\int_{\R^n_+} G_{\R^n_+}(x,y,\rho^2/4) 
	\phi_\rho(y) \overline{u} (y,t-\rho^2/4) dy  \\
	&\quad 
	+ \int_{t-\rho^2/4}^t \int_{\R^n_+} G_{\R^n_+}(x,y,t-s) 
	\phi_\rho |\overline{u}|^{p-1} \overline{u}  dyds \\
	&\quad 
	- \int_{t-\rho^2/4}^t \int_{\R^n_+}  
	G_{\R^n_+}(x,y,t-s) 
	( \overline{u} \Delta \phi_\rho +2 \nabla \phi_\rho\cdot \nabla \overline{u} ) 
	dy ds  \\
	&=: \overline{V}_1(x,t) + \overline{V}_2(x,t) + \overline{V}_3(x,t)
\end{aligned}
\end{equation}
for a.e. $x\in B^+_{\rho/2}$ and $-\rho^2/4<t<0$. 
For $V_1^k$, from \eqref{eq:derivKjes}, 
$0\leq \psi_k\leq 1$, the H\"older inequality and 
Lemma \ref{lem:exbul} (ii) and (iii), 
it follows that 
\[
\begin{aligned}
	\|V_1^k (\cdot,t) - \overline{V}_1(\cdot,t)\|_{L^1(B_{\rho/2}^+)} 
	&\leq 
	C\rho^{-n} \int_{B_{\rho/2}^+} \int_{\R^n_+} 
	\phi_\rho |\psi_k u_k 
	- \overline{u}| dydx  \\
	&\leq 
	C \| u_k(\cdot,t-\rho^2/4) - \overline{u}(\cdot,t-\rho^2/4) 
	\|_{L^1(B_\rho^+)} \\
	&\quad 
	+ C\| (\psi_k-1) \overline{u}(\cdot,t-\rho^2/4) \|_{L^1(B_\rho^+)}
	\to 0
\end{aligned}
\]
as $k\to\infty$. 
For $V_2^k$, by computations similar to that of $V_1^k$ 
with Young's inequality, we see that, as $k\to\infty$, 
\[
\begin{aligned}
	\|V_2^k (\cdot,t) - \overline{V}_2(\cdot,t)\|_{L^1(B_{\rho/2}^+)} 
	&\leq 
	C \int_{t-\rho^2/4}^t \| \psi_k |u_k|^{p-1}u_k 
	-|\overline{u}|^{p-1} \overline{u} \|_{L^1(B_\rho^+)} ds \\
	&\leq 
	C \int_{t-\rho^2/4}^t \| 
	(|u_k|^{p-1} + |\overline{u}|^{p-1}) |u_k-\overline{u}| \|_{L^1(B_\rho^+)} \\
	&\quad 
	+C \int_{t-\rho^2/4}^t 
	\| (\psi_k-1) |\overline{u}|^p ) \|_{L^1(B_\rho^+)} ds \to0. 
\end{aligned}
\]

For $V_3^k$, we focus on the most subtle term 
$\phi_\rho \psi_k \nabla'(\partial_{x_n} u_k) \cdot \nabla'f_k $
in $\cR_k$, that is, we prove that 
\[
	\tilde V_3^k (\cdot,t)
	:= 
	\int_{t-\rho^2/4}^t \int_{\R^n_+}  
	G_{\R^n_+}(\cdot,y,t-s) 
	\phi_\rho \psi_k \nabla'(\partial_{x_n} u_k) \cdot \nabla'f_k 
	dy ds
	\to 0
\]
in $L^1(B_{\rho/2}^+)$ for each $-\rho^2/4<t<0$ as $k \to \infty$. 
From integration by parts and \eqref{eq:derivKjes}, 
it follows that 
\[
\begin{aligned}
	|\tilde V_3^k(x,t) |
	&\leq 
	C\int_{t-\rho^2/4}^t \int_{\R^n_+}  
	K_1(x-y,t-s) 
	\phi_\rho \psi_k |\partial_{x_n} u_k| |\nabla'f_k|  dy ds  \\
	&\quad 
	+ C\int_{t-\rho^2/4}^t \int_{\R^n_+}  
	K_0 |\nabla'(\phi_\rho \psi_k)| | \partial_{x_n} u_k| |\nabla'f_k| 
	dy ds \\
	&\quad 
	 + C\int_{t-\rho^2/4}^t \int_{\R^n_+}  
	K_0 \phi_\rho \psi_k | \partial_{x_n} u_k| |\Delta'f_k| dy ds. 
\end{aligned}
\]
By Young's inequality, $|\nabla \phi_\rho|\leq C$, 
$|\nabla \psi_k|\leq C \delta_k \leq C$, 
H\"older inequality 
and $\int_{t-\rho^2/4}^t (t-s)^{-(1/2)\cdot(4/3)} ds\leq C$, 
we have 
\[
\begin{aligned}
	\|\tilde V_3^k(\cdot,t) \|_{L^1(B_{\rho/2}^+)}
	&\leq 
	C\left( 
	\int_{t-\rho^2/4}^t 
	\| |\nabla u_k(\cdot,s)| |\nabla'f_k| 
	\|_{L^1(B_\rho^+)}^4 ds \right)^{1/4} \\
	&\quad 
	+ C\int_{t-\rho^2/4}^t 
	\| | \nabla u_k(\cdot,s)| ( |\nabla'f_k| + |\Delta'f_k| )
	\|_{L^1(B_\rho^+)} ds. 
\end{aligned}
\]
Hence from 
the H\"older inequality 
for the Lorentz spaces, 
\eqref{eq:unifq*}, \eqref{eq:nabfkuni}, \eqref{eq:delfkuni} 
and the Lebesgue dominated convergence theorem, 
it follows that 
\[
	\|\tilde V_3^k(\cdot,t) \|_{L^1(B_{\rho/2}^+)} 
	\leq 
	C(M+M^p) \| |\nabla' f_k| 
	+ |\Delta'f_k|  \|_{L^\frac{q_*}{q_*-1}(B_\rho^+)} 
	\to 0
\]
for each $-\rho^2/4<t<0$ as $k \to \infty$. 
The other terms in $V_3^k$ can be handled more easily. 
Hence we obtain \eqref{eq:ubarinpre}. 
This implies the desired equality. 
\end{proof}

We give a gradient estimate for $\overline{u}$. 

\begin{lemma}\label{lem:scaleL2Lqc}
Assume \eqref{eq:R2def}. 
Then there exists a constant $C>0$ independent of $\rho$ such that 
\[
	\|\nabla \overline{u}\|_{L^2(Q_{\rho/2}^+)}
	\leq 
	C\rho^{ \frac{n}{2}-\frac{2}{p-1}-\frac{2}{q_c} } 
	(1 + M^{p-1}) 
	\|\overline{u} \|_{L^{q_c}(Q_\rho^+)}
\]
for any $0<\rho\leq 1$. 
\end{lemma}

\begin{proof}
By Lemma \ref{lem:ubarint}, 
$\nabla \phi_\rho(x)=\rho^{-1}\nabla\phi(x/\rho)$, 
$\Delta \phi_\rho(x)=\rho^{-2}\Delta \phi(x/\rho)$ and 
similar computations to \eqref{eq:U123def}, 
there exists $C>0$ 
independent of $\rho$ such that 
\[
\begin{aligned}
	|\nabla \overline{u}| 
	&\leq 
	C \int_{B_\rho^+} K_1(x-y,\rho^2/4) |\overline{u}(y,t-\rho^2/4)| dy \\
	&\quad 
	+ C \int_{t-\rho^2/4}^t \int_{B_\rho^+} K_1(x-y,t-s) 
	|\overline{u}(y,s)|^p  dyds \\
	&\quad 
	+ C \int_{t-\rho^2/4}^t \int_{\R^n_+} 
	\left( \frac{K_1}{\rho^2} + \frac{K_2}{\rho} \right) 
	|\overline{u}| \chi_{\overline{B_{4\rho/5}}\setminus B_{3\rho/5}} dy ds \\
	&=: CW_1+CW_2+C\rho^{-2}W_3+C\rho^{-1}W_4
\end{aligned}
\]
for a.e. $x\in B_{\rho/2}^+$ and $-\rho^2/4<t<0$, 
where $K_1$ and $K_2$ are given by \eqref{eq:Kjdef}.

By the H\"older inequality, we have 
\[
\begin{aligned}
	&\begin{aligned}
	W_1(x,t) &\leq 
	\|\overline{u}(\cdot,t-\rho^2/4)\|_{L^{q_c}(B_\rho^+)}
	\left( \int_{B_\rho^+} K_1(x-y,\rho^2/4)^\frac{q_c}{q_c-1} dy 
	\right)^{1-\frac{1}{q_c}} \\
	&\leq 
	C \rho^{-1-\frac{n}{q_c}} 
	\|\overline{u}(\cdot,t-\rho^2/4)\|_{L^{q_c}(B_\rho^+)}, 
	\end{aligned} \\
	&\begin{aligned}
	\|W_1\|_{L^2(Q_{\rho/2}^+)}
	&\leq 
	C \rho^{-\frac{n}{q_c}+\frac{n}{2}-\frac{2}{q_c}}
	\left( \int_{-\rho^2/4}^0 \int_{L^{q_c}(B_\rho^+)} 
	|\overline{u}(x,t-\rho^2/4)|^{q_c} dx dt 
	\right)^\frac{1}{q_c} \\
	&\leq 
	C \rho^{\frac{n}{2}-\frac{2}{p-1}-\frac{2}{q_c}}
	\| \overline{u} \|_{L^{q_c}(Q_\rho^+)}. 
	\end{aligned}
\end{aligned}
\]

We estimate $W_2$. We consider the cases $q_c/p\geq2$ and $q_c/p<2$, 
respectively. 
If $q_c/p\geq2$, then the H\"older inequality gives 
\[
\begin{aligned}
	\|W_2\|_{L^2(Q_{\rho/2}^+)}
	\leq C \rho^{\frac{n+2}{2}(1-\frac{2p}{q_c})}
	\|W_2\|_{L^\frac{q_c}{p}(Q_{\rho/2}^+)}. 
\end{aligned}
\]
From the same argument to prove the $L^{q_c/p}$-$L^{q_c/p}$ estimate, 
it follows that 
\[
\begin{aligned}
	\|W_2(\cdot,t)\|_{L^\frac{q_c}{p}(B_{\rho/2}^+)}
	&\leq 
	C \int_{t-\rho^2/4}^t (t-s)^{-\frac{1}{2}}
	\| |\overline{u}(\cdot,s)|^p \chi_{B_\rho^+} 
	\|_{L^\frac{q_c}{p}(\R^n)} ds \\
	&=
	C \int_0^{\rho^2/4} \tau^{-\frac{1}{2}}
	\| \overline{u}(\cdot,t-\tau) \|_{L^{q_c}(B_\rho^+)}^p  d\tau. 
\end{aligned}
\]
Thus, by $(-\tau-\rho^2/4,-\tau)\subset (-\rho^2,0)$ for $0<\tau<\rho^2/4$, 
we have 
\[
\begin{aligned}
	\|W_2\|_{L^\frac{q_c}{p}(Q_{\rho/2}^+)}
	&\leq 
	C\int_0^{\rho^2/4} \tau^{-\frac{1}{2}} 
	\left( \int_{-\rho^2/4}^0 \int_{B_\rho^+} 
	|\overline{u}(\cdot,t-\tau)|^{q_c} dx dt \right)^\frac{p}{q_c} d\tau  \\
	&\leq 
	C\int_0^{\rho^2/4} \tau^{-\frac{1}{2}} 
	\left( \int_{-\rho^2}^0 \int_{B_\rho^+} 
	|\overline{u}(\cdot,s)|^{q_c} dx ds \right)^\frac{p}{q_c} d\tau \\
	&\leq 
	C\rho^{1+\frac{2(p-1)}{q_c}} 
	M^{p-1} \|\overline{u} \|_{L^{q_c}(Q_\rho^+)}, 
\end{aligned}
\]
and so 
\[
\begin{aligned}
	\|W_2\|_{L^2(Q_{\rho/2}^+)}
	&\leq 
	C\rho^{ \frac{n}{2}-\frac{2}{p-1}-\frac{2}{q_c} } 
	M^{p-1} \|\overline{u} \|_{L^{q_c}(Q_\rho^+)}. 
\end{aligned}
\]

If $q_c/p<2$, then the same argument to prove the 
$L^{q_c/p}$-$L^2$ estimate yields 
\[
\begin{aligned}
	\|W_2(\cdot,t)\|_{L^2(B_\rho^+)} 
	&\leq 
	C\int_{t-\rho^2/4}^t 
	(t-s)^{-\frac{1}{2}-\frac{n}{2}(\frac{p}{q_c}-\frac{1}{2}) }
	\| |\overline{u}(\cdot,s)|^p \chi_{B_\rho^+} \|_{L^\frac{q_c}{p}(\R^n)} ds\\
	&\leq 
	C\int_\R
	|t-s|^{-1+\gamma }
	\| \overline{u}(\cdot,s) \|_{L^{q_c}(B_\rho^+)}^p 
	\chi_{(-\rho^2,0)}(s) ds \\
	&\leq 
	CM^{p-1}  \int_\R|t-s|^{-1+\gamma }
	\| \overline{u}(\cdot,s) \|_{L^{q_c}(B_\rho^+)}
	\chi_{(-\rho^2,0)} ds, 
\end{aligned}
\]
where 
\[
	\gamma:=\frac{1}{2}-\frac{n}{2}\left( \frac{p}{q_c}-\frac{1}{2} \right). 
\]
Note that $0<\gamma<1/2$ by $p>p_S$ and $q_c/p<2$. 
From the Hardy-Littlewood-Sobolev inequality and the H\"older inequality 
with $2/(2\gamma+1)<q_c$, it follows that 
\[
\begin{aligned}
	\|W_2\|_{L^2(Q_{\rho/2}^+)} 
	&=
	\left\| \|W_2(\cdot,t)\|_{L^2(B_\rho^+)} 
	\chi_{(-\rho^2/4,0)} \right\|_{L^2(\R)} \\
	&\leq 
	CM^{p-1} 
	\left\| \| \overline{u}(\cdot,t) \|_{L^{q_c}(B_\rho^+)}
	\chi_{(-\rho^2,0)} \right\|_{L^{\frac{2}{2\gamma+1}}(\R)} \\
	&\leq 
	C\rho^{ \frac{n}{2}-\frac{2}{p-1}-\frac{2}{q_c} } 
	M^{p-1} \|\overline{u} \|_{L^{q_c}(Q_\rho^+)}. 
\end{aligned}
\]

We consider $W_3$ and $W_4$. 
By the definitions of $K_1$ and $K_2$ in \eqref{eq:Kjdef}, 
there exists a constant $C>0$ independent of $\rho$ 
such that 
\[
	\rho^{-2}K_1(x-y,t)
	+\rho^{-1}K_2(x-y,t) \leq C\rho^{-n-3}
\]
for $x\in B_{\rho/2}$, $y\in \overline{B_{4\rho/5}}\setminus B_{3\rho/5}$ 
and $t>0$. 
Hence by the H\"older inequality, we have 
\[
\begin{aligned}
	&\rho^{-2}|W_3(x,t)|+\rho^{-1}|W_4(x,t)|
	\leq 
	C\rho^{-n-3+(n+2)(1-\frac{1}{q_c})}
	\|\overline{u}\|_{L^{q_c}(Q_\rho^+)}, \\
	&
	\rho^{-2}\|W_3\|_{L^2(Q_{\rho/2}^+)}
	+\rho^{-1}\|W_4\|_{L^2(Q_{\rho/2}^+)}
	\leq 
	C\rho^{\frac{n}{2}-\frac{2}{p-1}-\frac{2}{q_c}}
	\|\overline{u}\|_{L^{q_c}(Q_\rho^+)}. 
\end{aligned}
\]
The above estimates show the desired inequality. 
\end{proof}

\begin{remark}
Regularity estimates similar to the above are known 
for semilinear elliptic equations, see \cite{IMN17} 
and the references given there for recent developments.
\end{remark}

%%%%%%%%%%%%%%%%%%%%%%%%%%%%%%%%%%%%%%%%%%%%%%%%
%%%%%%%%%%%%%%%%%%%%%%%%%%%%%%%%%%%%%%%%%%%%%%%%
%%%%%%%%%%%%%%%%%%%%%%%%%%%%%%%%%%%%%%%%%%%%%%%%
\section{Compactness results}\label{sec:cpt}
%%%%%%%%%%%%%%%%%%%%%%%%%%%%%%%%%%%%%%%%%%%%%%%%
%%%%%%%%%%%%%%%%%%%%%%%%%%%%%%%%%%%%%%%%%%%%%%%%
%%%%%%%%%%%%%%%%%%%%%%%%%%%%%%%%%%%%%%%%%%%%%%%%
We recall an Aubin-Lions type compactness result   
from  \cite[Proposition 51.3]{QSbook2} and \cite[Proposition 2.1]{Qu03}. 
See also \cite{Am00} 
and \cite[Sections 2.7, 2.8]{Ambook} 
for more general statement. 
Note that a pair of Banach spaces $(E_0,E_1)$ 
is called an interpolation couple 
if there exists a locally convex space $E$ such that 
$E_0,E_1\hookrightarrow E$. 

\begin{proposition}\label{prop:cpt} 
Let $(E_0,E_1)$ be an interpolation couple. 
Assume that $E_1$ is compactly embedded in $E_0$.
Let $1 \le  p_0, p_1 < \infty$, $0<\theta<1$,
$1/p_{\theta} = (1-\theta)/p_0 + \theta/p_1$ 
and $s < 1-\theta$.
Then, 
\[
	W^{1,p_0} (0, T;E_0) \cap L^{p_1} (0, T;E_1) 
	\hookrightarrow W^{s,p_{\theta}} (0, T; (E_0, E_1)_{\theta, p_{\theta}} )
\]
and this embedding is compact.
\end{proposition}

We give a consequence of this result in a form which is 
used to prove Lemma \ref{lem:exbul}. 

\begin{lemma}\label{lem:WLWcom}
Let $1 < r<\infty$ and let $\cB$ be a smooth bounded domain in $\R^n$.  
Then, 
\[
	W^{1,5}(-1,0;L^r(\cB)) \cap L^5(-1,0;W^{2,r}(\cB))
	\hookrightarrow 
	C([-1,0]; W^{1,r}(\cB))
\]
and this embedding is compact. 
\end{lemma}

\begin{proof}
We write $L^r_x:=L^r(\cB)$, $W^{2,r}_x:=W^{2,r}(\cB)$ and so on. 
Since $W^{2,r}_x\hookrightarrow L^r_x$ is compact,
Proposition \ref{prop:cpt} with 
$p_0=p_1=5$, $\theta=2/3$ and $s=1/4$ gives 
\[
	W^{1,5}(-1,0;L^r_x) \cap L^5(-1,0;W^{2,r}_x)
	\hookrightarrow 
	W^{1/4,5}(-1,0;(L^r_x,W^{2,r}_x)_{2/3,5}). 
\]
By \cite[page 327]{Trbook78}, we have 
$(L^r_x,W^{2,r}_x)_{2/3,5}=B^{4/3}_{r,5,x} 
\hookrightarrow W^{1,r}_x$. 
Then the Sobolev embedding in time yields 
\[
	W^{1/4,5}(-1,0;(L^r_x,W^{2,r}_x)_{2/3,5})
	\hookrightarrow
	C([-1,0]; W^{1,r}_x),
\]
and hence the lemma follows. 
\end{proof}

%%%%%%%%%%%%%%%%%%%%%%%%%%%%%%%%%%%%%%%%%%%%%%%%
%%%%%%%%%%%%%%%%%%%%%%%%%%%%%%%%%%%%%%%%%%%%%%%%
%%%%%%%%%%%%%%%%%%%%%%%%%%%%%%%%%%%%%%%%%%%%%%%%
\section*{Acknowledgments}
%%%%%%%%%%%%%%%%%%%%%%%%%%%%%%%%%%%%%%%%%%%%%%%%
%%%%%%%%%%%%%%%%%%%%%%%%%%%%%%%%%%%%%%%%%%%%%%%%
%%%%%%%%%%%%%%%%%%%%%%%%%%%%%%%%%%%%%%%%%%%%%%%%
The authors would like to express their gratitude to 
Professors Kazuhiro Ishige, Hiroaki Kikuchi, Yukihiro Seki, 
Mitsuo Higaki  and Yifu Zhou 
for their valuable comments 
and to Professor Shi-Zhong Du for sending \cite{CD10} to the authors. 
The authors also gratefully acknowledge the anonymous referees 
for their careful reading and many helpful comments.

The first author was supported in part by JSPS KAKENHI 
Grant Numbers 16H06339, 17K05312, 21H00991 and 21H04433. 
The second author was supported in part 
by JSPS KAKENHI Grant Numbers 19K14567, 22H01131 
22KK0035 and 23K12998.

\end{document}